\documentclass[twoside,11pt]{article}

\usepackage{microtype}
\usepackage{graphicx}
\usepackage{subfigure}
\usepackage{booktabs} 

\usepackage{color}
\newcommand{\congfang}[1]{{#1}}

\newcommand{\rev}[1]{{#1}}
\usepackage{amsmath}
\usepackage{amssymb}
\usepackage{mathtools}
\usepackage{amsthm}
\usepackage{algorithm}
\usepackage{algorithmic}
\usepackage[misc,geometry]{ifsym}
\usepackage{enumitem}
\newcommand{\ud }{\mathrm{d}}

     \def\BB{\mathbb{B}}

     \def\EE{\mathbb{E}}

     \def\NN{\mathbb{N}}

     \def\RR{\mathbb{R}}

 \def\cB{{\cal  B}}

 \def\cE{{\cal  E}}

 \def\cL{{\cal  L}}
 
 \def\cN{{\cal  N}}
 \def\cO{{\cal  O}}
 \def\cP{{\cal  P}}
 \def\cQ{{\cal  Q}}
 \def\cR{{\cal  R}}

 \def\cX{{\cal  X}}
 \def\cY{{\cal  Y}}

\def\dirac{\boldsymbol{\delta}}

\def\rmBiLin{\mathrm{BiLin}}
\def\rmBiLinSf{\mathrm{BiLin}_{\EE_{\xi}f}}
\def\rmErr{\mathrm{Err}}
\def\rmExpectErr{\mathrm{ExpectErr}}
\def\rmStoErr{\mathrm{StoErr}}

\DeclareMathOperator*{\argmin}{argmin}
\DeclareMathOperator*{\argmax}{argmax}
\usepackage{blindtext}

\usepackage{jmlr2e}

\newtheorem{assumption}{Assumption}[section]

\makeatletter
\renewcommand*{\thanks}[1]{
	\footnotemark
	\protected@xdef\@thanks{\@thanks
		\protect\footnotetext[\arabic{footnote}]{#1}}%
}
\makeatother

\usepackage{lastpage}
\jmlrheading{25}{2024}{1-\pageref{LastPage}}{11/23; Revised
7/24}{10/24}{23-1522}{Shihong Ding$^*$, Hanze Dong$^*$, Cong Fang, Zhouchen Lin, Tong Zhang}

\ShortHeadings{Particle-based Primal-Dual Algorithm for MNE}{Ding$^*$, Dong$^*$, Fang, Lin, Zhang}
\firstpageno{1}

\begin{document}

\title{PAPAL: A Provable PArticle-based Primal-Dual ALgorithm for Mixed Nash Equilibrium}

\def\thefootnote{$^*$}
\footnotetext{Equal Contribution. Alphabetical order.}

\def\thefootnote{$^\dag$}
\footnotetext{Work done at the Hong Kong University of Science and Technology.}

\author{\name Shihong Ding$^*$            \email dingshihong@stu.pku.edu.cn\\
       \addr Peking University
       \AND
       \name Hanze Dong$^*$$^\dag$ \email hendrydong@gmail.com\\
       \addr Salesforce AI Research
       \AND
       \name Cong Fang $^\text{\Letter}$  \email fangcong@pku.edu.cn\\
       \addr Peking University
       \AND  
       \name Zhouchen Lin\email zlin@pku.edu.cn \\
       \addr Peking University
       \AND  
       \name Tong Zhang \email tongzhang@tongzhang-ml.org \\
       \addr University of Illinois Urbana-Champaign
       }

\editor{Francesco Orabona}

\maketitle

\begin{abstract} 
\noindent We consider the non-convex non-concave objective function in two-player zero-sum continuous games. The existence of pure Nash equilibrium requires stringent conditions, posing a major challenge for this problem. To circumvent this difficulty, we examine the problem of identifying a mixed Nash equilibrium, 
where strategies are randomized and characterized by probability distributions over continuous domains.
To this end, we propose PArticle-based Primal-dual ALgorithm (PAPAL) 
tailored for a weakly entropy-regularized min-max optimization over probability distributions. 
 This algorithm employs the stochastic movements of particles to represent the updates of random strategies for the $\epsilon$-mixed Nash equilibrium. We offer a comprehensive convergence analysis of the proposed algorithm, demonstrating its effectiveness.
In contrast to prior research that attempted to update particle importance without movements, PAPAL is the first implementable particle-based algorithm accompanied by non-asymptotic quantitative convergence results, running time, and sample complexity guarantees. 
Our framework contributes novel insights into the particle-based algorithms for continuous min-max optimization in the general non-convex non-concave setting.

 \end{abstract}


\newpage
\section{Introduction}

The problem of finding equilibrium of min-max optimization is a fundamental research topic,
which has been a focus of numerous areas, such as mathematics, statistics, machine learning, economics, and computer science \citep{bacsar1998dynamic,roughgarden2010algorithmic,von1947theory,sinha2017certifying}. In machine learning,  the  applications of minimax optimization have shown great potential in real practice recently, such as 
generative adversarial networks \citep{goodfellow2014generative,salimans2016improved}, adversarial training \citep{ganin2016domain,madry2017towards}, reinforcement learning \citep{busoniu2008comprehensive,silver2017mastering,omidshafiei2017deep}. These applications reveal the potential of broader learning systems beyond empirical risk minimization regime \citep{vapnik1991principles}. Unfortunately, the theoretical properties of these objectives are not well studied, since min-max problems in real problems are usually non-convex non-concave. Practitioners often resort to heuristically implementing gradient-based algorithms for convex-concave min-max problems, which are convenient and scalable.
Despite the potential and popularity of the applications with min-max optimization, the theoretical analysis lags far behind: the investigation on how to find a global solution with a provable algorithm in general min-max problems remains preliminary.
The inherent complexities of min-max optimization impede our understanding of certain unanticipated phenomena and problematic scenarios in practical applications, such as instability, divergence, and oscillation \citep{goodfellow2014generative,salimans2016improved,daskalakis2017training,mescheder2018training,daskalakis2018limit}. It is urgent to study these problems and propose provable algorithms.

The min-max problems are usually written as:
\begin{equation}\label{question}
  \min_{x\in\cX} \max_{y\in\cY}  f(x,y),
\end{equation}
where $f:\cX\times\cY\rightarrow\RR$ is the payoff function on the strategy space $\cX\times\cY$.
The solution
to Eq. \eqref{question}, termed the pure Nash equilibrium,  is a strategy pair $(x_*,y_*)$ that both $\min$ and $\max$ side cannot improve the corresponding loss value unilaterally. The Nash equilibrium is a single point pair in the strategy space. However, it exists only  under very restrictive conditions \citep{dasgupta1986existence}. 
For general min-max problems, such as non-convex non-concave case, the pure Nash equilibrium does not usually exist \citep{arora2017generalization}.
Given these challenges, researchers have explored alternative solutions to general min-max problems.
One potential direction is to solve the general min-max problem  in the sense of local equilibrium under varied  assumptions  \citep{evtushenko1974some,daskalakis2018limit,adolphs2019local,fiez2019convergence,mazumdar2020gradient,jin2020local}.  These assumptions are often strong.  For example,
 \citet{fiez2019convergence} propose ``Differential Stackelberg Equilibrium'' which restricts the behavior of the Hessian matrix.
 \citet{diakonikolas2021efficient} assume the minty variational inequality for the local Nash equilibrium, which is analogous to the Polyak-Łojasiewicz inequality in min-max problems.
\citet{mangoubi2021greedy} consider the local case with several regularity conditions that the function $f$ is smooth, bounded with Lipschitz Hessian. 
It was shown by \citet{daskalakis2021complexity} that finding a local min-max Nash equilibrium without these restrictions is also computationally expensive in general. Moreover, a local  equilibrium is more or less limited compared with a global one.

Another criterion is called \emph{mixed Nash equilibrium} (MNE), which proposes a random strategy over the space.
Consider the extension of  Eq.~\eqref{question} as the following problem,
\begin{align}\label{eq:problem_particle}
    \min_{p\in\mathcal{P}_2(\mathcal{X})} \max_{q\in\mathcal{P}_2(\mathcal{Y})} &\EE_{p(x)} \EE_{q(y)} [f(x,y)]  + \cR(p) - \cR(q),
\end{align}
where the optimization variables $p$ and $q$ are densities of probability measures over  $\mathcal{X}$ and $\mathcal{Y}$, respectively, and $\cR(\cdot)$ is a proper regularizer (such as relative entropy) to make the distributions well-behaved. By choosing some prior distribution $\pi$, the entropy regularizer can be written as 
\begin{align}\label{eq:prior}
 \cR(p) = \lambda \EE_p \log p - \lambda \EE_p \log \pi,
\end{align}
where $\lambda>0$ is the parameter to control the strength of the regularization term. The goal of mixed Nash equilibrium (MNE) is to find a probability distribution pair $(p_*,q_*)$ in the strategy space that balances both sides. The regularization term ensures the problem is mathematically well-posed, and its influence becomes negligible as $\lambda \to 0$. 
Compared with the pure Nash equilibrium, the existence of MNE in min-max problems is guaranteed in the infinite-dimensional compact space \citep{glicksberg1952further}.

Finding the MNE with a provable algorithm is highly non-trivial since Eq.~\eqref{eq:problem_particle} is an infinite-dimensional problem. There have been some early attempts to solve the problem, where the core idea is to use a batch of particles (random samples) to parameterize the strategy and optimize the particles to approximate a MNE. 
\citet{balandat2016minimizing} extend dual averaging
to the min-max problem in continuous probability space, but the convergence rate is not discussed.
\citet{hsieh2019finding} propose a mirror-descent algorithm on probability measures to find MNEs. They show that a particle-based algorithm performs well in min-max problems empirically. However, the theoretical results for their algorithm remain open. 
\citet{domingo2020mean} attempt to propose a mean-field algorithm that uses both the weights and positions of the particles to parameterize the mixed strategy.
However, in their analysis, particles do not move nearly as the weights' updates dominate the particles' movements in their setting. Such an algorithm will require a large number of particles, which is not implementable for high-dimensional problems. 

Recent advances give us inspiration to seek the $\epsilon$-mixed Nash equilibrium ($\epsilon$-MNE) of Eq.~\eqref{eq:problem_particle}, \citet{nitanda2021particle} proposed a particle dual average algorithm with convergence guarantee, which optimizes the two-layer network with the second-order moment and negative entropy regularization successfully. Some other works focus on the fast convergence of continuous mean-field GDA dynamics \citep{lu2022twoscale}, the convergence of quasistatic Wassersetin gradient flow \citep{ma2022provably} and the local convergence of particle algorithm \citep{wang2022anexpon}. Therefore, we aim to address a research question:
\begin{quote}
    \emph{Can we design optimization algorithm that provides provable, quantitative guarantees for solving $\epsilon$-MNE in continuous probability space?}
\end{quote}
In this paper, we first prove the existence of $\epsilon$-MNE with relative entropy regularizers. We also propose the PArticle-based Primal-dual ALgorithm (PAPAL), a mirror proximal primal-dual algorithm in the continuous probability space, where the KL divergence is chosen as the corresponding Bregman divergence (wrt the primal-dual algorithm).
Unlike the finite dimensional case, the sub-problems from the proximal operator in continuous probability space
cannot be obtained in closed form. 
As a result, PAPAL is a double-looped algorithm, where the outer loop performs a mirror proximal primal-dual algorithm and the loop is approximated by the particle solution of the sub-problems.
It is important to note that the challenges of particle-based algorithms primarily stem from errors introduced by both the finite-particle parameterization and the inexact solver. 

The early works \citep{balandat2016minimizing,hsieh2019finding}
provide the idea of mirror descent in continuous probability space, but they only consider the continuous dynamics on the probability space. 
More recently, \citet{domingo2020mean} 
construct a dynamics of particles driven by both particle updates and reweighting. However, the underlying dynamics is dominated by weight updates, which means the proposed algorithm nearly does NOT move particles. Optimization within the probabilistic space, from their perspective, employs importance sampling. This approach is recognized for its vulnerability to the curse of dimensionality \citep{donoho2000high}. Thus, it is still not a real ``implementable'' algorithm with theoretical guarantees. In essence, both the inexact sub-problem solution and the functional estimation by finite particles are not discussed in previous work, which capture the behavior of computation complexity and sample size in real problems. 
To fill the gap, 
we show that such errors can be controlled to guarantee global convergence of the primal-dual algorithm when the prior distribution admits $\pi \propto \exp\left(-\lambda'\Vert x\Vert^{2}\right)$ (Gaussian distribution), where $\lambda'>0$ controls the variance of the prior. 

We also highlight the connection between PAPAL and PDA \citep{nitanda2021particle}, in terms of their high-level motivation. 
Initially developed for constructing sequences within probability space and demonstrating convergence properties for mean-field problems with polynomial sample complexity, PDA served as a source of inspiration for our research.
PAPAL surpasses the scope of PDA by extending its applicability to general min-max problems. \rev{From the perspective of convergence rate, without inner-loop error, PAPAL can achieve linear convergence rate. This relies on the momentum structure in the algorithm and sophisticated selection of hyperparameters. However, the convergence rate of PDA is significantly slower than linear rate.} Specifically, PAPAL introduces a novel sequence construction methodology that diverges from the conventional emphasis on mean-field problems.
We aim to extend the concept by devising multiple sub-problems in minimax optimizations. To achieve this, we have made modifications to the iterative formulation of the primal-dual approach to tailor it to particle approximation, while introducing novel analytical techniques beyond the realm of optimality conditions. To the best of our knowledge, we present the first particle-based algorithm that is not only practically implementable but also capable of approximating an MNE in continuous min-max optimization problems. Moreover, we provide a rigorous global convergence guarantee for our proposed approach.

\paragraph{Contributions.} 
    We propose PAPAL to solve the general $\epsilon$-mixed Nash equilibrium ($\epsilon$-MNE) 
    in two-player zero-sum continuous games with minimal regularization. We validate the existence of MNE within our framework. 
    Our analysis focuses on the convergence of the primal-dual method in addressing MNE. Specifically, we demonstrate that the sampling sub-problems exhibit log-Sobolev properties, and we can ensure global convergence even when utilizing inexact solvers for sub-problems.
    We present quantitative convergence metrics, runtime estimations, and sample complexity for PAPAL. 
    Additionally, we elucidate the analysis of the stochastic gradient variant of the PAPAL algorithm.
    To our knowledge, this is the first demonstration of the computational and sample feasibility of a primal-dual method in a continuous probability framework to resolve min-max challenges.

\section{Related Work}

\paragraph{Min-Max Optimization.} The study of solving min-max problems is quite extensive since its inception with \citep{von1928theory}. 
The gradient-based algorithms, including those presented by \citet{korpelevich1976extragradient,nemirovski1978cezari,chen1997convergence,nemirovski2004prox}, leverages convexity and concavity to obtain the convergence properties.
Notably,
gradient descent ascent (GDA) finds an $\epsilon$-approximate stationary point within $\cO(\log 1/\epsilon)$ iterations for strongly-convex strongly-concave minimax problems, and $\cO(\epsilon^{-2})$ iterations for convex concave games \citep{nemirovski2004prox,nedic2009subgradient}.  
While some alternative gradient methodologies extend to non-convex concave scenarios \citep{namkoong2016stochastic,grnarova2017online,rafique2019non,lu2020hybrid,lin2020gradient},  the general min-max domain remains relatively uncharted. A central challenge emerges from the dependency of general min-max problems on stringent conditions, coupled with a lack of consensus on assumptions. By worst-case evaluations, non-convex non-concave min-max problems can be computationally formidable. 
\citet{daskalakis2009complexity} highlighted the PPAD-completeness of establishing a general pure Nash equilibrium without the condition of convexity/concavity. 
Similarly, \citet{chen2006computing} prove the PPAD-completeness for non-zero two-player mixed Nash equilibria in bimatrix game, and subsequent studies 
\citep{daskalakis2013complexity,rubinstein2016settling} underlined this complexity even for approximated mixed Nash equilibria.  Finding an $(\epsilon,\delta)$-local min-max Nash equilibrim in non-convex non-concave min-max problems is also PPAD-complete \citep{daskalakis2021complexity}.
While recent endeavors \citep{daskalakis2018limit, adolphs2019local, mazumdar2020gradient, jin2020local, mangoubi2021greedy} have delved into the properties of local min-max points in non-convex frameworks, real-world interpretations of the imposed conditions remain elusive. Consequently, general pure Nash Equilibrium's feasibility is widely regarded as remote, barring justifiable restrictive conditions.

\paragraph{Mixed Nash Equilibrium.} Mixed Nash Equilibrium lifts the strategy space to a probability space and convert min-max problem to be infinite-dimensional.
Importantly, the existence of Mixed Nash Equilibrium can be guaranteed when the payoff function is differentiable with Lipschitz continuous gradient  \citep{glicksberg1952further}. Recently, with the growing importance of general min-max problems in deep learning and Generative Adversarial Networks (GAN), the investigation of Mixed Nash Equilibrium  has gained more attention due to its good theoretical properties. The global existence without convexity assumption make it more reasonable in real problems.
\citet{grnarova2017online} study the 
 behavior in a special setting: they consider the online setting of GANs with a two-layer  neural network as the discriminator. \citet{arora2017generalization} study the generalization properties of the Mixed Nash Equilibrium in GAN.  Moreover, there is a line of works that try to solve the Mixed Nash Equilibrium with mirror-descent-like algorithms in continuous probability space.
\citet{balandat2016minimizing} first solve Mixed Nash Equilibrium  with dual averaging algorithm, but the convergence rate is not discussed. 
Earlier discussions between \citep{hsieh2019finding} and \citep{domingo2020mean} render repetition unnecessary. In a word, there are some early attempts to perform mirror proximal algorithms in probability space.  In essence, while preliminary endeavors into mirror proximal algorithms in the probability space have emerged, with empirical evidence underscoring their significance, a rigorous theoretical dissection of "implementable" particle-based algorithms is still in its nascence.

\paragraph{Mean-field Analysis.} 
Mean field analysis focuses on the evolution of particles on infinite-dimensional space \citep{dobrushin1979vlasov}, which can be described by some non-linear partial differential equations (PDE). 
Notably, many non-convex dynamical systems in modern machine learning can be analyzed by lifting the original problem to infinite-dimensional space. 
Such a procedure convexifies the original problem, which has better mathematical properties \citep{bengio2005convex,bach2017breaking}.
The limiting PDE in mean-field analysis can derive the global convergence and well-posed behaviors. For example, the mean-field analysis over the neurons in wide neural networks captures the feature evolution of the optimization process of deep learning \citep{mei2018mean,mei2019mean,yang2020feature,nguyen2020rigorous,fang2021modeling,yang2022efficient,fang2022convex}.  
There have been a number of works based on noisy gradient descent to investigate the quantitative convergence rate when minimizing the nonlinear functionals with entropy regularizers \citep{hu2021mean,nitanda2021particle,chizat2022mean}.  In summary, mean-field analysis stands as a powerful theoretical tool in today's machine learning theory, which helps to convexify high-dimensional problems through randomization (a special lifting to infinite dimensional space).

\section{Preliminary}\label{sec:preli}

\paragraph{Notation.} We use lower-case letters $x, y$ to denote
vectors. For function $f:\RR^d\rightarrow\RR$, 
$\nabla f(\cdot)$ and $\nabla^2 f(\cdot)$ denote its gradient and  Hessian matrix, respectively. 
For function $f:\RR^d\rightarrow\RR^d$, $\nabla f(\cdot)$ denotes the
Jacobian matrix.
For multivariate function $f(x,y)$, $\nabla_x$, $\nabla_y$ denote the gradient over $x$, $y$.
Notation $\Vert\cdot\Vert$ denotes 2-norm for both vector and matrix. 
The continuous probability measures in this paper are considered to be absolutely continuous by default with respect to the Lebesgue measure, which induces density function $p$ and $q$. The smoothness is guaranteed by Gaussian initialization and Gaussian smoothing during the updates.
For twice differentiable function $f:\RR^d\rightarrow\RR$, we say that $f$ is $L_0$-Lipschitz if  $|f(x)-f(y)|\leq L_0 \Vert x-y\Vert$ for any $x,y\in\RR^d$; we say that $f$ has $L_1$-Lipschitz gradient if $\|\nabla f(x)-\nabla f(y)\|\leq L_1 \Vert x-y\Vert$ for any $x,y\in\RR^d$. For any a topological space  
$\mathcal{X}$, $\mathcal{P}(\mathcal{X})$ denotes set of probability measures on $\cX$ and 
$\mathcal{P}_2(\mathcal{X})$ denotes the set of probability measures with bounded second moments.
For density $p$, we define $\cE(p)=\mathbb{E}_p[\log(p)]$, $\mathrm{KL}(p\|q)=\mathbb{E}_p[\log (p/q)]$. For the generalized density for particles, we use Dirac delta distribution $\dirac_{x}$ to represent the Dirac delta distribution at $x$, $\dirac_{\{x_1,\cdots,x_n\}} = n^{-1}\sum_{i=1}^n\dirac_{x_i}$.

\paragraph{Problem Setup.}
Assume that the min-max problem is defined on the joint parameter space $\cX \times \cY$ and a payoff function $f: \cX \times \cY\rightarrow \RR$. 
We consider the bounded function case. For simplicity without loss of generality, the
domain $\cX \times \cY$ is the 
Euclidean space $\RR^m\times\RR^n$ and the co-domain of $f$  is a bounded interval $[-1,1]$.

In our work, we consider the regularized minimax problem defined in Eq.~\eqref{eq:problem_particle}. For the regularization term,  
we set 
\begin{equation}\label{eq:relative_reg}
\cR(p(x))=\EE_{p(x)}\left[\lambda_1\|x\|^{2}+\lambda_2\log(p(x))\right] 
\end{equation}
where $\lambda_1,\lambda_2>0$ are constants to determine the regularization scaling as well as the variance of the prior.
We use $\cL(p,q)$ to denote $\EE_{p(x)} \EE_{q(y)} f(x,y)  + \cR(p) - \cR(q)$.
The mixed Nash equilibrium denotes a distribution pair $(p_*,q_*)$ such that for all $(p,q)\in \cP(\cX )\times \cP(\cY)$, $\cL(p_*,q) \leq \cL(p_*,q_*)\leq \cL(p,q_*)$. Therefore, for any distribution pair $(\bar{p},\bar{q})$, the duality gap $\max\limits_{q\in\cP(\cY)}\mathcal{L}(\bar{p},q)-\min\limits_{p\in\cP(\cX)}\mathcal{L}(p,\bar{q})$ forms the basis for a standard optimality criterion. Formally, we define $\epsilon$-MNE as below.
\begin{definition}[$\epsilon$-MNE]
    A distribution pair $(\bar{p},\bar{q})$ is an $\epsilon$-mixed Nash equilibrium ($\epsilon$-MNE) of function $\mathcal{L}(\cdot,\cdot)$ if $$\max\limits_{q\in\cP(\cY)}\mathcal{L}(\bar{p},q)-\min\limits_{p\in\cP(\cX)}\mathcal{L}(p,\bar{q})\leq \epsilon.$$ When $\epsilon=0$, $(\bar{p},\bar{q})$ is a mixed Nash equilibrium.
\end{definition}

\paragraph{Assumptions.}
\begin{enumerate}[label={\textbf{[A$_\text{\arabic*}$]}}]
    \item There exists $L_0$ such that $f$ is $L_0$-Lipschitz.\label{ass:a1}
    \item $f$ has $L_1$-Lipschitz gradient.\label{ass:a2}
 \item For any $x,y\in \cX\times\cY$, $| f(x,y)|\leq 1$.\label{ass:a4}
\end{enumerate}

We consider the bounded function with growth regularities. The bound of $|f|$ is chosen as 1 for notation simplicity.
Assumptions \ref{ass:a1} and \ref{ass:a2} regularize the growth of $f$, $\nabla f$ respectively and the definition of  Lipschitz continuity was introduced in \textbf{Notation} part. It is clear that our assumption is quite weak so  our setting is general.

\paragraph{Primal-Dual Algorithms for Min-Max Optimization.}
Primal-dual method is a classical approach to solve saddle point problems \citep{nemirovski2004prox,he2017algorithmic,lan2020first,bredies2015preconditioned,luke2018globally}. Consider the min-max problem
\begin{align}\label{saddle-point-problem}
    \min_{x\in \cX}\max_{y\in \cY}\left\{f(x)+\langle Ax,y\rangle-g(y)\right\},
\end{align}
where $A\in\mathbb{R}^{m\times n}$ denotes a linear mapping, and $f:\cX\rightarrow \mathbb{R}$ and $g:\cY\rightarrow \mathbb{R}$ are convex functions. General optimization problems with linear constraints: $\min\limits_{x\in\cX}\{f(x):Ax=b \}$ \congfang{can be transformed to}
Eq.~\eqref{saddle-point-problem}. The proximal primal-dual algorithm updates $x_t$ and $y_t$ by
\begin{align}
    \Tilde{x}_t=&x_{t-1}+\mu_t(x_{t-1}-x_{t-2}),
    \\
    y_t=&\arg\min\limits_{y\in\cY}\langle-A\Tilde{x}_t,y\rangle+g(y)+\tau_t W(y_{t-1},y),
    \label{eq:minmax_y}\\
    x_t=&\arg\min\limits_{x\in\cX}\langle y_t,Ax\rangle+f(x)+\eta_t V(x_{t-1}, x),\label{eq:minmax_x}
\end{align}
under a finite dimension condition with $W$ and $V$ representing Bregman's distances on $\cY$ and $\cX$, respectively. Moreover, $\tau_t, \eta_t, \gamma_t$ are algorithmic parameters to be specified.  The
Bregman divergence is particularly important for other Riemannian structures.
In probability space, the Bregman divergence is usually chosen as KL divergence or Wasserstein distance.

\paragraph{Unadjusted Langevin Algorithm.} We introduce the 
\emph{Unadjusted Langevin Algorithm (ULA)}
to solve the mirror steps in our setting. In particular, some isoperimetry inequalities shall be introduced to make each step well-posed.

\begin{definition} \citep{gross1975logarithmic}
We say that density $p$ satisfies log-Sobolev inequality (LSI) with a constant $\alpha>0$, if
\begin{align}\label{LSI}
    \mathbb{E}_{p}[g^2\log(g^2)]-\mathbb{E}_{p}[g^2]\log\mathbb{E}_{p}[g^2]\leq\frac{2}{\alpha}\mathbb{E}_{p}[\|\nabla g\|^2]
\end{align}
for all differentiable function $g$ with 
 $\mathbb{E}_{p}[g^2]<\infty$. When $g = \sqrt{q/p}$, we have 
$$
\mathrm{KL}(q\Vert p) \leq \frac{1}{2\alpha} \EE_q \left\Vert \nabla \log \frac{q}{p}\right\Vert^2.
$$
\end{definition}

Assume we aim to sample from a smooth target probability density $p\propto e^{-f(x)}$ where $f(x)$ has $L_1$-Lipschitz gradient and satisfies a smoothness condition. The ULA with step size $\epsilon>0$ is the discrete-time algorithm
\begin{align}
    x_{s+1}=x_s-\epsilon\nabla f(x_s)+\sqrt{2\epsilon}\zeta_s,
\end{align}
where $s=0,\cdots,T$ for some positive integer $T$; $x_0$ is given as an initialization;  
$\zeta_s\sim\mathcal{N}(0,I)$ is an isotropic Gaussian noise. Moreover, if $p$ satisfies LSI, we can guarantee the convergence rate of ULA.
The $\text{KL}(\rho_s\|p)$ along ULA would also converge to $\cO(\epsilon)$ with $\cO(\epsilon)$ fixed step size. 
Since it is equivalent with gradient descent with additional Gaussian noise, ULA is easy to implement in practice. 
The easy-to-implement property and the arbitrarily small bias make it a popular in sampling literature. Thus, we choose it to solve our sub-problems.

\begin{algorithm}[h]
\caption{PAPAL: PArticle-based Primal-dual ALgorithm for Mixed Nash Equilibrium} \label{alg:algorithm}
{\small
{\bf Input:} 
 number of outer iterations $T$, outer-loop learning rates $\{\eta_t\}_{t=1}^T$ and $\{\tau_t\}_{t=1}^T$,  outer-loop hyper-parameters $\{\mu_t\}_{t=1}^T$; number of inner iterations $\{T_t\}_{t=1}^T$, inner-loop learning rates $\{\iota_t\}_{t=1}^T$, 
 number of particles $M$
 \\
{\bf Output:} estimated equilibrium   
$(\frac{1}{M}\dirac_{\hat{X}},\frac{1}{M}\dirac_{\hat{Y}})$
\begin{algorithmic}[1]
\STATE Initialize $\{\Tilde{x}_r^{(0)}\}_{r=0}^M\sim p_0(x)\propto\exp\left\{-\frac{\lambda_1}{\lambda_2}\|x\|^2\right\},\ \{\Tilde{y}_r^{(0)}\}_{r=0}^M\sim q_0(y)\propto\exp\left\{-\frac{\lambda_1}{\lambda_2}\|y\|^2\right\}.$ 

Assign $(\tilde{X}^{(-1)},\tilde{Y}^{(-1)})=(\tilde{X}^{(0)},\tilde{Y}^{(0)})\leftarrow\left(\{\Tilde{x}_r^{(0)}\}_{r=1}^M, \{\Tilde{y}_r^{(0)}\}_{r=1}^M\right)$.

Assign $h^{(0)}(y) = \frac{\lambda_1}{\lambda_2}\|y\|^2 $, $g^{(0)}(x) = \frac{\lambda_1}{\lambda_2}\|x\|^2 $.
\FOR{$t=0$ to $T-1$}
    \STATE Initialize $(X^{(0)},Y^{(0)})\sim p_0(x)\times q_0(y)$.
    \STATE Update $
    h^{(t+1)}(y) 
    =\frac{1} {\lambda_2+\tau_{t+1}}\left\{\hat{\phi}_{t+1}(y)
+\lambda_1\|y\|^{2}+\tau_{t+1}h^{(t)}(y)\right\}.
    $
    \FOR{$k=0$ to $T_t$, $r=0$ to $M$}
        \STATE Run noisy GD
        $y_r^{(k+1)}\leftarrow y_r^{(k)}-\iota_t\nabla_y h^{(t+1)}(y_r^{(k)})+\sqrt{2\iota_t}\xi_r^{(k)}$, where $\xi_r^{(k)}\sim\mathcal{N}(0, I_{n})$.
    \ENDFOR
    \STATE Assign $\Tilde{Y}^{(t+1)}\leftarrow Y^{(T_t+1)}$.
    \STATE Update $
    g^{(t+1)}(x)
    =\frac{1} {\lambda_2+\eta_{t+1}}\left\{\hat{\psi}_{t+1}(x)+\lambda_1\|x\|^{2}+\eta_{t+1} g^{(t)}(x)\right\}.
    $
    
    \FOR{$k=0$ to $T_t$, $r=0$ to $M$}
        \STATE Run noisy GD $x_r^{(k+1)}\leftarrow x_r^{(k)}-\iota_t\nabla_x g^{(t+1)}(x_r^{(k)})+\sqrt{2\iota_t}\zeta_r^{(k)}$, where $\zeta_r^{(k)}\sim\mathcal{N}(0, I_{m})$.
    \ENDFOR
    \STATE Assign $\Tilde{X}^{(t+1)}\leftarrow X^{(T_t+1)}$.
    
\ENDFOR 
\STATE\textbf{Option I:} $
(\frac{1}{M}\dirac_{\hat{X}},\frac{1}{M}\dirac_{\hat{Y}}) \leftarrow 
(\frac{1}{M}\dirac_{\Tilde{X}^{(T)}},\frac{1}{M}\dirac_{\Tilde{Y}^{(T)}})$.
\STATE \textbf{Option II:} Randomly pick up $(\frac{1}{M}\dirac_{\tilde{X}^{(t)}},\frac{1}{M}\dirac_{\tilde{Y}^{(t)}})$ from $t\in\{1,\cdots,T\}$ following the probability $\mathbb{P}[t]=\frac{\mu_t}{\sum_{t=1}^T\mu_t}$ as $(\frac{1}{M}\dirac_{\hat{X}},\frac{1}{M}\dirac_{\hat{Y}})$.
\end{algorithmic}
}
\end{algorithm}

\section{PAPAL: PArticle-based Primal-dual ALgorithm}
We propose PArticle-based Primal-dual ALgorithm (PAPAL) to solve the regularized minimax problem Eq.~\eqref{eq:problem_particle} by applying Mirror-Prox Primal-dual method to MNE.

As an optimization algorithm on the space of  continuous probability measures, our proposed PAPAL (Algorithm \ref{alg:algorithm}) consists of two components: inner loop and outer loop. Specifically,
our outer loop extends the classical primal-dual scheme (Sec. \ref{sec:preli}) to infinite dimensional optimization problems (Sec. \ref{sec:pd}). Each outer step can be rewritten as a sampling task for Gibbs distribution (Sec. \ref{sec:energy}). We perform ULA as the inner loop to approximate a proximal Gibbs distribution, which is applied to the outer loop optimization to converge to the Mixed Nash Equilibrium distributions $p^*$ and $q^*$ (Sec. \ref{sec:solver}). We will provide the motivation of our algorithm from an optimization view (Sec. \ref{sec:opt_view}).

\subsection{Exact Primal-dual Method}\label{sec:pd}
We construct the target density of $p^*_t$ and $q^*_t$ for each iteration. Particularly, given $\hat{p}_{-1}$, $\hat{p}_0$, and $\hat{q}_0$,
$q_t^*(y)$ and $p_t^*(x)$ satisfy the following recursive condition,
    \begin{align}              
q_t^*:=&\argmin_{q\in\cP_2(\cY)}\mathbb{E}_q[\hat{\phi}_t(y)]+\cR(q) +
        \tau_t \mathrm{KL}(q\|q_{t-1}^*), \label{minimizer-q}
        \\ p_t^*:=&\argmin_{p\in\cP_2(\cX)}\mathbb{E}_p[\hat{\psi}_t(x)]+\cR(p)
        +\eta_t \mathrm{KL}(p\|p_{t-1}^*), \label{minimizer-p}
    \end{align}
    where $\hat{\phi}_t(y)=-(1+\mu_t)\mathbb{E}_{\hat{p}_{t-1}}[f(\cdot,y)]+\mu_t\mathbb{E}_{\hat{p}_{t-2}}[f(\cdot,y)]$, and $\hat{\psi}_t(x)=\mathbb{E}_{\hat{q}_t}[f(x,\cdot)]$. 

    If $\hat{q}_t = q_t^*$, $\hat{p}_t = p_t^*$, Eq.~\eqref{minimizer-q} and \eqref{minimizer-p} perform exact primal-dual method in continuous probability space, with KL divergence as the Bregman divergence. It is quite direct to compare them to Eq.~\eqref{eq:minmax_y} and \eqref{eq:minmax_x}.
    However,
    there two gaps for the aforementioned updates: 
    (1) the forms of $q_{t-1}^*$ and $p_{t-1}^*$ are not explicit in Eq.~\eqref{minimizer-q} and Eq.~\eqref{minimizer-p};
    (2) the exact solutions for $\hat{q}_t\equiv q_t^*$ and $\hat{p}_t\equiv p_t^*$ are not tractable.
    Thus, it is necessary to derive a explicit form of $(q_{t-1}^*,p_{t-1}^*)$ and approximate the optimal $(q_t^*,p_t^*)$ with finite $M$ particles in $(\hat{q}_t, \hat{p}_t)$.



\subsection{Recursive Energy Function}\label{sec:energy}

Section \ref{sec:pd} provides us an algorithm to find mixed Nash equilibrium in the probability space.
In this part, we would discuss the feasibility of the proposed updates.

The following Lemma provides us insight to obtain the explicit form of $q_*^{t-1}$ and $p_*^{t-1}$ in  Eq.~\eqref{minimizer-q} and Eq.~\eqref{minimizer-p} recursively, which can be implemented in an algorithm.

\begin{lemma}\label{lemma:optimal-KL-prox}
    Let $\lambda>0$ be a positive real number and $\tilde{l}(\theta), l(\theta)$ be bounded continuous functions. Consider a probability density $\Bar{p}(\theta)\propto\exp\left(-\tilde{l}(\theta)\right)$, then $p(\theta)\propto\exp\left(-\frac{1}{\lambda+\tau}l(\theta)-\frac{\tau}{\lambda+\tau}\tilde{l}(\theta)\right)$ is an optimal solution to the following problem  
    \begin{align}
        \min\limits_{p\in\mathcal{P}_2(\mathbb{R}^{d_{\theta}})}\{\mathbb{E}_{p}[l(\theta)]+\lambda \cE(p)+\tau \mathrm{KL}(p\|\Bar{p})\}.
    \end{align}
\end{lemma}

Once we obtained two approximate distributions
$\hat{p}_i(y)$ and $\hat{p}_i(x)$, Lemma \ref{lemma:optimal-KL-prox} indicates that if $p_{t-1}^*$ and $q_{t-1}^*$ are proximal Gibbs distribution
with explicit form. 
Then, $p_t^*$ and $q_t^*$ have the similar expression as well. Therefore, we denote 
    \begin{align}
    h^{(t)}(y) 
    =\frac{1} {\lambda_2+\tau_t}\left\{\hat{\phi}_t(y)
+\lambda_1\|y\|^{2}+\tau_th^{(t-1)}(y)\right\},\label{h_def}
    \\
    g^{(t)}(x)
    =\frac{1} {\lambda_2+\eta_t}\left\{\hat{\psi}_t(x)+\lambda_1\|x\|^{2}+\eta_t g^{(t-1)}(x)\right\},\label{g_def}
    \end{align}
    which consist of the Gibbs potential on $\mathcal{P}_2(\mathcal{Y})$ and $\mathcal{P}_2(\mathcal{X})$ after $t-1$ iterations. Thus, at each step-$t$, our objective probability densities can be written as 
    \begin{align}
        q_t^*(y)=\frac{\exp\left(-h^{(t)}(y)\right)}{Z_{t,1}},\, \, \, \, p_t^*(x)=\frac{\exp\left(-g^{(t)}(x)\right)}{Z_{t,2}},\notag
    \end{align}
    where $Z_{t,1}$ and $Z_{t,2}$ are normalizing constants. 

\subsection{Inexact Solver and Particle Approximation}\label{sec:solver}

At every step \( t \), achieving the ideal scenario where \( \hat{q}_t = q_t^* \) and \( \hat{p}_t = p_t^* \) is not tractable within the algorithm. This limitation arises because obtaining exact solutions for Equations \eqref{minimizer-q} and \eqref{minimizer-p} is not feasible for a general function \( f \).
Fortunately, the optimization involving random particles can be effectively approximated using the log density. ULA presents a viable approach for approximating the desired target distribution. By performing ULA,
we can obtains random particle samples $(\{x_r^{T_t}\}_{r=1}^M,\{y_r^{T_t}\}_{r=1}^M)$ that can approximate  $(q^*_t,p^*_t)$. We let that $(p_t,q_t) = (\mathrm{Law}(\{x_r^{T_t}\}_{r=1}^M),\mathrm{Law}(\{y_r^{T_t}\}_{r=1}^M))$ and $(\hat{q}_t, \hat{p}_t)$ is finite sample approximation of $(p_t,q_t)$ (delta distribution by finite samples).


We first initialize a set of particles and duplicate it by setting $\{\tilde{x}_r^{(-1)},\tilde{y}_r^{(-1)}\}_{r=1}^M =\{\tilde{x}_r^{(0)},\tilde{y}_r^{(0)}\}_{r=1}^M
$. 
For each step, we run ULA to optimize the particles in $\cY$, where the gradient at the $k$-th inner step is given by $\nabla_y h^{(t)}(y_r^{(k)})$ estimated by $\{\tilde{x}_r^{(t-1)}\}_{r=1}^M$ and $\{\tilde{x}_r^{(t-2)}\}_{r=1}^M$. Then, we execute the same process as above starting from the prior where the gradient at the $k$-th inner step is given by $\nabla_y g^{(t)}(x_r^{(k)})$ estimated by $\{\tilde{y}_r^{(t)}\}_{r=1}^M$. 
Thus, we obtain the full procedure of PAPAL in Algorithm \ref{alg:algorithm}. We also provide empirical results in Appendix \ref{sec:empi}.


\begin{remark} To make a clear
distinction between  $(\hat{q}_t,\hat{p}_t)$, $(q_t,p_t)$, and $(q^*_t,p^*_t)$, we reiterate the definition here: we initialize $(\hat{q}_{-1},\hat{p}_{-1})=(\hat{q}_0,\hat{p}_0)$ from the prior distribution. From $t=1$ to $t=T$, $(q^*_t,p^*_t)$ is the exact solution from Eq.~\eqref{minimizer-q} and Eq.~\eqref{minimizer-p}; $(q_t,p_t)$ is the ULA approximation of $(q^*_t,p^*_t)$;  $(\hat{q}_t,\hat{p}_t)$ is the realization of $(q_t,p_t)$ with $M$ samples (mixture of dirac delta distributions). All the densities can be defined by recursion.
\end{remark}
\begin{remark}
 Other sampling algorithms can be implemented as the inexact solver, such as Hamiltonian Monte Carlo \citep{duane1987hybrid,neal2011mcmc}, Metropolis-adjusted Langevin algorithm  \citep{rossky1978brownian,grenander1994representations}.
We use ULA here for demonstration and simplicity. More sophisticated solvers may accelerate the convergence.
\end{remark}

\subsection{Optimization View of PAPAL}\label{sec:opt_view}

In this subsection we explain the intuition behind PAPAL from an optimization perspective. Recall that Algorithm \ref{alg:algorithm} iteratively updates the particles $\hat{q}_t$ and $\hat{p}_t$ which the approximate the minimizers of the linearized potentials Eq.~\eqref{minimizer-q} and \eqref{minimizer-p}. 
Actually, we obtain the classical Primal-dual algorithm  \citep{nemirovski2004prox} if Eq.~\eqref{minimizer-q} and Eq.~\eqref{minimizer-p} degenerate to finite dimension. It is obvious that we can find $q_t^*$ and $p_t^*$ with negligible error under finite dimension condition which implies the total complexity would be derived easily. 
However, there are three problems if we rely on Langevin algorithm when solving the sub-problem for infinite dimension: 
\begin{enumerate}
    \item Express the objective densities of the sub-problem as the form of $e^{-h^{(t)}(y)}/Z_{t,1}$ and $e^{-g^{(t)}(x)}/Z_{t,2}$, which can be explicitly obtained in an algorithm;
    \item Analysis the outer loop
    with the bounded KL error induced by ULA at each step-$t$, which converges linearly with fixed step size under LSI;
    \item Control the error induced by finite particles at each step-$t$, which should be done by concentration.
\end{enumerate}
The first problem can be solved by forward recursion of $(q_{t-1}^*,p_{t-1}^*)$ and Lemma \ref{lemma:optimal-KL-prox} combining the approximated solutions $\left\{(\hat{q}_k, \hat{p}_k)\right\}_{k=-1}^{t-1}$ we have known.
Next, in order to bound the error between approximate densities and precise densities in each iteration of Primal-dual method, we introduce the LSI condition \citep{gross1975logarithmic}, which can guarantee the fast convergence of sampling algorithms, such as ULA.
Therefore, we add $\|\cdot\|^{2}$ term with coefficient $\lambda_1$, ($\lambda_1$ can be an arbitrary positive number) to the regularizer which guarantee that $q_t^*$ and $p_t^*$ satisfy the LSI condition. 
Finally, the finite particle $\hat{\phi}_t$ and $\hat{\psi}_t$ approximate the true expectation 
${\phi}_t(y)=-(1+\mu_t)\mathbb{E}_{{p}_{t-1}}[f(\cdot,y)]+\mu_t\mathbb{E}_{{p}_{t-2}}[f(\cdot,y)]$, and
$ {\psi}_t(x)=\mathbb{E}_{{q}_t}[f(x,\cdot)]$. The gap between $(\hat{\phi}_t,\hat{\psi}_t)$ and $({\phi}_t,{\psi}_t)$ shall be controlled by concentration techniques.

To sum up, we connect bilinear forms $\langle Ax,y\rangle$ under finite dimension and $\EE_{p(x)}\EE_{q(y)}f(x,y)$ under infinite dimension with proximal Gibbs distribution. In addition, we apply the estimation of sub-problem error and finite sample error 
to global convergence analysis for our particle-based algorithm. Therefore, we may interpret the PAPAL as a finite-particle approximation of a primal-dual method. PAPAL combines the efficient convergence of Primal-dual method in outer loop and the KL error bound of ULA under LSI condition in inner loop. More importantly, the loss functional can be estimated with polynomial finite particles, which makes our algorithm implementable.

\section{Existence of MNE}
Our objectives are two-fold: 
1) To achieve an \( \epsilon \)-approximate Mixed Nash Equilibrium (MNE) using the PAPAL algorithm.
2) To minimize the distance between the continuous probability density function generated by the algorithm and the MNE of the problem defined in Eq.~\eqref{eq:problem_particle}, with the regularizer outlined in Eq.~\eqref{eq:relative_reg}, aiming to reduce this gap as close to zero as possible. A fundamental step towards these goals is establishing the existence of the MNE for the problem in Eq.~\eqref{eq:problem_particle}. To this end, the subsequent theorem asserts the existence of an MNE for this problem, considering the regularization term in Eq.~\eqref{eq:relative_reg}. \rev{It is worth noting that the existence result in this paper is established in $\mathbb{R}^m\times\mathbb{R}^n$
, while the existence result proposed by \citet{domingo2020mean} is based on compact Riemann manifolds without
boundary embeded in $\mathbb{R}^m\times\mathbb{R}^n$.}
\begin{theorem}\label{existence-nash}
    Considering problem Eq.~\eqref{eq:problem_particle} with regularizer Eq.~\eqref{eq:relative_reg} and assuming Assumptions \ref{ass:a1}, \ref{ass:a2} and \ref{ass:a4} hold, there exists a MNE $(p_*,q_*)$ satisfies that 
    \begin{equation}\label{explicit-expression}
        \begin{split}
            p_*=&\arg\min\limits_{p\in\mathcal{P}_2(\cX)}\mathbb{E}_{p(x)}\mathbb{E}_{q_*(y)}[f(x,y)]+\mathcal{R}(p)-\mathcal{R}(q_*),
	\\	        
        q_*=&\arg\min\limits_{q\in\mathcal{P}_2(\cY)}\mathbb{E}_{p_*(x)}\mathbb{E}_{q(y)}[f(x,y)]+\mathcal{R}(p_*)-\mathcal{R}(q).
        \end{split}
    \end{equation}
\end{theorem}
The proof of Theorem \ref{existence-nash} can be formulated into a two-step process. Step 1: constructing a suitable probability density function family and leveraging well-established lemmas in functional analysis to verify that this function family is a compact convex set. Step 2: proving that the operator $\boldsymbol{F}(p,q)$ (as precisely defined in the Appendix \ref{existence-MNE}) is a continuous mapping from the aforementioned function family to itself. Finally, we can achieve the conclusion of Theorem \ref{existence-nash} combining Brouwer fixed-point theorem.

\section{Convergence Analysis}
In this section,  we provide quantitative global convergence for PAPAL. 
We first provide the finite sample error bound in Sec.~\ref{sec:finite}. Then, under deterministic version, we introduce the outer loop settings in Sec.~\ref{subsection-5.1} and conclude the properties of the updates. By assuming the approximate optimality of the inner loop iterations, the error bound of the outer loop can be obtained. In Sec.~\ref{subsection-5.2}, we discuss the convergence rate for the inner loop iterations, which is the KL convergence of ULA algorithm. In Sec.~\ref{subsection-5.3}, we combine the outer loop and the inner loop to discuss the total complexity of deterministic PAPAL and provide the global convergence results of our main theorem. As a supplementary remark, Sec \ref{Sto-PAPAL} presents a global convergence analysis of the stochastic PAPAL algorithm.

\subsection{Finite Sample Error}\label{sec:finite}
We introduce two probability sets $\cQ_{\cX}(C)$ and $\cQ_{\cY}(C)$ for positive constant $C$ as follows
 \begin{align*}
     &\cQ_{\cX}(C):=\left\{p(x)\in\cP_2(\cX)\left|p(x)\propto e^{-l_1(x)-\frac{\lambda_1}{\lambda_2}\|x\|^2}; \|l_1\|_\infty\leq \frac{C}{\lambda_2}\right.\right\},\\
     &\cQ_{\cY}(C):=\left\{q(y)\in\cP_2(\cY)\left|q(y)\propto e^{-l_2(y)-\frac{\lambda_1}{\lambda_2}\|y\|^2};\|l_2\|_\infty\leq \frac{C}{\lambda_2}\right.\right\}.
 \end{align*}
 To establish connections between $\phi_t$ and $\hat{\phi}_t$, as well as $\psi_t$ and $\hat{\psi}_t$, in the convergence analysis, it is sufficient to bound the discretization error on linear term by using the form of objective function in Eq.~\eqref{eq:problem_particle}. Therefore, we use the following lemma to estimate the discretization error between $\hat{p}_t$ and $p_t$ ($\hat{q}_t$ and $q_t$) at each step-$t$ (refer to Appendix \ref{discre}).
 \begin{lemma}\label{finite-particle-approximation}
For any $\epsilon>0, \delta\in(0,1)$,  the errors of finite particles satisfy
     \begin{align}
         &\mathbb{P}\left[\left|(\hat{p}_t-p_t)(x)f(x,y)q(y)dydx\right|\geq\epsilon\right]\leq\delta, 
         &\mathbb{P}\left[\left|(\hat{q}_t-q_t)(y)f(x,y)p(x)dxdy\right|\geq\epsilon\right]\leq\delta, \label{concentration-1}
     \end{align}
     with the required numbers of particles $$M:=\mathcal{O}\left(\epsilon^{-2}\left[-\log(\delta)+d\log\left(1+L_0r_{\max\{C_1,C_2\}}(\epsilon,\lambda_1,\lambda_2)\epsilon^{-1}\right)\right]\right)$$ 
     on $\hat{p}_t(x)$ and $\hat{q}_t(y)$ for all $p(x)\in\cQ_{\cX}(C_1),\, q(y)\in\cQ_{\cY}(C_2)$ where $d=\max\{m,n\}$  and $r_C(\epsilon,\lambda_1,\lambda_2):=8\sqrt{\frac{\lambda_2+C}{\lambda_1}\max\{\log(\epsilon^{-1}), d\}}$.
 \end{lemma}

 We introduce the max-min gap function 
\begin{align}\label{max_min_gap}
    Q(\Bar{w},w):=\mathcal{L}(\Bar{p},q)-\mathcal{L}(p,\Bar{q}),
\end{align}
where $\Bar{w}:=(\Bar{p},\Bar{q})$ and $w:=(p,q)$. Note that the maximum of Eq.~\eqref{max_min_gap} equals to the duality gap on $(\Bar{p},\Bar{q})$. Considering the objective function in Eq.~\eqref{eq:problem_particle} with regularizer \eqref{eq:relative_reg}, we have
\begin{equation}\label{max-min-gap-control-entropy}
\begin{split}\max_{w\in\cP_2(\cX)\times\cP_2(\cY)}Q(\Bar{w},w)\geq&\lambda_2(\mathrm{KL}(\Bar{p}\|\Bar{p}^*)+\mathrm{KL}(\Bar{q}\|\Bar{q}^*)),\\
    \text{s.t.}\ \Bar{w}^*=(\Bar{p}^*,\Bar{q}^*)=&\argmax_{w'\in\cP_2(\cX)\times\cP_2(\cY)}Q(\Bar{w},w').
\end{split}
\end{equation}
As a result, the max-min gap on $\Bar{w}$ can regulate the KL divergence between $\Bar{w}$ and $\Bar{w}^*$.

To simplify the notation of the formula and capture the error of our algorithm, we use concise symbols to represent some complex intermediate variables: $r_C(\epsilon,\lambda_1,\lambda_2):=8\sqrt{\frac{\lambda_2+C}{\lambda_1}\max\{\log(\epsilon^{-1}), m, n\}}$ and $g(\lambda_2):=\log\left(1+\lambda_2+\lambda_2^2\right)$, and assume that the KL divergence between $(q_t,p_t)$ and $(q^*_t,p^*_t)$ can be bounded by small errors in outer loop, i.e.,
$$\mathrm{KL}(p_t\|p_t^*)\leq\delta_{t,1},\quad\mathrm{KL}(q_t\|q_t^*)\leq\delta_{t,2},$$
which is guaranteed by ULA algorithm with small step sizes (refer to section \ref{subsection-5.2}). We use the 
error bound quantities $\delta_{t,1}$
and $\delta_{t,2}$ to analyze the outer loop properties.

\subsection{Convergence of PAPAL}
\subsubsection{Outer Loop Error Bound}\label{subsection-5.1}
 The following theorem proves the convergence properties of our PAPAL.
\begin{theorem}\label{main-text-outer-loop-thm}
By setting parameters $\gamma_t\mu_t=\gamma_{t-1}$,\,  $\gamma_t\tau_t\leq\gamma_{t-1}(\tau_{t-1}+\lambda_2)$,\, $\gamma_t\eta_t\leq\gamma_{t-1}(\eta_{t-1}+\lambda_2)$,\, $\tau_t\eta_{t-1}\geq\mu_t,\, \mu_t\leq 1$ for any $t\geq1$ and optimizing the sub-problem with errors bounded by $\delta_{t,1}, \delta_{t,2}$ mentioned above, iterates of our method satisfies
{
\begin{align*}
    \sum_{t=1}^T\gamma_t Q(w_t,w)\leq&\gamma_1\eta_1\mathrm{KL}(p\|p_0^*)-\gamma_T(\eta_T+\lambda_2)\mathrm{KL}(p\|p_T^*)
    \\
    &+\gamma_1\tau_1\mathrm{KL}(q\|q_0^*)-\gamma_T\left(\tau_T+\lambda_2-\frac{1}{4\eta_T}\right)\mathrm{KL}(q\|q_T^*)
        \\
        &+4\sum_{t=1}^k\gamma_t (1+\mu_t)\epsilon
        \\
        &+\sum_{t=1}^k\gamma_t\left[\left(\lambda_2+\frac{2\lambda_1}{\alpha_{p_t^*}}\right)\delta_{t,1}+\left(30+\frac{8\lambda_1\sigma(p_t^*)}{\sqrt{\alpha_{p_t^*}}}\right)\sqrt{\delta_{t,1}}\right]\notag
        \\
&+\sum_{t=1}^k\gamma_t\left[\left(\lambda_2+\frac{2\lambda_1}{\alpha_{q_t^*}}\right)\delta_{t,2}+\left(34+\frac{8\lambda_1\sigma(q_t^*)}{\sqrt{\alpha_{q_t^*}}}\right)\sqrt{\delta_{t,2}}\right],
\end{align*}
}
for any $p\in\cQ_{\cX}(1), q\in\cQ_{\cY}(3)$ with probability at least $1-{\delta}$ for sample size $M:=\mathcal{O}\left(\epsilon^{-2}\left[\log(T\delta^{-1})+d\log\left(1+L_0r_{3}(\epsilon,\lambda_1,\lambda_2)\epsilon^{-1}\right)\right]\right)$
in each iteration where $d:=\max\{m,n\}$ , $\alpha_p$ is the log-Sobolev coefficient of density $p$ and $\sigma(p)$ denotes $(\mathbb{E}_p[\|\theta\|^2])^{1/2}$.
\end{theorem}

To obtain the complete convergence rate of the max-min gap, we need to specify the hyper-parameters $\tau_t,\eta_t,\gamma_t$ and sub-problem error bound $\delta_{t,1},\delta_{t,2}$ for the minimax problem \eqref{eq:problem_particle}. 
By setting the input parameters as
\begin{equation}\label{parameter-setting}
\mu_t = \mu=\frac{1}{2}\left(2+\lambda_2^2-\lambda_2\sqrt{4+\lambda_2^2}\right),\quad
    \eta_t=\tau_t=\frac{\lambda_2\mu}{1-\mu},\quad\gamma_t=\mu^{-t},
\end{equation}
and defining $c(p)=\min\left\{\frac{1}{2}, \alpha_{p}, \frac{\alpha_{p}}{2(\sigma(p))^2}\right\}$, $c(q)=\min\left\{\frac{1}{2}, \alpha_{q}, \frac{\alpha_{q}}{2(\sigma(q))^2}\right\}$ as the upper bound of the inner loop errors $\delta_{t,1}$, $\delta_{t,2}$, we are ready to provide the following outer loop error bound of our proposed algorithm with an inexact solver. It is worth noting that the parameter selection in \eqref{parameter-setting} meets the pre-conditions of Theorem \ref{main-text-outer-loop-thm}, in addition to satisfying $\mu\leq1$ and $\mu^2-(2+\lambda_2^2)\mu+1=0$. Moreover, noticing that MNE $w_*=(p_*,q_*)$ satisfies Eq.~\eqref{explicit-expression} and combining the optimality condition in Lemma \ref{lemma:optimal-KL-prox}, we can deduce that $p_*\in\cQ_{\cX}(1)$ and $q_*\in\cQ_{\cY}(3)$. Therefore, we derive the following convergence result by combining Theorem \ref{main-text-outer-loop-thm}.
\begin{corollary}\label{main-text-outer-loop-complexity}
    Setting the parameters as Eq.~\eqref{parameter-setting} and assuming the sub-problem error bounds satisfy $\max\{\delta_{t,1},\delta_{t,2}\}\leq\frac{1}{2}\min\{c(p_t^*), c(q_t^*)\}T^{-2J}$ for any fixed positive integer $J$ and $\epsilon>0$ at each step-$t$, the max-min gap between $\Bar{w}_T:=\frac{\sum_{t=1}^T\gamma_t w_t}{\sum_{t=1}^T\gamma_t}$ and $\bar{w}_T^*:=\argmax\limits_{w\in\mathcal{P}_2(\mathcal{X})\times\mathcal{P}_2(\mathcal{Y})}Q(\bar{w}_T,w)$ satisfies
    \begin{align}\label{max-min-gap}
        0\leq Q(\Bar{w}_T,\bar{w}_T^*)\leq&\frac{\lambda_2\mu^T}{1-\mu}\left[\mathrm{KL}(\bar{p}_T^*\|p_0)+\mathrm{KL}(\bar{q}_T^*\|q_0)\right]
        +8\epsilon+(10\lambda_1+\lambda_2+32)T^{-J}.
    \end{align}
    with probability at least $1-{\delta}$ for sample size $$M=\mathcal{O}\left(\epsilon^{-2}\left[\log( T\delta^{-1})+d\log\left(1+L_0r_{3}(\epsilon,\lambda_1,\lambda_2)\epsilon^{-1}\right)\right]\right),$$ in each iteration where $d=\max\{m,n\}$ , $\Bar{w}_T=\frac{\sum_{t=1}^T\gamma_t w_t}{\sum_{t=1}^T\gamma_t}$. Moreover, our high probability convergence also holds under the KL-divergence
    \begin{align}\label{KL-divergence-convergence}
        \mathrm{KL}(p_*\|p_T^*)+\frac{1}{2}\mathrm{KL}(q_*\|q_T^*)
        \leq&\mu^T[\mathrm{KL}(q_*\|q_0^*)+\mathrm{KL}(p_*\|p_0^*)]
        \notag
        \\
        &+\lambda_2^{-1}\left[8\epsilon+(10\lambda_1+\lambda_2+32)T^{-J}\right],
    \end{align}
    and the $W_2$ distance as follows
    \begin{equation*}
    \label{Wassenstain-convergence}
    \begin{split}
        \frac{1}{2}W_2^2(q_T, q_*)+W_2^2(p_T, p_*)\leq&\frac{4}{\min\{\alpha_{p_T^*}, \alpha_{q_T^*}\}}\left\{\mu^{T}\left[\mathrm{KL}(q_*\|q_0^*)+\mathrm{KL}(p_*\|p_0^*)\right]\right.
        \notag
        \\
        &+\left.\lambda_2^{-1}\left[8\epsilon+(10\lambda_1+\lambda_2+32)T^{-J}\right]\right\}+3T^{-2J},
    \end{split}
    \end{equation*}
    with respect to the MNE $w_*=(p_*,q_*)$ of Eq.~\eqref{eq:problem_particle}.
\end{corollary}


\subsubsection{Inner Loop Error Bound}\label{subsection-5.2}
In the preceding sub-section, we made the assumption that the KL divergence between $(q_t,p_t)$ and $(q_t^*,p_t^*)$ can be upper bounded by $(\delta_{t,1},\delta_{t,2})$. In this sub-section, we provide evidence that the ULA algorithm effectively constrains the inner error bound $\mathrm{KL}(p_t\|p_t^*), \mathrm{KL}(q_t\|q_t^*)$ to a sufficiently small value (i.e. $\delta_{t,1},\delta_{t,2}$) within polynomial time. 
We demonstrate that the summation of a quadratic function and a bounded function exhibits the log-Sobolev property. Similar to PDA \citep{nitanda2021particle}, we utilize the Holley-Strook argument to estimate the log-Sobolev constant. Consequently,  ULA converges efficiently to the target distribution, leveraging the log-Sobolev property.
\rev{There are also other ways of estimating the log-Sobolev constants, such as Miclo's trick. However, when leveraging the Lipschitz continuity, the log-Sobolev constant has exponential dependency on both dimension and Lipschitz constant. We focus on the  the bounded condition for better computation complexity.}

\begin{lemma}\label{log-sobolev-exp-bounded}
\citep{holley1986logarithmic}
    Consider a probability density $p(\theta)$ on $\mathbb{R}^{d_\theta}$ satisfying the log-Sobolev inequality with a constant $\alpha$. For a bounded function $H:\mathbb{R}^{d_\theta}\rightarrow \mathbb{R}$, we let a probability density $p_H(\theta)$ denote
    $
        p_H(\theta):=\frac{\exp{(H(\theta))}p(\theta)}{\mathbb{E}_p[\exp{(H(\theta))}]}.
    $
    Then, $p_H(\theta)$ satisfies the log-Sobolev inequality with $\alpha/\exp{(4\|H\|_\infty)}$.
\end{lemma}
It is clear that Eq.~\eqref{h_def} and \eqref{g_def} can be rewritten as a summation of a bounded function and quadratic function. Lemma \ref{log-sobolev-exp-bounded} shows that these two sub-problem can be log-Sobolev.
For any $t\geq0$,
we can compute that the log-Sobolev constant for  $q_t^*$, 
$p_t^*$ are $\frac{\lambda_1}{\lambda_2}\exp{\left(-\frac{12}{\lambda_2}\right)}$, $\frac{\lambda_1}{\lambda_2}\exp{\left(-\frac{4}{\lambda_2}\right)}$, respectively, only depending on the choice of $\lambda_1$ and $\lambda_2$ (Refer to Appendix Lemma \ref{estimation-log} for the constant estimation). The next theorem shows the convergence rate.
\begin{theorem}\label{main-text-thm-ULA} \citep{vempala2019rapid}
    Assume that $f(\theta)$ is smooth and has $L_1$-Lipschitz gradient, and consider a probability density $p(\theta)\propto\exp{(-f(\theta))}$ satisfying the log-Sobolev inequality with constant $\alpha$, ULA with step size $0<\iota<\frac{\alpha}{4L_1^2}$ satisfies
    $
        \mathrm{KL}(\rho_k\|p)\leq e^{-\alpha\iota k}\mathrm{KL}(\rho_0\|p)+\frac{8\iota d_\theta L_1^2}{\alpha},
    $
    where $d_\theta$ is the dimension.
\end{theorem}

\subsubsection{Global Convergence}\label{subsection-5.3}
By leveraging Corollary \ref{main-text-outer-loop-complexity} and Theorem \ref{main-text-thm-ULA}, we are now ready to provide the global convergence of PAPAL. The ensuing theorem establishes the global error bound, achieved through the acquisition of a pair of $\frac{1}{2}\min\{c(p_t^*), c(q_t^*)\}T^{-2J}$-approximate sub-problem solutions $(p_t,q_t)$ under metrics $\mathrm{KL}(\cdot\|p_t^*)$ and $\mathrm{KL}(\cdot\|q_t^*)$ at each step-$t$, along the trajectory of ULA.
\begin{theorem}\label{main-text-global-convergence}
    Let $\epsilon$ be the desired accuracy and set the parameters $\gamma_t,\, \eta_t,\, \tau_t,\, \mu_t$ as in Eq.~\eqref{parameter-setting}. Assuming Assumptions \ref{ass:a1},  \ref{ass:a2} and \ref{ass:a4}, executing Langevin algorithm with step size $\iota=\mathcal{O}\left(\min\left\{1,\frac{\lambda_1^2}{\lambda_2^2}\right\}\frac{\lambda_1\lambda_2\epsilon^2}{d^2(L_1+\lambda_1)^2\exp{(30/\lambda_2)}}\right)$ for $T_t=\mathcal{O}\left(\iota^{-1}\lambda_1^{-1}[(\lambda_2+1)\exp{(30/\lambda_2)}\right.$ $\left.\log(\epsilon^{-1})]\right)$ iterations on the inner loop yeilds an $\epsilon$-mixed Nash equilibrium: $\max\limits_{w\in\cP_2(\cX)\times\cP_2(\cY)}Q(\Bar{w},w)\leq\epsilon$ with probability at least $1-\delta$ when running Algorithm \ref{alg:algorithm} for iterations given by
    $$T=\mathcal{O}\left(\max\left\{\log\left((1+\lambda_2^{-1})\epsilon^{-1}\right)g^{-1}(\lambda_2),(32+10\lambda_1+\lambda_2)\epsilon^{-1/J}\right\}\right),$$ 
    and sample size $M=\mathcal{O}\left(\epsilon^{-2}\left[\log\left(T\delta^{-1} \right)+d\log\left(1+L_0r_{3}(\epsilon,\lambda_1,\lambda_2)\epsilon^{-1}\right)\right]\right)$
    on the outer loop, where $d = \max\{m,n\}$.
\end{theorem}
Then we obtain the computation time $T=\Tilde{\mathcal{O}}(\max\{\log(\epsilon^{-1})\lambda_2^{-1},\, \epsilon^{-1/J}\})$ if $\lambda_2 \rightarrow 0$, which indicates the limiting behavior of vanishing entropy regularizer. The dependence on $\lambda_2^{-1}$ in our convergence rate is a direct consequence of the classical LSI perturbation lemma \citep{holley1986logarithmic}, which is likely unavoidable for Langevin-based methods in the most general setting \citep{menzPoin2014}. Moreover, according to Eq.~\eqref{max-min-gap-control-entropy}, theorem \ref{main-text-global-convergence} can also control $\mathrm{KL}(\Bar{p}\|\Bar{p}^*)$ and $\mathrm{KL}(\Bar{q}\|\Bar{q}^*)$ similarly. In addition, based on Corollary \ref{main-text-outer-loop-complexity}, we can also propose the following global convergence under metrics $\mathrm{KL}(p_*\|\cdot), \mathrm{KL}(q_*\|\cdot)$ and $W_2(\cdot,p_*), W_2(\cdot,q_*)$ with respect to MNE $w_*=(p_*,q_*)$ of Eq.~\eqref{eq:problem_particle}.
\begin{corollary}\label{col:kl_conv}
    Let $\epsilon$ be the desired accuracy. Under Assumptions \ref{ass:a1}, \ref{ass:a2} and  \ref{ass:a4}, if we execute ULA with a step size of $\iota=\mathcal{O}\left(\min\left\{\frac{\lambda_1}{\lambda_2},\frac{\lambda_1^3}{\lambda_2^3}\right\}\frac{\lambda_2^4\epsilon^2}{(\max\{m,n\})^2(L_1+\lambda_1)^2\exp{(30/\lambda_2)}}\right)$ for $T_t=\mathcal{O}\left(\iota^{-1}\lambda_1^{-1}[(\lambda_2+1)\exp{(30/\lambda_2)}\log(\lambda_2^{-1}\epsilon^{-1})]\right)$ iterations on the inner loop, we can approximate an $\epsilon$-accurate KL divergence: $\mathrm{KL}(p_*\|p_T^*)+\frac{1}{2}\mathrm{KL}(q_*\|q_T^*)\leq \epsilon$ with probability no less than $1-\delta$ 
    when running Algorithm \ref{alg:algorithm} for iterations $T$
    and sample size $M$ satisfy
    $$T=\mathcal{O}\left(\max\left\{\log\left((\lambda_2\epsilon)^{-1}\right)g^{-1}(\lambda_2),\, \, (32+10\lambda_1+\lambda_2)(\lambda_2\epsilon)^{-1/J}\right\}\right),$$ 
    $$M = \mathcal{O}\left(\lambda_2^{-2}\epsilon^{-2}\left[\log(T\delta^{-1})+d\log\left(1+L_0r_3(\lambda_2\epsilon,\lambda_1,\lambda_2)\right)\right]\right),$$
    on the outer loop with respect to the MNE $w_*=(p_*,q_*)$ of \eqref{eq:problem_particle}, where $d = \max\{m,n\}$. Furthermore, noting that $p_T$ and $q_T$ satisfying $\max\{\mathrm{KL}(p_T\|p_T^*),\, \mathrm{KL}(q_T\|q_T^*)\}\leq\mathcal{O}(\lambda_2^2\epsilon^2)$, we can also approximate an $\epsilon$-accurate $W_2$ distance: $W_2^2(p_T,p_*)+\frac{1}{2}W_2^2(q_T,q_*)\leq\epsilon$ with similar high probability when running Algorithm \ref{alg:algorithm} for iterations $T$
    and sample size $M$ satisfy
    $$T=\mathcal{O}\left(\max\left\{\log\left((\Tilde{g}(\lambda_1,\lambda_2)\epsilon)^{-1}\right)g^{-1}(\lambda_2),\, \, (32+10\lambda_1+\lambda_2)(\Tilde{g}(\lambda_1,\lambda_2)\epsilon)^{-1/J}\right\}\right),$$ 
    $$M = \mathcal{O}\left(\Tilde{g}^{-2}(\lambda_1,\lambda_2)\epsilon^{-2}\left[\log(T\delta^{-1})+d\log\left(1+L_0r_3(\Tilde{g}(\lambda_1,\lambda_2)\epsilon,\lambda_1,\lambda_2)\right)\right]\right),$$
    on the outer loop where $\Tilde{g}(\lambda_1,\lambda_2):=\min\left\{1,\lambda_2,\lambda_1(\lambda_2)^{-1}\exp\{-12(\lambda_2)^{-1}\}\right\}$.
\end{corollary}
\rev{Inspired by MALA-based sampling algorithms \citep{chewi2021optimal, wu2022minimax, altschuler2024faster}, we also provide the global convergence of PAPAL using proximal sampler with MALA (refer to Algorithm \ref{alg:proximal}) for inner loop.
\begin{corollary}
    \label{remark-MALA}
    Let $\epsilon$ be the desired accuracy and set the parameters $\gamma_t,\, \eta_t,\, \tau_t,\, \mu_t$ as in Eq.~\eqref{parameter-setting}. Assuming Assumptions \ref{ass:a1}, \ref{ass:a2} and \ref{ass:a4}, executing proximal sampler with MALA (Algorithm \ref{alg:proximal}) for $T_t=\tilde{\mathcal{O}}\left((3L_1+\lambda_1)\exp(12/\lambda_2)\lambda_1^{-1}d^{1/2}\left(\log(\epsilon^{-1})+1+\lambda_2^{-1}\right)^4\right)$ iterations on the inner loop yeilds an $\epsilon$-mixed MNE: $\max\limits_{w\in\cP_2(\cX)\times\cP_2(\cY)}Q(\Bar{w},w)\leq\epsilon$ with probability at least $1-\delta$ when running Algorithm \ref{alg:algorithm} for iterations given by
    $$T=\mathcal{O}\left(\max\left\{\log\left((1+\lambda_2^{-1})\epsilon^{-1}\right)g^{-1}(\lambda_2),(32+10\lambda_1+\lambda_2)\epsilon^{-1/J}\right\}\right),$$ 
    and sample size $M=\mathcal{O}\left(\epsilon^{-2}\left[\log\left(T\delta^{-1} \right)+d\log\left(1+L_0r_{3}(\epsilon,\lambda_1,\lambda_2)\epsilon^{-1}\right)\right]\right)$
    on the outer loop, where $d = \max\{m,n\}$.
\end{corollary}}

\subsection{Convergence of Stochastic PAPAL}\label{Sto-PAPAL}
In this subsection, we consider the stochastic version of  Eq.\eqref{eq:problem_particle} as follows
\begin{align}\label{eq:sto-finite-particle-problem}
    \min_{p\in\mathcal{P}_2(\mathcal{X})} \max_{q\in\mathcal{P}_2(\mathcal{Y})} &\EE_{p(x)} \EE_{q(y)} [f(x,y):=\EE[G(x,y,\xi)]]+\cR(p) - \cR(q),
\end{align}
where $\cX=\RR^m, \cY=\RR^n$, and the stochastic component $G(x,y,\xi)$ is indexed by some random vector $\xi$ whose probability distribution $\rho$ is supported on $\Lambda\subset\RR^d$ and $G:\cX\times\Lambda\rightarrow\RR$. We assume that the expectation 
	\begin{align}
		\EE[G(x,y,\xi)]=\int_{\Lambda}G(x,y,\xi)d\rho(\xi),
	\end{align}
	is well defined and finite valued for every $(x,y)\in\cX\times\cY$. To prevent ambiguity, we will continue to use $\cL(p,q)$ to denote $\EE_{p(x)} \EE_{q(y)} [f(x,y)]+\cR(p) - \cR(q)$.

For stochastic version, we have additional assumptions,
\begin{enumerate}[label={\textbf{[A$_\text{\arabic*}$]}}]
  \setcounter{enumi}{3}
  \item \label{Sto-Assumptions1} There exists $L_0$ such that $G(\cdot,\cdot,\xi)$ is $L_0$-Lipschitz for any $\xi\in\Lambda$.
			\item \label{Sto-Assumptions2}  For any $\xi\in\Lambda$, $G(\cdot,\cdot,\xi)$ has $L_1-$Lipschitz gradient.
			\item \label{Sto-Assumptions3}  For any $x,y\in\cX\times\cY$ and $\xi\in\Lambda$, $|G(x,y,\xi)|\leq1.$
\end{enumerate}
To accommodate the problem setting Eq.~\eqref{eq:sto-finite-particle-problem}, we need to make some straightforward modifications to the deterministic PAPAL algorithm: we replace function $\hat{\phi}_t$ with $(\hat{\phi}_t)_{\Bar{\xi}_N}:=-(1+\mu_t)\EE_{\hat{p}_{t-1}}\left[\frac{1}{N}\sum_{i=1}^NG(\cdot,y,\xi_{t,i})\right]+\mu_t\EE_{\hat{p}_{t-2}}\left[\frac{1}{N}\sum_{i=1}^NG(\cdot,y,\xi_{t,i})\right]$, and $\hat{\psi}_t$ with $(\hat{\psi}_t)_{\Bar{\xi}_N}:=\EE_{\hat{q}_{t}}\left[\frac{1}{N}\sum_{i=1}^NG(\cdot,y,\xi_{t,i})\right]$ in each iteration $t$ of the Algorithm \ref{alg:algorithm}, where $(\xi_{t,1},\cdots,\xi_{t,N})$ denote $N$ i.i.d random variables, depending on $t$-th update with $\xi_{t,i}=\xi$ for any $i\in[1:N]$.

\subsubsection{Outer Loop Error Bound}
In order to simplify the notation of our formula, we also assume that the KL divergence between $(q_t,p_t)$ and $(q^*_t,p^*_t)$ can be bounded by small errors $(\delta_{t,1}, \delta_{t,2})$ at each step-$t$. As an extension of Corollary \ref{main-text-outer-loop-complexity} in the stochastic version, the following lemma provides the high-probability global convergence of the stochastic PAPAL algorithm under the condition of sufficiently small inner loop error.
\begin{lemma}\label{Stochastic-Outer-Loop-Lemma}
        Setting the parameters $\gamma_t,\, \eta_t,\, \tau_t,\, \mu_t$ as in Eq.~\eqref{parameter-setting} and assuming the sub-problem error bounds satisfy $\max\{\delta_{t,1},\delta_{t,2}\}\leq\frac{1}{2}\min\{c(p_t^*), c(q_t^*)\}T^{-2J}$ for any fixed positive integer $J$ and $\epsilon>0$ at each step-$t$. If Assumptions \ref{Sto-Assumptions1}, \ref{Sto-Assumptions2} and \ref{Sto-Assumptions3} hold, then the max-min gap between $\Bar{w}_T$ and $\bar{w}_T^*:=\argmax\limits_{w\in\mathcal{P}_2(\mathcal{X})\times\mathcal{P}_2(\mathcal{Y})}Q(\bar{w}_T,w)$ satisfies
    \begin{align}\label{Max-Min-gap}
        0\leq Q(\Bar{w}_T,\bar{w}_T^*)\leq&\frac{\lambda_2\mu^T}{1-\mu}\left[\mathrm{KL}(\bar{p}_T^*\|p_0)+\mathrm{KL}(\bar{q}_T^*\|q_0)\right]
        +12\epsilon+(10\lambda_1+\lambda_2+32)T^{-J}.
    \end{align}
    with probability at least $1-{\delta}$ for sample size $$M=N=\mathcal{O}\left(\epsilon^{-2}\left[\log(2T\delta^{-1})+d\log\left(1+L_0r_{3}(\epsilon,\lambda_1,\lambda_2)\epsilon^{-1}\right)\right]\right),$$ in each iteration where $d=\max\{m,n\}$ , $\Bar{w}_T=\frac{\sum_{t=1}^T\gamma_t w_t}{\sum_{t=1}^T\gamma_t}$. In addition, our high probability convergence also holds under the KL-divergence
    \begin{align}\label{KL-divergence-convergence-sto}
        \mathrm{KL}(p_*\|p_T^*)+\frac{1}{2}\mathrm{KL}(q_*\|q_T^*)
        \leq&\mu^T[\mathrm{KL}(q_*\|q_0^*)+\mathrm{KL}(p_*\|p_0^*)]
        \notag
        \\
        &+\lambda_2^{-1}\left[12\epsilon+(10\lambda_1+\lambda_2+32)T^{-J}\right],
    \end{align}
    and the $W_2$ distance as follows
    \begin{equation*}
    \label{Wassenstain-convergence-sto}
    \begin{split}
        \frac{1}{2}W_2^2(q_T, q_*)+W_2^2(p_T, p_*)\leq&\frac{4}{\min\{\alpha_{p_T^*}, \alpha_{q_T^*}\}}\left\{\mu^{T}\left[\mathrm{KL}(q_*\|q_0^*)+\mathrm{KL}(p_*\|p_0^*)\right]\right.
        \notag
        \\
        &+\left.\lambda_2^{-1}\left[12\epsilon+(10\lambda_1+\lambda_2+32)T^{-J}\right]\right\}+3T^{-2J},
    \end{split}
    \end{equation*}
    with respect to the MNE $w_*=(p_*,q_*)$ of Eq.~\eqref{eq:sto-finite-particle-problem}.
\end{lemma}
For the inner loop, the stochastic version is equivalent to the deterministic one (refer \ref{subsection-5.2}). Thus, we omit the discussion of the inner loop error bound for Stochastic PAPAL.

\subsubsection{Global Convergence}
\begin{theorem}\label{sto-main-text-global-convergence}
    Let $\epsilon$ be the desired accuracy. Under Assumptions \ref{Sto-Assumptions1},  \ref{Sto-Assumptions2} and \ref{Sto-Assumptions3}, if we run Langevin algorithm with step size $\iota=\mathcal{O}\left(\min\left\{1,\frac{\lambda_1^2}{\lambda_2^2}\right\}\frac{\lambda_1\lambda_2\epsilon^2}{d^2(L_1+\lambda_1)^2\exp{(30/\lambda_2)}}\right)$ for $T_t=\mathcal{O}\left(\iota^{-1}\lambda_1^{-1}[(\lambda_2+1)\exp{(30/\lambda_2)}\right.$ $\left.\log(\epsilon^{-1})]\right)$ iterations on the inner loop, then we can achieve an $\epsilon$-mixed Nash equilibrium: $\max\limits_{w\in\cP_2(\cX)\times\cP_2(\cY)}Q(\Bar{w},w)\leq\epsilon$ with probability at least $1-\delta$ when running Algorithm \ref{alg:algorithm} for iterations
    $$T=\mathcal{O}\left(\max\left\{\log\left((1+\lambda_2^{-1})\epsilon^{-1}\right)g^{-1}(\lambda_2),(32+10\lambda_1+\lambda_2)\epsilon^{-1/J}\right\}\right),$$ 
    and sample size
    $$M=N=\mathcal{O}\left(\epsilon^{-2}\left[\log\left(2T\delta^{-1} \right)+d\log\left(1+L_0r_{3}(\epsilon,\lambda_1,\lambda_2)\epsilon^{-1}\right)\right]\right),$$
    on the outer loop, where $d = \max\{m,n\}$ .
\end{theorem}
Similar to the conclusion of Corollary \ref{col:kl_conv} in the deterministic setting, the stochastic PAPAL algorithm also exhibits global convergence properties with respect to the KL divergence and $W_2$ distance of MNE $w_*=(p_*,q_*)$.
\begin{corollary}\label{Sto-col:kl_conv}
    Let $\epsilon$ be the desired accuracy. Under Assumptions \ref{Sto-Assumptions1}, \ref{Sto-Assumptions2} and  \ref{Sto-Assumptions3}, if we execute ULA with a step size of $\iota=\mathcal{O}\left(\min\left\{\frac{\lambda_1}{\lambda_2},\frac{\lambda_1^3}{\lambda_2^3}\right\}\frac{\lambda_2^4\epsilon^2}{(\max\{m,n\})^2(L_1+\lambda_1)^2\exp{(30/\lambda_2)}}\right)$ for $T_t=\mathcal{O}\left(\iota^{-1}\lambda_1^{-1}[(\lambda_2+1)\exp{(30/\lambda_2)}\log(\lambda_2^{-1}\epsilon^{-1})]\right)$ iterations on the inner loop, we can approximate an $\epsilon$-accurate KL divergence: $\mathrm{KL}(p_*\|p_T^*)+\frac{1}{2}\mathrm{KL}(q_*\|q_T^*)\leq \epsilon$ with probability no less than $1-\delta$ 
    when running Algorithm \ref{alg:algorithm} for iterations $T$
    and sample size $M,N$ satisfy
    $$T=\mathcal{O}\left(\max\left\{\log\left((\lambda_2\epsilon)^{-1}\right)g^{-1}(\lambda_2),\, \, (32+10\lambda_1+\lambda_2)(\lambda_2\epsilon)^{-1/J}\right\}\right),$$ 
    $$M=N= \mathcal{O}\left(\lambda_2^{-2}\epsilon^{-2}\left[\log(T\delta^{-1})+d\log\left(1+L_0r_3(\lambda_2\epsilon,\lambda_1,\lambda_2)\right)\right]\right),$$
    on the outer loop with respect to the MNE $w_*=(p_*,q_*)$ of \eqref{eq:problem_particle}, where $d = \max\{m,n\}$. Furthermore, noting that $p_T$ and $q_T$ satisfying $\max\{\mathrm{KL}(p_T\|p_T^*),\, \mathrm{KL}(q_T\|q_T^*)\}\leq\mathcal{O}(\lambda_2^2\epsilon^2)$, we can also approximate an $\epsilon$-accurate $W_2$ distance: $W_2^2(p_T,p_*)+\frac{1}{2}W_2^2(q_T,q_*)\leq\epsilon$ with similar high probability when running Algorithm \ref{alg:algorithm} for iterations $T$
    and sample size $M,N$ satisfy
    $$T=\mathcal{O}\left(\max\left\{\log\left((\Tilde{g}(\lambda_1,\lambda_2)\epsilon)^{-1}\right)g^{-1}(\lambda_2),\, \, (32+10\lambda_1+\lambda_2)(\Tilde{g}(\lambda_1,\lambda_2)\epsilon)^{-1/J}\right\}\right),$$ 
    $$M=N= \mathcal{O}\left(\Tilde{g}^{-2}(\lambda_1,\lambda_2)\epsilon^{-2}\left[\log(T\delta^{-1})+d\log\left(1+L_0r_3(\Tilde{g}(\lambda_1,\lambda_2)\epsilon,\lambda_1,\lambda_2)\right)\right]\right),$$
    on the outer loop where $\Tilde{g}(\lambda_1,\lambda_2):=\min\left\{1,\lambda_2,\lambda_1(\lambda_2)^{-1}\exp\{-12(\lambda_2)^{-1}\}\right\}$.
\end{corollary}

\rev{
 Although stochastic PAPAL needs large batch size $\Theta(\epsilon^{-2})$ for the outer loop, the inner loop can also be implemented with stochastic algorithms, such as some SGLD-based algorithms \citep{freund2022convergence,das2023utilising,huang2024faster}.
According to the result of \citet{das2023utilising}, we can extend our algorithm with batch size $1$. We begin the discussion with the following assumption:
\begin{assumption}[Sub-Gaussian Stochastic Gradient Growth]\label{ssgg}
    For $L_0>0$, the function $G(x,y,\xi)$  satisfies the following norm-subgaussianity condition for the stochastic gradient
    \begin{align}
        \mathbb{P}\left\{\|\nabla G(x,y,\xi)-\nabla f(x,y)\|\geq t\left| x,y\right.\right\}\leq 2\exp\left\{-t^2/(8L_0^2)\right\}, \forall (x,y)\in\cX\times\cY.\notag
    \end{align}
\end{assumption}
Therefore, applying \citet[Theorem 5]{das2023utilising}, we have following estimation of sub-problem
\begin{align}
    \max\left\{\mathrm{KL}(p_t||p_t^*),\mathrm{KL}(q_t||q_t^*)\right\}\leq&\exp\{-\lambda\iota T_t/2\}\max\left\{\mathrm{KL}(p_0||p_t^*),\mathrm{KL}(q_0||q_t^*)\right\}\notag
    \\
    &+\cO\left(\frac{(3L_1+\lambda_1)^2\iota d}{\lambda\lambda_2^2}+\frac{M\iota (3L_1+\lambda_1)L_0^2}{\lambda\lambda_2^3 N}+\frac{M\iota L_0^4}{\lambda \lambda_2^4 N^2}\right),\notag
\end{align}
using SGLD with arbitrary batch size $N\in\mathbb{N}_+$ in each iteration $t$ of PAPAL algorithm, where $M$ is the number of particles, $T_t$ is the number of iterations on the inner loop, $\iota$ is step size, $\lambda=\lambda_1\exp\{-12\lambda_2^{-1}\}\lambda_2^{-1}$ and $d=\max\{m,n\}$.
\begin{corollary}
    Let $\epsilon$ be the desired accuracy set the parameters $\gamma_t,\, \eta_t,\, \tau_t,\, \mu_t$ as default. Under Assumptions \ref{Sto-Assumptions1}, \ref{Sto-Assumptions2}, \ref{Sto-Assumptions3} and subgaussian stochastic gradient growth condition \ref{ssgg}, if we execute SGLD with batch size $N\in\mathbb{N}_+$ and a step size of 
    $$
    \iota=\cO\left(\frac{\lambda_1}{\lambda_2}\min\left\{\lambda_1^2,\lambda_2^{2}\right\}\min\left\{\frac{\lambda_2^{2}N^2}{L_0^4},\frac{1}{(3L_1+\lambda_1)^2}\right\}\frac{\exp(-30/\lambda_2)}{\max\{M,d\}}\right)
    $$ 
    for $T_t=\cO\left(\iota^{-1}\lambda_1^{-1}[(\lambda_2+1)\exp(30/\lambda_2)\log(\epsilon^{-1})]\right)$ iterations on the inner loop yeilds an $\epsilon-$mixed Nash equilibrium with probability at least $1-\delta$ when running PAPAL for iteration given by 
    $$T=\mathcal{O}\left(\max\left\{\log\left((1+\lambda_2^{-1})\epsilon^{-1}\right)g^{-1}(\lambda_2),(32+10\lambda_1+\lambda_2)\epsilon^{-1/J}\right\}\right),$$ 
    and sample size $M=\mathcal{O}\left(\epsilon^{-2}\left[\log\left(T\delta^{-1} \right)+d\log\left(1+L_0r_{3}(\epsilon,\lambda_1,\lambda_2)\epsilon^{-1}\right)\right]\right)$
    on the outer loop, where $d = \max\{m,n\}$.
\end{corollary}}

\section{Conclusion}
We propose a particle-based algorithm -- PAPAL to solve min-max problems in continuous probability space. This algorithm explores the global convergence property for finding general mixed Nash equilibria in infinite-dimensional two-player zero-sum games. We provide a quantitative computation and sample complexity analysis of PAPAL with a suitable entropy regularization, which first shows the feasibility of an implementable particle-based algorithm in this task.




\vskip 0.2in
\bibliography{bib}

\appendix




\section{Auxiliary lemmas}

\begin{proposition}
\citep{bakry2006diffusions}
    Consider a probability density $p(\theta)\propto\exp{(-f(\theta))}$, where $f:\mathbb{R}^{d_\theta}\rightarrow\mathbb{R}$ is a smooth function. If there exists $\alpha>0$ such that $\nabla^2 f\succeq \alpha I_{d_\theta}$, then $p(\theta)$ satisfies log-Sobolev inequality with $\alpha$.
\end{proposition}
 
\begin{definition}[Talagrand's inequality]\label{Talagrand's-inequality} 
    We say that a probability density $p(\theta)$ satisfies Talagrand's inequality \citep{talagrand1995concentration} with a constant $\alpha>0$ if for any probability density $p'(\theta)$,
    $
        \frac{\alpha}{2}W_2^2(p',p)\leq\mathrm{KL}(p'\|p),
    $
    where $W_2(p',p)$ denotes the 2-Wasserstein distance between $p'(\theta)$ and $p(\theta)$.
\end{definition}
\begin{lemma} \citep{otto2000generalization,bobkov2001hypercontractivity} If a probability density $p(\theta)$ satisfies the log-Sobolev inequality with a constant $\alpha>0$, then $p(\theta)$ satisfies Talagrand's inequality with the same $\alpha$.
\end{lemma}

\begin{theorem} \label{ULA-error-bound}\citep{vempala2019rapid}
    Assume $f(\theta)$ is smooth and has $L_1$-Lipschitz gradient, and consider a probability density $p(\theta)\propto\exp{(-f(\theta))}$ satisfying the log-Sobolev inequality with constant $\alpha$, ULA with step size $0<\iota<\frac{\alpha}{4L_1^2}$ satisfies
    \begin{align}
        \mathrm{KL}(\rho_k\|p)\leq e^{-\alpha\iota k}\mathrm{KL}(\rho_0\|p)+\frac{8\iota d_\theta L_1^2}{\alpha}.
    \end{align}
\end{theorem}

\begin{lemma}[Tail bound for Chi-squared variable \cite{laurent2000adaptive}]\label{Tail-bound}
    Let $\theta\sim\cN(0,\sigma^2 I_{d_\theta})$ be a Gaussian random variable on $\mathbb{R}^{d_\theta}$. Then, we obtain
    \begin{align}
        \mathbb{P}[\|\theta\|^2\geq2c]\leq\exp{\left(-\frac{c}{10\sigma^2}\right)},\, \, \, \, \forall\, c\geq d_\theta\sigma^2.
    \end{align}
\end{lemma}
\begin{lemma}[Hoeffding's inequality \citep{hoeffding1963probability}]\label{Hoeffding's-inequality}
    Let $X, X_1, \cdots, X_M$ be i.i.d random variables taking values in $[-a, a]$ for $a>0$. Then, for any $\rho>0$, we get
    \begin{align}
        \mathbb{P}\left[\left|\frac{1}{M}\sum_{r=1}^M X_r-\mathbb{E}[X]\right|>\rho\right]\leq2\exp{\left(-\frac{\rho^2 M}{2a^2}\right)}.
    \end{align}
\end{lemma}

\section{Existence of Mixed Nash Equilibrium}\label{existence-MNE}

	\begin{proof}[Proof of Theorem \ref{existence-nash}]
        For the purpose of ensuring a seamless proof, it is necessary to introduce the following lemma.
        \begin{lemma}\label{auxiliary-lemma}
		The following set
		\begin{align}
			\mathcal{P}(g,\mathcal{L},\RR^d):=&\left\{p:\mathbb{R}^d\rightarrow\mathbb{R}^+\left|p\leq g \text{ a.e. on } \RR^d,\, p\in\mathcal{B}(\mathcal{L}),\, \int_{\mathbb{R}^d}p(x)\mathrm{d}x=1\right.\right\},
		\end{align}
		is a covex compact set where
		$$
		\mathcal{B}(\mathcal{L}):=\left\{f:\mathbb{R}^d\rightarrow\mathbb{R}\left|\max\limits_{x,y\in B_k(\mathbf{0})}\frac{|f(x)-f(y)|}{\|x-y\|}\leq L_k,\, \forall k\in\mathbb{N}^+\right.\right\},
		$$
		$g\in L^1$ and $\mathcal{L}:=\{L_k\}_{k=1}^{\infty}$.
	\end{lemma}
	\begin{proof}
		It is easy to notice that $\mathcal{P}(g,\mathcal{L},\RR^d)$ is convex. Considering the infinite sets $\{p_n\}_{n=1}^{\infty}\subset\mathcal{P}(g,\mathcal{L}, \RR^d)$ and $\left\{B_n(\mathbf{0})\subset\mathbb{R}^d\right\}_{n=1}^{\infty}$, we proceed to construct a subsequence of $\{p_n\}_{n=1}^{\infty}$ in the following manner:
		\begin{enumerate}
			\item Applying the Arzelà–Ascoli theorem, we ascertain the existence of a subsequence
			$\{p_{1n}\}_{n=1}^{\infty}$ of $\{p_n\}_{n=1}^{\infty}$, which converges uniformly to 
			$p^*|_{B_1(\mathbf{0})}$ on $B_1(\mathbf{0})$. Moreover, $p^*|_{B_1(\mathbf{0})}$ satisfies the Lipschitz condition: 
            $$
            \max\limits_{x,y\in B_1(\mathbf{0})}\frac{|p^*|_{B_1(\mathbf{0})}(x)-p^*|_{B_1(\mathbf{0})}(y)|}{\|x-y\|}\leq L_1.
            $$
			\item For $k\in[1:\infty)$, we employ the Arzelà–Ascoli theorem once again to obtain a subsequence $\{p_{(k+1)n}\}_{n=1}^{\infty}$ of $\{p_{kn}\}_{n=k}^{\infty}$, which converges uniformly to $p^*|_{B_{k+1}(\mathbf{0})}$ on $B_{k+1}(\mathbf{0})$.Furthermore,  $p^*|_{B_{k+1}(\mathbf{0})}$ satisfies the Lipschitz condition: 
            $$
            \max\limits_{x,y\in B_{k+1}(\mathbf{0})}\frac{|p^*|_{B_{k+1}(\mathbf{0})}(x)-p^*|_{B_{k+1}(\mathbf{0})}(y)|}{\|x-y\|}\leq L_{k+1}.
            $$
		\end{enumerate}
		As a result, we can establish the existence of a subsequence $\{p_{kk}\}_{k=1}^{\infty}$ of $\{p_n\}_{n=1}^{\infty}$, which converges to $p^*$ a.e. on $\mathbb{R}^d$. Noting $g(x)$ belongs to $L^1$, then it's intuitive that $p^*\in L^1$ and $\|p_{kk}-p^*\|_{L^1}\rightarrow 0$ by using dominate convergence theorem, which indicates that $p^*\in\mathcal{P}(g,\mathcal{L}, \RR^d)$.
	\end{proof}
 
		Note that the following two expressions are equivalent
		\begin{enumerate}
			\item \begin{align}
				\hat{p}=&\argmin\limits_{p\in\mathcal{P}_2(\RR^m)}\mathbb{E}_{p(x)}\mathbb{E}_{\bar{q}(y)}[f(x,y)]+\mathcal{R}(p)-\mathcal{R}(\bar{q}),
				\\
				\hat{q}=&\argmin\limits_{q\in\mathcal{P}_2(\RR^n)}\mathbb{E}_{\bar{p}(x)}\mathbb{E}_{q(y)}[f(x,y)]+\mathcal{R}(\bar{p})-\mathcal{R}(q).
			\end{align}
			\item \begin{align}
				\hat{p}\propto&\exp\left\{-(\lambda_2)^{-1}\left[\mathbb{E}_{\bar{q}(y)}[f(x,y)]+\lambda_1\|x\|^2\right]\right\},
				\\
				\hat{q}\propto&\exp\left\{-(\lambda_2)^{-1}\left[-\mathbb{E}_{\bar{p}(x)}[f(x,y)]+\lambda_1\|y\|^2\right]\right\},
			\end{align}
		\end{enumerate}
	for any $(\bar{p},\bar{q})\in\mathcal{P}_2(\RR^m)\times\mathcal{P}_2(\RR^n)$.
	Therefore, we denote operator $\boldsymbol{F}(p,q):\mathcal{P}_2(\RR^m)\times\mathcal{P}_2(\RR^n)\rightarrow \mathcal{P}_2(\RR^m)\times\mathcal{P}_2(\RR^n)$ as
	\begin{align}\label{operator-expression}
		\boldsymbol{F}(p,q):=\begin{pmatrix}
			\frac{\exp\left\{-(\lambda_2)^{-1}\left[\mathbb{E}_{q(y)}[f(x,y)]+\lambda_1\|x\|^2\right]\right\}}{Z_q}
			\\
			\frac{\exp\left\{-(\lambda_2)^{-1}\left[-\mathbb{E}_{p(x)}[f(x,y)]+\lambda_1\|y\|^2\right]\right\}}{Z_p}
		\end{pmatrix},
	\end{align}
	where $Z_q$ and $Z_p$ represent the normalization factors. Specifically, 
	\begin{align}
		Z_q=&\int_{\RR^n}\exp\left\{-(\lambda_2)^{-1}\left[\mathbb{E}_{q(y)}[f(x,y)]+\lambda_1\|x\|^2\right]\right\}\mathrm{d}x,
		\\ Z_p=&\int_{\RR^m}\exp\left\{-(\lambda_2)^{-1}\left[-\mathbb{E}_{p(x)}[f(x,y)]+\lambda_1\|y\|^2\right]\right\}\mathrm{d}y.
	\end{align} 
	Hence, to establish the proof of the theorem, it is sufficient to demonstrate the existence of a fixed point for the operator $\boldsymbol{F}$. As established by Lemma \ref{auxiliary-lemma},  $\mathcal{P}(g,\mathcal{L},\RR^d)$ is a compact set for fixed $g$ and $\mathcal{L}$. Therefore, the next step is to prove that for any $(p,q)\in\mathcal{P}(\RR^m)\times\mathcal{P}(\RR^n)$, $\boldsymbol{F}(p,q)$ belongs to $\mathcal{P}(g_p,\mathcal{L}_p,\RR^m)\times\mathcal{P}(g_q,\mathcal{L}_q,\RR^n)$ where the selection of $g_p, g_q$ and $\mathcal{L}_p, \mathcal{L}_q$ depends on following estimation. For simplicity, we define $(\boldsymbol{F}_{x}(p,q),\boldsymbol{F}_{y}(p,q))=\boldsymbol{F}(p,q)$. The formulation of Eq.~\eqref{operator-expression} implies that
	\begin{equation}
		\begin{split}
			\nabla_x \boldsymbol{F}_x(p,q)=&-(\lambda_2)^{-1}\left[\mathbb{E}_{q(y)}\left[\nabla_x f(x,y)\right]+2\lambda_1 x\right]\boldsymbol{F}_x(p,q),
			\\
			\nabla_y \boldsymbol{F}_y(p,q)=&-(\lambda_2)^{-1}\left[-\mathbb{E}_{p(x)}\left[\nabla_y f(x,y)\right]+2\lambda_1 y\right]\boldsymbol{F}_y(p,q).
		\end{split}
	\end{equation}
	According to Assumption 1 and Proposition \ref{estimation-operator}, we have 
	\begin{align}
		F_x(p,q)\leq&Z_q^{-1}\exp\{(\lambda_2)^{-1}\left[1-\lambda_1\|x\|^2\right]\}\leq\gamma_m\exp\left\{-(\lambda_2)^{-1}\lambda_1\|x\|^2\right\},\notag
		\\
		F_y(p,q)\leq&Z_p^{-1}\exp\{(\lambda_2)^{-1}\left[1-\lambda_1\|y\|^2\right]\}\leq\gamma_n\exp\left\{-(\lambda_2)^{-1}\lambda_1\|y\|^2\right\},\notag
	\end{align} 
	and
	\begin{align}
		\max\limits_{x\in\RR^m}\left\|\nabla_x\boldsymbol{F}_x(p,q)\right\|\leq&\gamma_m\left[(\lambda_2)^{-1}L_0+(2\lambda_1)^{1/2}(\lambda_2)^{-1/2}\right],\notag
		\\
		\max\limits_{y\in\RR^n}\left\|\nabla_y\boldsymbol{F}_y(p,q)\right\|\leq&\gamma_n\left[(\lambda_2)^{-1}L_0+(2\lambda_1)^{1/2}(\lambda_2)^{-1/2}\right],\notag
	\end{align}
	where $\gamma_m=\left(\frac{\lambda_1}{\lambda_2}\right)^{m/2}\pi^{-m/2}\exp\left\{2(\lambda_2)^{-1}\right\}$, $\gamma_n=\left(\frac{\lambda_1}{\lambda_2}\right)^{n/2}\pi^{-n/2}\exp\left\{2(\lambda_2)^{-1}\right\}$. Hence, we can choose 
	\begin{align}
		g_p(x):=&\gamma_m\exp\left\{-(\lambda_2)^{-1}\lambda_1\|x\|^2\right\},\notag
		\\
		g_q(y):=&\gamma_n\exp\left\{-(\lambda_2)^{-1}\lambda_1\|y\|^2\right\},\notag
	\end{align}
	and 
	\begin{align}
		L_{p,k}:=&\gamma_m\left[(\lambda_2)^{-1}L_0+(2\lambda_1)^{1/2}(\lambda_2)^{-1/2}\right],\notag
		\\ L_{q,k}:=&\gamma_n\left[(\lambda_2)^{-1}L_0+(2\lambda_1)^{1/2}(\lambda_2)^{-1/2}\right],\notag
	\end{align}
	for any $k\in\NN^+$. We have completed the proof w.r.t that $\boldsymbol{F}(p,q)$ is a continuous mapping (refer to Proposition \ref{F-continuity}) from the convex compact set $\mathcal{P}(g_p,\mathcal{L}_p,\RR^m)\times\mathcal{P}(g_q,\mathcal{L}_q,\RR^n)$ onto itself. Noticing that $\left(\mathcal{N}\left(0,(2\lambda_1)^{-1}\lambda_2\mathbf{I}_m\right), \mathcal{N}\left(0,(2\lambda_1)^{-1}\lambda_2\mathbf{I}_n\right)\right)$ belongs to $\mathcal{P}(g_p,\mathcal{L}_p,\RR^m)\times\mathcal{P}(g_q,\mathcal{L}_q,\RR^n)$, then we finish the proof by combining Brouwer fixed-point theorem.
	\end{proof}

        Proposition \ref{estimation-operator} aims to demonstrate that the probability density function $\boldsymbol{F}_x(p,q)$ and $\boldsymbol{F}_y(p,q)$ fall within functions family $\cB(\cL)$ which has been mentioned in Lemma \ref{auxiliary-lemma}. Proposition \ref{F-continuity} is proposed to establish the continuity of the operator $\boldsymbol{F}(p,q)$.
	\begin{proposition}\label{estimation-operator}
		We have
		\begin{align}
			\int_{\RR^d}\exp\{-\lambda\|x\|^2\}\mathrm{d}x&=(\lambda)^{-d/2}\pi^{d/2},\label{equation-gaussian}
			\\
			\max\limits_{x\in\RR^d}\|x\|\exp\{-\lambda\|x\|^2\}&\leq(2\lambda)^{-1/2}\exp\{-1/2\}.\label{max-norm}
		\end{align}
	\end{proposition}
	\begin{proof}
		Considering $z\in\RR^2$, we obtain
		\begin{align}
			\int_{\RR^2}\exp\{-\lambda\|z\|^2\}\mathrm{d}z&=2\pi\int_{0}^{\infty}r\exp\{-\lambda r^2\}\mathrm{d}r\notag
			\\
			&=(\lambda)^{-1}\pi.
		\end{align}
		Therefore, we derive Eq.~\eqref{equation-gaussian} by using the fact that
		\begin{align}
			\int_{\RR^d}\exp\{-\lambda\|x\|^2\}\mathrm{d}x=\left[\int_{\RR^2}\exp\{-\lambda\|z\|^2\}\mathrm{d}z\right]^{d/2}.
		\end{align}
	Notice that function $g(r)=r\exp\{-\lambda r^2\}$ for $r\geq0$ attains maximum when $r=(2\lambda)^{-1/2}$. Hence, Eq.~\eqref{max-norm} is deduced from $g(r)\leq(2\lambda)^{-1/2}\exp\{-1/2\}$ directly.
    \end{proof}
    \begin{proposition}\label{F-continuity}
        Assuming Assumption \ref{ass:a1} holds, the operator $\boldsymbol{F}(p,q)$ is $L^1$ continuous w.r.t $(p,q)\in\cP(\RR^m)\times\cP(\RR^n)$ under $\|\cdot\|_{L^1}$.
    \end{proposition}
    \begin{proof}
        Without loss of generality, it is sufficient to provide the continuity proof for one side of $p$.
        For simplicity, we denote $\hat{\boldsymbol{F}}_p(y)=\exp\left\{-(\lambda_2)^{-1}\left[-\EE_{p(x)}[f(x,y)]+\lambda_1\|y\|^2\right]\right\}$ and $\zeta(\lambda)=\exp\{3\lambda^{-1}\}\lambda^{-1}$. According to 
        \begin{align}\label{continuity-1}
            \|\hat{\boldsymbol{F}}_p-\hat{\boldsymbol{F}}_{p'}\|_{L^1}\overset{a}{\leq}&\int_{\RR^n}\zeta(\lambda_2)\left|\EE_{p(x)}[f(x,y)]-\EE_{p'(x)}[f(x,y)]\right|\exp\{-(\lambda_2)^{-1}\lambda_1\|y\|^2\}\mathrm{d}y
            \notag
            \\
            \leq&\zeta(\lambda_2)\|p-p'\|_{L^1}\int_{\RR^n}\exp\{-(\lambda_2)^{-1}\lambda_1\|y\|^2\}\mathrm{d}y,
        \end{align}
        where (a) follows from the fact that $|\exp\{x\}-\exp\{y\}|\leq\exp\{3c\}|x-y|$ for any $x,y\in[-c,c]$ and $c>0$, we have 
        \begin{align}\label{continuity-2}
            |Z_p-Z_{p'}|\leq\|\hat{\boldsymbol{F}}_p-\hat{\boldsymbol{F}}_{p'}\|_{L^1}\leq\zeta(\lambda_2)\|p-p'\|_{L^1}\int_{\RR^n}\exp\{-(\lambda_2)^{-1}\lambda_1\|y\|^2\}\mathrm{d}y.
        \end{align}
        Combining Eq.~\eqref{continuity-1} and Eq.~\eqref{continuity-2}, we obtain
        \begin{align}
            \left\|(Z_p)^{-1}\hat{\boldsymbol{F}}_p-(Z_{p'})^{-1}\hat{\boldsymbol{F}}_{p'}\right\|_{L^1}\leq&(Z_{p'})^{-1}\left[|Z_p-Z_{p'}|+\|\hat{\boldsymbol{F}}_p-\hat{\boldsymbol{F}}_{p'}\|_{L^1}\right]\notag
            \\
            \leq&\zeta(\lambda_2/2)\|p-p'\|_{L^1}.
        \end{align}
    \end{proof}

\section{Outer Loop Error Bound}

For simplicity, we define $w:=(p,q)$, $\Bar{w}:=(\Bar{p},\Bar{q})$, and
\begin{align}
    Q(\Bar{w},w):=&\lambda_1\mathbb{E}_{\Bar{p}}[\|x\|^{2}]+\lambda_2 \cE(\Bar{p})+\mathbb{E}_{(\Bar{p},q)}[f(x,y)]-\lambda_1\mathbb{E}_{q}[\|y\|^{2}]-\lambda_2 \cE(q)\notag
    \\
    &-[\lambda_1\mathbb{E}_{p}[\|x\|^{2}]+\lambda_2 \cE(p)+\mathbb{E}_{(p,\Bar{q})}[f(x,y)]-\lambda_1\mathbb{E}_{\Bar{q}}[\|y\|^{2}]-\lambda_2 \cE(\Bar{q})],
    \\
    \cQ_{x}(C):=&\left\{p(x)\in\mathcal{P}_2(\mathbb{R}^m)\left|p(x)\propto e^{-l_1(x)-\frac{\lambda_1}{\lambda_2}\|x\|^2}; \|l_1\|_\infty\leq \frac{C}{\lambda_2}\right.\right\},\label{define-auxi-function-1}\\
    \cQ_{y}(C):=&\left\{q(y)\in\cP_2(\mathbb{R}^n)\left|q(y)\propto e^{-l_2(y)-\frac{\lambda_1}{\lambda_2}\|y\|^2};\|l_2\|_\infty\leq \frac{C}{\lambda_2}\right.\right\},\label{define-auxi-function-2}
\end{align}
for some positive constant $C$. Moreover, we denote that 
	\begin{align}
		F_{\psi}^x(p):=&\int_{\cX}(\psi(x)+\lambda_1\|x\|^2)p(x)\ud x+\lambda_2\cE(p),\notag
		\\
		F_{\phi}^y(q):=&\int_{\cY}(\phi(y)+\lambda_1\|y\|^2)q(y)\ud y+\lambda_2\cE(q),\notag
	\end{align}
	for any $\psi:\cX\rightarrow\RR, \phi:\cY\rightarrow\RR$, 
	\begin{align}
		\rmBiLin_f(p,q):=\iint_{\cX\times\cY}p(x)f(x,y)q(y)\ud x\ud y,
	\end{align}
	for any $f:\cX\times\cY\rightarrow\RR$, $p:\cX\rightarrow\RR$ and $q:\cY\rightarrow\RR$. In addition, we introduce some elementary symbols for $p(x)\in\cP_2(\cX)$ satisfying LSI with constant $\alpha_{p}$ and $q(y)\in\cP_2(\cY)$ satisfying LSI with constant $\alpha_{q}$, and $\delta>0$:
	\begin{align}
		\rmErr_p^x(\delta):=&\left[\lambda_2+\frac{2\lambda_1}{\alpha_{p}}\right]\delta+\left[30+\frac{8\lambda_1\sigma(p)}{\sqrt{\alpha_{p}}}\right]\sqrt{\delta},\notag
		\\
		\rmErr_q^y(\delta):=&\left[\lambda_2+\frac{2\lambda_1}{\alpha_{q}}\right]\delta+\left[34+\frac{8\lambda_1\sigma(q)}{\sqrt{\alpha_{q}}}\right]\sqrt{\delta}.\notag
	\end{align}
    We say a density function $p$ satisfies log-Sobolev inequality with a constant $\alpha_p$ i.e. 
\begin{align}\label{LSI}
    \mathbb{E}_{p}[g^2\log(g^2)]-\mathbb{E}_{p}[g^2]\log\mathbb{E}_{p}[g^2]\leq\frac{2}{\alpha_p}\mathbb{E}_{p}[\|\nabla g\|^2]
\end{align}
In addition, we define $\sigma(p):=(\mathbb{E}_p[\|\theta\|^2])^{1/2}$ for any $p\in\mathcal{P}_2(\Pi)$.
Recalling that our algorithm updates $p_t$ and $q_t$ according to
\begin{align}
    q_t^*&=\arg\min_{q}\mathbb{E}_q[\hat{\phi}_t(y)]+\lambda_1\mathbb{E}_{q}[\|y\|^{2}]+\lambda_2\mathcal{E}(q)+\tau_t \mathrm{KL}(q\|q_{t-1}^*),\label{subproblem:q}
    \\ p_t^*&=\arg\min_{p}\mathbb{E}_p[\hat{\psi}_t(x)]+\lambda_1\mathbb{E}_{p}[\|x\|^{2}]+\lambda_2\mathcal{E}(p)+\eta_t \mathrm{KL}(p\|p_{t-1}^*),\label{subproblem:p}
\end{align}
where $\hat{\phi}_t(y)=-(1+\mu_t)\mathbb{E}_{\hat{p}_{t-1}}[f(x,y)]+\mu_t\mathbb{E}_{\hat{p}_{t-2}}[f(x,y)],\, \hat{\psi}_t(x)=\mathbb{E}_{\hat{q}_t}[f(x,y)]$, $p_t$ and $q_t$ are the approximate solutions of $p_t^*$ and $q_t^*$. 
The following Lemma summarizes the property of $p_t^*(x)$ and $q_t^*(y)$ at each step-$t$.
\begin{lemma}\label{estimation-log}
    Assume $\mu_t\leq1$ for all $t\geq0$, then $p_t^*(x)$ and $q_t^*(y)$ satisfy
    \begin{align}
        p_t^*(x)\in\mathcal{Q}_x(1),\, q_t^*(y)\in\mathcal{Q}_y(3).
    \end{align}
\end{lemma}
\begin{proof}
    Note that
    \begin{align}
        \|\hat{\phi}_t\|_\infty\leq 3,\, \, \, \, \|\hat{\psi}_t\|_\infty\leq1.
    \end{align}
    When $\mu_t\leq 1$ uniformly, we have
    \begin{align}
         h^{(t)}(y)=&\frac{1} {\lambda_2+\tau_t}\left\{\hat{\phi}_t(y)+\lambda_1\|y\|^{2}+\tau_th^{(t-1)}(y)\right\},\notag
         \\
         =&\frac{\lambda_2}{\lambda_2+\tau_t}\left\{\frac{1}{\lambda_2}\hat{\phi}_t(y)+\frac{\lambda_1}{\lambda_2}\|y\|^{2}\right\}+\frac{\tau_t}{\lambda_2+\tau_t}h^{(t-1)}(y),\label{explicit-1}\\
         g^{(t)}(x)=&\frac{1} {\lambda_2+\eta_t}\left\{\hat{\psi}_t(x)+\lambda_1\|x\|^{2}+\eta_t g^{(t-1)}(x)\right\},\notag
         \\
         =&\frac{\lambda_2}{\lambda_2+\eta_t}\left\{\frac{1}{\lambda_2}\hat{\psi}_t(x)+\frac{\lambda_1}{\lambda_2}\|x\|^{2}\right\}+\frac{\eta_t}{\lambda_2+\eta_t}g^{(t-1)}(x),\label{explicit-2}
         \\
         q_t^*(y)\propto& \exp\left(-h^{(t)}(y)\right),\quad 
         p_t^*(x)\propto \exp\left(-g^{(t)}(x)\right).\label{explicit-3}  
    \end{align}

    Suppose that $p_{t-1}^*(x)\in\mathcal{Q}_x(1), q_{t-1}^*(y)\in\mathcal{Q}_y(3)$ which indicates that $g^{t-1}(x)=-l_{1,t-1}(x)-\frac{\lambda_1}{\lambda_2}\|x\|^2$ with $\|l_{1,t-1}\|_{\infty}\leq\frac{1}{\lambda_2}$ and $h^{t-1}(y)=-l_{2,t-1}(y)-\frac{\lambda_1}{\lambda_2}\|y\|^2$ with $\|l_{2,t-1}\|_{\infty}\leq\frac{3}{\lambda_2}$, we obtain
    \begin{align}
        \left\|l_{1,t}\right\|_{\infty}&\leq\frac{\lambda_2}{\lambda_2+\eta_t}\left\|\frac{1}{\lambda_2}\hat{\psi}\right\|_{\infty}+\frac{\eta_t}{\lambda_2+\eta_t}\|l_{1,t-1}\|_{\infty}\leq\frac{1}{\lambda_2},\notag
        \\
        \left\|l_{2,t}\right\|_{\infty}&\leq\frac{\lambda_2}{\lambda_2+\eta_t}\left\|\frac{1}{\lambda_2}\hat{\phi}\right\|_{\infty}+\frac{\eta_t}{\lambda_2+\eta_t}\|l_{2,t-1}\|_{\infty}\leq\frac{1}{\lambda_2}.
    \end{align}
    Hence, the desired results can be obtained by combining the fact that $p_0^*(x)\in\mathcal{Q}_x(1),\, q_0^*(y)\in\mathcal{Q}_y(3)$ and mathematical induction.
\end{proof}
We use Lemma \ref{lemma-estimation-p-barp} to deal with the error bound between inexact solution $p_t$ and exact solution $p_t^*$ (as well as $q_t$ and $q_t^*$) in Theorem \ref{out-layer-thm}.
\begin{lemma}\label{lemma-estimation-p-barp}
    Assume a probability density $\Bar{p}(\theta)\propto\exp{(-l(\theta)-\frac{\lambda_1}{\lambda_2}\|\theta\|^2)}$ with bounded $l:\mathbb{R}^{d_\theta}\rightarrow\mathbb{R}$ which satisfies $\|l\|_{L^{\infty}}\leq C$, then for all probability density $p(\theta)\in\mathcal{P}_2(\mathbb{R}^{d_\theta})$ we have
    \begin{align}
        \mathcal{E}(p)-\mathcal{E}(\Bar{p})\leq&\mathrm{KL}(p\|\Bar{p})+\left(C+\frac{4\lambda_1\sigma(\Bar{p})}{\lambda_2\sqrt{\alpha_{\Bar{p}}}}\right)\sqrt{\mathrm{KL}(p\|\Bar{p})},\label{auxi-1}
        \\
        \int \|\theta\|^2(p-\Bar{p})(\theta)d\theta\leq&\frac{2}{\alpha_{\Bar{p}}}\mathrm{KL}(p\|\Bar{p})+\frac{4\sigma(\Bar{p})}{\sqrt{\alpha_{\hat{p}}}}\sqrt{\mathrm{KL}(p\|\Bar{p})}.\label{auxi-2}
    \end{align}
\end{lemma}
\begin{proof}
    According to the definition of $\mathcal{E}(\cdot)$, we obtain
    \begin{align}\label{gap-E}
        \mathcal{E}(p)-\mathcal{E}(\Bar{p})=&\mathrm{KL}(p\|\Bar{p})+\int\log(\Bar{p}(\theta))(p-\Bar{p})(\theta)d\theta\notag
        \\
        =&\mathrm{KL}(p\|\Bar{p})+\int\left(l(\theta)+\frac{\lambda_1}{\lambda_2}\|\theta\|^2\right)(\Bar{p}-p)(\theta)d\theta\notag
        \\
        \leq&\mathrm{KL}(p\|\Bar{p})+C\|\Bar{p}-p\|_{\mathrm{TV}}+\frac{\lambda_1}{\lambda_2}\int\|\theta\|^2(\Bar{p}-p)(\theta)d\theta.
    \end{align}
    Therefore, in order to complete the proof, we just need to bound $\left|\int \|\theta\|^2(p-\Bar{p})(\theta)d\theta\right|$. Letting $\rho(\theta,\theta')$ denote an optimal coupling between $p(\theta)$ and $\Bar{p}(\theta)$, we have
    \begin{align}\label{abs-2}
        \int \|\theta\|^2(\Bar{p}-p)(\theta)d\theta=&\int\|\theta'\|^2 \bar{p}(\theta')d\theta'-\int\|\theta\|^2 p(\theta)d\theta\notag
        \\
        =&\iint(\langle\theta'-\theta, \theta'+\theta\rangle)\rho(\theta,\theta')d\theta d\theta'\notag
        \\
        =&\iint(-\|\theta'-\theta\|^2+2\langle\theta'-\theta, \theta'\rangle)\rho(\theta,\theta')d\theta d\theta'\notag
        \\
        \leq&2W_2(p,\Bar{p})(\mathbb{E}_{\Bar{p}}[\|\theta\|^2])^{1/2},
    \end{align}
    where the last inequality in Eq.~\eqref{abs-2} follows from H\"{o}lder's inequality. Then, applying Eq.~\eqref{abs-2} on Eq.~\eqref{gap-E} and combining Talagrand's inequality \ref{Talagrand's-inequality}, Pinsker's inequality $\|\Bar{p}-p\|_{\mathrm{TV}}\leq\sqrt{\mathrm{KL}(p\|\Bar{p})/2}$, we conclude Eq.~\eqref{auxi-1}. On the other hand, we can obtain
    \begin{align}\label{abs-1}
        \int \|\theta\|^2(p-\Bar{p})(\theta)d\theta=&\int\|\theta\|^2 p(\theta)d\theta-\int\|\theta'\|^2 \bar{p}(\theta')d\theta'\notag
        \\
        =&\iint(\|\theta\|^2-\|\theta'\|^2)\rho(\theta,\theta')d\theta d\theta'\notag
        \\
        =&\iint(\|\theta-\theta'\|^2+2\langle\theta-\theta', \theta'\rangle)\rho(\theta,\theta')d\theta d\theta'\notag
        \\
        \leq&W_2^2(p,\Bar{p})+2W_2(p,\Bar{p})(\mathbb{E}_{\Bar{p}}[\|\theta\|^2])^{1/2},
    \end{align}
    where the last inequality in Eq.~\eqref{abs-1} follows from H\"{o}lder's inequality as well. Finally, we deduce Eq.~\eqref{auxi-2} by combining Eq.~\eqref{abs-1} and Talagrand's inequality \ref{Talagrand's-inequality}.
\end{proof}

Lemma \ref{lemma:optimal-KL-prox} describe the form of optimal solution in each iteration which inplies that if $p_{t-1}^*$ and $q_{t-1}^*$ are proximal Gibbs distribution, then $p_t^*$ and $q_t^*$ have the similar expression as well.
\begin{proof}[Proof of Lemma \ref{lemma:optimal-KL-prox}]
For simplicity, assume that
   \begin{align}
        F(p')\overset{\underset{\mathrm{def}}{}}{=}\mathbb{E}_{p'} [l(\theta)] +\lambda \mathcal{E}(p')+\tau \mathrm{KL}(p'\|\Bar{p}),
    \end{align}
    for any $p'\in\mathcal{P}(\Pi)$.
    
    Clearly, $p\in\mathcal{P}(\Pi)$ and $p(\theta)\propto(\bar{p}(\theta))^{\frac{\tau}{\lambda+\tau}}\exp\left\{-\frac{1}{\lambda+\tau}l(\theta)\right\}$. Hence, for any $p'\in\mathcal{P}(\Pi)$,
    \begin{align}\label{min-F(p)}
        F(p')=&\mathbb{E}_{p'}[l(\theta)]+\lambda \mathcal{E}(p')+\tau \mathrm{KL}(p'\|\Bar{p})\notag
        \\
        =&\mathbb{E}_{p}[l(\theta)]+\lambda \mathcal{E}(p)+\tau \mathrm{KL}(p\|\Bar{p}) +\int (p'-p)(\theta)[l(\theta)-\tau\log(\Bar{p}(\theta)]d\theta\notag
        \\
        &+(\lambda+\tau)[\mathcal{E}(p')-\mathcal{E}(p)]\notag
        \\
        =&F(p)+\int(p'-p)[-(\lambda+\tau)\log(p)](\theta)d\theta+(\lambda+\tau)[\mathcal{E}(p')-\mathcal{E}(p)]\notag
        \\
        =&F(p)+(\lambda+\tau)\mathrm{KL}(p'\|p)\geq F(p).
    \end{align}
\end{proof}
By choosing $l(\theta)= f(\theta)+\lambda_1 \Vert \theta\Vert^{2}$ and the probability density $\Bar{p}(\theta)\propto\exp\left(-\tilde{l}(\theta)\right)$, 
we obtain that the minimizer of  
$$\mathbb{E}_{p}[f(\theta)]+\lambda_1\mathbb{E}_{p}[\|\theta\|^{2}]+\lambda \cE(p)+\tau \mathrm{KL}(p\|\Bar{p}),$$
on $\mathcal{P}(\Pi)$ is $p\propto\exp(-\frac{1}{\lambda+\tau}f(\theta)-\frac{\tau}{\lambda+\tau}\tilde{l}(\theta)-\frac{\lambda_1}{\lambda+\tau}\|\theta\|^{2})$ .

Lemma \ref{lemma:subproblem-optimal} indicates the property of optimal solution in each step-$t$, which plays a crucial role in evaluating the outer loop error bound (in Theorem \ref{out-layer-thm}).
\begin{lemma}\label{lemma:subproblem-optimal}
    Assume that $l:\Pi\rightarrow\mathbb{R}$. If
    \begin{align}\label{argmin:linear_p}
        \hat{p}=\arg\min_{p}\{\mathbb{E}_p [l(\theta)]+\mu \mathcal{E}(p)+\mathrm{KL}(p\|\Tilde{p})\},
    \end{align}
    then 
    \begin{align}
        \mathbb{E}_{\hat{p}}[l(\theta)]+\mu \mathcal{E}(\hat{p})+\mathrm{KL}(\hat{p}\|\Tilde{p})+(\mu+1)\mathrm{KL}(p\|\hat{p})=\mathbb{E}_{p}[l(\theta)]+\mu \mathcal{E}(p)+\mathrm{KL}(p\|\Tilde{p}),\, \, \, \, \forall p\in\mathcal{P}.\notag 
    \end{align}
\end{lemma}
\begin{proof}
    The conclusion of this lemma can be derived directly from Eq.~\eqref{min-F(p)} by applying $\lambda=\mu$ and $\tau=1$.
\end{proof}

\begin{theorem}[Long Version of Theorem \ref{main-text-outer-loop-thm}]\label{out-layer-thm} 
	Assuming Assumptions \ref{ass:a1}, \ref{ass:a2} and \ref{ass:a4} hold, we set
		\begin{align}
			\gamma_t\mu_t&=\gamma_{t-1},\label{out-layer-1}\\
			\gamma_t\tau_t&\leq \gamma_{t-1}(\tau_{t-1}+\lambda_2),\label{out-layer-2}\\
			\gamma_t\eta_t&\leq\gamma_{t-1}(\eta_{t-1}+\lambda_2),\label{out-layer-3}\\
			\tau_t\eta_{t-1}&\geq\mu_t,\label{out-layer-4}\\
			\mu_t&\leq1,\label{out-layer-6}
		\end{align}
		and assume that subproblem error bounds at each step-$t$ satisfy
		\begin{align}
			\mathrm{KL}(p_t\|p_t^*)&\leq\delta_{t,1},\, \mathrm{KL}(q_t\|q_t^*)\leq\delta_{t,2},\label{out-layer-5}
		\end{align}
		then
		\begin{align}\label{out-layer-result1}
			\sum_{t=1}^k\gamma_t Q(w_t,w)\leq&\gamma_1\eta_1\mathrm{KL}(p\|p_0^*)-\gamma_k(\eta_k+\lambda_2)\mathrm{KL}(p\|p_k^*)+\gamma_1\tau_1\mathrm{KL}(q\|q_0^*)\notag
			\\
			&-\gamma_k\left(\tau_k+\lambda_2-\frac{1}{4\eta_k}\right)\mathrm{KL}(q\|q_k^*)\notag
			\\
			&+\sum_{t=1}^k\gamma_t\left[4(1+\mu_t)\epsilon+\rmErr_{p_t^*}^x(\delta_{t,1})+\rmErr_{q_t^*}(\delta_{t,2})\right],
		\end{align}
		for any $\epsilon>0$ and $p\in\cQ_{x}(C_1), q\in\cQ_{y}(C_2)$ with probability at least $1-{\delta}$ for sample size $M=\Theta(\epsilon^{-2}[\log(T\delta^{-1})+d\log(1+L_0r_{\max\{C_1,C_2\}}(\epsilon,\lambda_1,\lambda_2)\epsilon^{-1})])$ in each iteration where $d=\max\{m,n\}$  and $r_{C}(\epsilon,\lambda_1,\lambda_2):=8\sqrt{\frac{\lambda_2+C}{\lambda_1}\max\{\log(\epsilon^{-1}), m, n\}}$. Moreover, we have
		\begin{align}
			\gamma_k\left(\tau_k+\lambda_2-\frac{1}{4\eta_k}\right)\mathrm{KL}(q_*\|q_k^*)\leq&\gamma_1\eta_1\mathrm{KL}(p_*\|p_0^*)+\gamma_1\tau_1\mathrm{KL}(q_*\|q_0^*)\notag
			\\
			&+\sum_{t=1}^k\gamma_t\left[4(1+\mu_t)\epsilon+\rmErr_{p_t^*}(\delta_{t,1})+\rmErr_{q_t^*}(\delta_{t,2})\right],\label{out-layer-result2}
		\end{align}
		and
		\begin{align}
			\gamma_k\left(\eta_k-\frac{1}{4(\tau_k+\lambda_2)}\right)\mathrm{KL}(p_k^*\|p_{k-1}^*)\leq&\gamma_1\eta_1\mathrm{KL}(p_*\|p_0^*)+\gamma_1\tau_1\mathrm{KL}(q_*\|q_0^*)\notag
			\\
			&+\sum_{t=1}^k\gamma_t\left[4(1+\mu_t)\epsilon+\rmErr_{p_t^*}(\delta_{t,1})+\rmErr_{q_t^*}(\delta_{t,2})\right],\label{out-layer-result3}
		\end{align}
		with respect to the MNE $w_*=(p_*, q_*)$ of Eq.~\eqref{eq:problem_particle}.
	\end{theorem}
	\begin{proof}
		By applying Lemma \ref{lemma:subproblem-optimal} to Eq.~\eqref{subproblem:q} and Eq.~\eqref{subproblem:p}, we can obtain the following equalities
		\begin{align}        
			&F_{\hat{\phi}_t}^y(q_t^*)-F_{\hat{\phi}_t}^y(q)=\tau_t[\mathrm{KL}(q\|q_{t-1}^*)-\mathrm{KL}(q_t^*\|q_{t-1}^*)]-(\tau_t+\lambda_2)\mathrm{KL}(q\|q_t^*),\label{optimal-y}
			\\
			&F_{\hat{\psi}_t}^x(p_t^*)-F_{\hat{\psi}_t}^x(p)=\eta_t[\mathrm{KL}(p\|p_{t-1}^*)-\mathrm{KL}(p_t^*\|p_{t-1}^*)]-(\eta_t+\lambda_2)\mathrm{KL}(p\|p_t^*).\label{optimal-x}
		\end{align}
		According to Eq.~\eqref{optimal-y}, we have
		\begin{align}\label{estimation-phi-1}
			F_{\hat{\phi}_t}^y(q_t)-F_{\hat{\phi}_t}^y(q)=&F_{\hat{\phi}_t}^y(q_t^*)-F_{\hat{\phi}_t}^y(q)-\left[F_{\hat{\phi}_t}^y(q_t^*)-F_{\hat{\phi}_t}^y(q_t)\right]\notag
			\\
			=&\tau_t[\mathrm{KL}(q\|q_{t-1}^*)-\mathrm{KL}(q_t^*\|q_{t-1}^*)]-(\tau_t+\lambda_2)\mathrm{KL}(q\|q_t^*)\notag
			\\
			&+\underbrace{\lambda_1\int\|y\|^{2}(q_t-q_t^*)(y)dy+\lambda_2(\mathcal{E}(q_t)-\mathcal{E}(q_t^*))}_{\mathrm{a}}\notag
			\\
			&+\underbrace{\int\hat{\phi}_t(y)(q_t-q_t^*)(y)dy}_{\mathrm{b}}.
		\end{align}
		It is worth noting that we can derive the following inequality from Eq.~\eqref{out-layer-6}
		\begin{align}\label{bound-1}
			\int\hat{\phi}_t(y)(q_t-q_t^*)(y)dy\leq 3\|q_t^*-q_t\|_{\mathrm{TV}}\leq3\sqrt{\mathrm{KL}(q_t\|q_t^*)}.
		\end{align}
		Next, we apply Lemma \ref{lemma-estimation-p-barp} to (a) and bound Eq.~\eqref{bound-1} in (b) of Eq.~\eqref{estimation-phi-1}. This allows us to propose the estimation as follows
		\begin{align}\label{estimation-phi-2}
			F_{\hat{\phi}_t}^y(q_t)-F_{\hat{\phi}_t}^y(q)\leq&\tau_t[\mathrm{KL}(q\|q_{t-1}^*)-\mathrm{KL}(q_t^*\|q_{t-1}^*)]-(\tau_t+\lambda_2)\mathrm{KL}(q\|q_t^*)\notag
			\\
			&+\left[\lambda_2+\frac{2\lambda_1}{\alpha_{q_t^*}}\right]\delta_{t,2}+\left[6+\frac{8\lambda_1\sigma(q_t^*)}{\sqrt{\alpha_{q_t^*}}}\right]\sqrt{\delta_{t,2}}.
		\end{align}
		Likewise, we can deduce that
		\begin{align}\label{estimation-psi-1}
			F_{\hat{\psi}_t}^x(p_t^*)-F_{\hat{\psi}_t}^x(p)\leq&\eta_t[\mathrm{KL}(p\|p_{t-1}^*)-\mathrm{KL}(p_t^*\|p_{t-1}^*)]-(\eta_t+\lambda_2)\mathrm{KL}(p\|p_t^*)\notag
			\\
			&+\left[\lambda_2+\frac{2\lambda_1}{\alpha_{p_t^*}}\right]\delta_{t,1}+\left[2+\frac{8\lambda_1\sigma(p_t^*)}{\sqrt{\alpha_{p_t^*}}}\right]\sqrt{\delta_{t,1}}.
		\end{align}
		By utilizing the boundedness of $f$ and carefully selecting constants $C_1$, $C_2$, we can adjust the number of sampling points for both $\hat{p}_t$ and $\hat{q}_t$ to obtain the following result
		\begin{align}\label{finite-particle-condition}
			\left|\rmBiLin_f(p_t-\hat{p}_t,q)\right|\leq\epsilon,\, \, \, \, 
			\left|\rmBiLin_f(p,q_t-\hat{q}_t)\right|\leq\epsilon,
		\end{align}
		with probability $1-\delta$ for all $q(y)\in\cQ_\cY(C_2)$, $p(x)\in\cQ_\cX(C_1)$ uniformly (Refer to Lemma \ref{finite-particle-approximation} in Appendix \ref{discre}). In order to deal with finite particles, we denote 
		$\phi_t(y)=-(1+\mu_t)\mathbb{E}_{p_{t-1}}[f(x,y)]+\mu_t\mathbb{E}_{p_{t-2}}[f(x,y)]$, $ \psi_t(x)=\mathbb{E}_{q_t}[f(x,y)]$. Therefore, summing up these inequalities in Eq.~\eqref{estimation-phi-2}, Eq.~\eqref{estimation-psi-1} and using the definition of the Max-Min gap function $Q$, we have
		\begin{align}\label{relation-1}
			&Q(w_t,w)+\rmBiLin_f((p_t-p_{t-1})-\mu_t(p_{t-1}-p_{t-2}),q_t-q)\notag
			\\
			\overset{\mathrm{c}}{=}&F_{\phi_t}^y(q_t)-F_{\phi_t}^y(q)+F_{\psi_t}^x(p_t)-F_{\psi_t}^x(p)
			\notag
			\\
			=&\underbrace{\int(\phi_t-\hat{\phi}_t)(y)(q_t^*-q)(y)dy}_{\text{bounded by Eq.~\eqref{finite-particle-condition}}}+\underbrace{\int(\phi_t-\hat{\phi}_t)(y)(q_t-q_t^*)(y)dy}_{\text{bounded by Pinsker's inequality}}\notag
			\\
			&+\underbrace{\int(\psi_t-\hat{\psi}_t)(x)(p_t^*-p)(x)dx}_{\text{bounded by Eq.~\eqref{finite-particle-condition}}}+\underbrace{\int(\psi_t-\hat{\psi}_t)(x)(p_t-p_t^*)(x)dx}_{\text{bounded by Pinsker's inequality}}
			\notag
			\\
			&+F_{\hat{\phi}_t}^y(q_t)-F_{\hat{\phi}_t}^y(q)+F_{\hat{\psi}_t}^x(p_t)-F_{\hat{\psi}_t}^x(p)
			\notag
			\\
			\overset{\mathrm{d}}{\leq}&4(1+\mu_t)\epsilon+\tau_t[\mathrm{KL}(q\|q_{t-1}^*)-\mathrm{KL}(q_t^*\|q_{t-1}^*)]-(\tau_t+\lambda_2)\mathrm{KL}(q\|q_t^*)\notag
			\\
			&+\eta_t[\mathrm{KL}(p\|p_{t-1}^*)-\mathrm{KL}(p_t^*\|p_{t-1}^*)]-(\eta_t+\lambda_2)\mathrm{KL}(p\|p_t^*)+\left[\lambda_2+\frac{2\lambda_1}{\alpha_{p_t^*}}\right]\delta_{t,1}\notag
			\\
			&+\left[6+\frac{8\lambda_1\sigma(p_t^*)}{\sqrt{\alpha_{p_t^*}}}\right]\sqrt{\delta_{t,1}}+\left[\lambda_2+\frac{2\lambda_1}{\alpha_{q_t^*}}\right]\delta_{t,2}+\left[18+\frac{8\lambda_1\sigma(q_t^*)}{\sqrt{\alpha_{q_t^*}}}\right]\sqrt{\delta_{t,2}},
		\end{align}
		where equality (c) follows from the expression of $Q$ and inequality (d) is deduced from plugging Eq.~\eqref{estimation-phi-2} into  Eq.~\eqref{estimation-psi-1}.
		
		Notice that
		\begin{align}\label{cross-term}
			&\rmBiLin_f((p_t-p_{t-1})-\mu_t(p_{t-1}-p_{t-2}),q_t-q)
			\notag
			\\
			=&\rmBiLin_f(p_t-p_{t-1},q_t-q)-\mu_t\rmBiLin_f(p_{t-1}-p_{t-2},q_{t-1}-q)\notag
			\\
			&+\mu_t\rmBiLin_f(p_{t-1}-p_{t-2},q_{t-1}-q_{t}).
		\end{align}
		
		Moreover, we can derive
		\begin{align}\label{relation-2}
			&\rmBiLin_f\left((p_t^*-p_{t-1}^*)-\mu_t(p_{t-1}^*-p_{t-2}^*),q_t^*-q\right)\notag
			\\
			&-\rmBiLin_f\left((p_t-p_{t-1})-\mu_t(p_{t-1}-p_{t-2}),q_t-q\right)
			\notag
			\\
            \leq&8\left[\sqrt{\delta_{t,1}}+\sqrt{\delta_{t-1,1}}+\sqrt{\delta_{t-2,1}}+\sqrt{\delta_{t,2}}\right],
		\end{align}
		where Eq.~\eqref{relation-2} follows from
		\begin{align}\label{estimation-cross-item-1}
			&\rmBiLin_f\left((p_t^*-p_{t-1}^*)-\mu_t(p_{t-1}^*-p_{t-2}^*),q_t^*-q\right)\notag
			\\
			&-\rmBiLin_f((p_t-p_{t-1})-\mu_t(p_{t-1}-p_{t-2}),q_t-q)
			\notag
			\\
			=&\rmBiLin_f\left((p_t^*-p_t)-(1+\mu_t)(p_{t-1}^*-p_{t-1})+\mu_t(p_{t-2}^*-p_{t-2}),q_t^*-q\right)\notag
			\\
			&+\rmBiLin_f((p_t-p_{t-1})-\mu_t(p_{t-1}-p_{t-2}),q_t^*-q_t)
			\notag
			\\
			\overset{\mathrm{e}}{\leq}&2[\|p_t^*-p_t\|_{\mathrm{TV}}+(1+\mu_t)\|p_{t-1}^*-p_{t-1}\|_{\mathrm{TV}}\notag
			\\
			&+\mu_t\|p_{t-2}^*-p_{t-2}\|_{\mathrm{TV}}+(1+\mu_t)\|q_t^*-q_t\|_{\mathrm{TV}}],
		\end{align}
		and Inequality (e) can be deduced from the boundness of $\int f(x,y)(q_t^*-q)(y)dy$, $\int f(x,y)(p_t-p_{t-1})(x)dx$ and $\int f(x,y)(p_{t-1}-p_{t-2})(x)dx$. 

		By applying $p_{t-2}^*, p_{t-1}^*, p_t^*, q_{t-1}^*, q_t^*$ to Eq.~\eqref{cross-term}, and then combining the two aforementioned relations Eq.~\eqref{relation-1} and Eq.~\eqref{relation-2}, we can multiply both sides of the resulting inquality by $\gamma_t\geq 0$ to obtain
		\begin{align}\label{ineq-above}
			\sum_{t=1}^k\gamma_t Q(w_t,w)\leq&\sum_{t=1}^k\gamma_t[\eta_t \mathrm{KL}(p\|p_{t-1}^*)-(\eta_t+\lambda_2)\mathrm{KL}(p\|p_t^*)]\notag
			\\
			&+\sum_{t=1}^k\gamma_t[\tau_t \mathrm{KL}(q\|q_{t-1}^*)-(\tau_t+\lambda_2) \mathrm{KL}(q\|q_t^*)]+4\sum_{t=1}^k\gamma_t (1+\mu_t)\epsilon\notag
			\\
			&+\sum_{t=1}^k\gamma_t\left[\mu_t\rmBiLin_f(p_{t-1}^*-p_{t-2}^*, q_{t-1}^*-q)-\rmBiLin_f(p_t^*-p_{t-1}^*, q_t^*-q)\right]\notag
			\\
			&-\sum_{t=1}^k\gamma_t\left[\tau_t \mathrm{KL}(q_t^*\|q_{t-1}^*)+\eta_t \mathrm{KL}(p_t^*\|p_{t-1}^*)+\mu_t\rmBiLin_f(p_{t-1}^*-p_{t-2}^*, q_{t-1}^*-q_{t}^*)\right]\notag
			\\
			&+\sum_{t=1}^k\gamma_t\left[\rmErr_{p_t^*}^x(\delta_{t,1})+\rmErr_{q_t^*}^y(\delta_{t,2})\right].
		\end{align}
		Taking into account Eq.~\eqref{out-layer-1}-\eqref{out-layer-3} and the fact that $p_0^*=p_{-1}^*$ and $q_0^*=q_{-1}^*$, the above inequality Eq.~\eqref{ineq-above} implies that
		\begin{align}\label{main-estimate}
			\sum_{t=1}^k\gamma_t Q(w_t,w)\leq& \gamma_1\eta_1 \mathrm{KL}(p\|p_0^*)-\gamma_k(\eta_k+\lambda_2)\mathrm{KL}(p\|p_k^*)\notag
			\\
			&+\gamma_1\tau_1 \mathrm{KL}(q\|q_0^*)-\gamma_k(\tau_k+\lambda_2) \mathrm{KL}(q\|q_k^*)+4\sum_{t=1}^k\gamma_t (1+\mu_t)\epsilon\notag
			\\
			&-\sum_{t=1}^k \gamma_t\left[\tau_t \mathrm{KL}(q_t^*\|q_{t-1}^*)+\eta_t \mathrm{KL}(p_t^*\|p_{t-1}^*)+\mu_t\rmBiLin_f(p_{t-1}^*-p_{t-2}^*, q_{t-1}^*-q_{t}^*)\right]\notag
			\\
			&-\gamma_k\rmBiLin_f(p_k^*-p_{k-1}^*, q_k^*-q)+\sum_{t=1}^k\gamma_t\left[\rmErr_{p_t^*}^x(\delta_{t,1})+\rmErr_{q_t^*}^y(\delta_{t,2})\right].
		\end{align}
		By Eq.~\eqref{out-layer-1} and Eq.~\eqref{out-layer-4}, we have
		\begin{align}\label{part-estimate}
			&-\sum_{t=1}^k \gamma_t\left[\tau_t \mathrm{KL}(q_t^*\|q_{t-1}^*)+\eta_t \mathrm{KL}(p_t^*\|p_{t-1}^*)+\mu_t\rmBiLin_f(p_{t-1}^*-p_{t-2}^*, q_{t-1}^*-q_{t}^*)\right]\notag
			\\
			\overset{\mathrm{f}}{\leq}&-\sum_{t=2}^k\left[\gamma_t\tau_t\mathrm{KL}(q_t^*\|q_{t-1}^*)+\gamma_{t-1}\eta_{t-1}\mathrm{KL}(p_{t-1}^*\|p_{t-2}^*)\right.\notag
			\\
			&\left.-\gamma_t\mu_t\|p_{t-1}^*-p_{t-2}^*\|_{\mathrm{TV}}\|q_{t-1}^*-q_{t}^*\|_{\mathrm{TV}}\right]-\gamma_k\eta_k\mathrm{KL}(p_k^*\|p_{k-1}^*)\notag
			\\
			\overset{\mathrm{g}}{\leq}&-\gamma_k\eta_k\mathrm{KL}(p_k^*\|p_{k-1}^*),
		\end{align}
		where the
		inequality (f) follows from
		\begin{align}
			\rmBiLin_f(p_{t-1}^*-p_{t-2}^*, q_{t-1}^*-q_{t}^*)\leq\|p_{t-1}^*-p_{t-2}^*\|_{\mathrm{TV}}\|q_{t-1}^*-q_{t}^*\|_{\mathrm{TV}},
		\end{align}
		and the
		inequality (g) is
		derived from the Pinsker's inequality and the parameter settings in Eq.~\eqref{out-layer-1} and Eq.~\eqref{out-layer-4}.
		
		Combining the two inequalities Eq.~\eqref{main-estimate} and Eq.~\eqref{part-estimate}, we obtain
		\begin{align}\label{estimation-combination-Q}
			\sum_{t=1}^k\gamma_t Q(w_t,w)\leq&\gamma_1\eta_1 \mathrm{KL}(p\|p_0^*)-\gamma_k(\eta_k+\lambda_2)\mathrm{KL}(p\|p_k^*)+\gamma_1\tau_1 \mathrm{KL}(q\|q_0^*)\notag
			\\
			&-\gamma_k(\tau_k+\lambda_2) \mathrm{KL}(q\|q_k^*)-\gamma_k\eta_k\mathrm{KL}(p_k^*\|p_{k-1}^*)-\gamma_k\rmBiLin_f(p_k^*-p_{k-1}^*, q_k^*-q)\notag
			\\
			&+4\sum_{t=1}^k\gamma_t (1+\mu_t)\epsilon+\sum_{t=1}^k\gamma_t\left[\rmErr_{p_t^*}^x(\delta_{t,1})+\rmErr_{q_t^*}^y(\delta_{t,2})\right].
		\end{align}
		The result in Eq.~\eqref{out-layer-result1} then follows from the above inequality \eqref{estimation-combination-Q} and the fact that by Eq.~\eqref{out-layer-4},
		\begin{align}\label{estimation-pfq}
			&-(\tau_k+\lambda_2)\mathrm{KL}(q\|q_k^*)-\eta_k\mathrm{KL}(p_k^*\|p_{k-1}^*)-\rmBiLin_f(p_k^*-p_{k-1}^*, q_k^*-q)\notag
			\\
			\leq&-(\tau_k+\lambda_2)\mathrm{KL}(q\|q_k^*)-\eta_k\mathrm{KL}(p_k^*\|p_{k-1}^*)+\sqrt{\mathrm{KL}(p_k^*\|p_{k-1}^*)}\sqrt{\mathrm{KL}(q\|q_k^*)}\notag
			\\
			\leq&-\left(\tau_k+\lambda_2-\frac{1}{4\eta_k}\right)\mathrm{KL}(q\|q_k^*).
		\end{align}
		Fixing $w=w_*$ in the above inequality \eqref{estimation-combination-Q} and combining Eq.~\eqref{estimation-pfq}, we obtain the result in  Eq.~\eqref{out-layer-result2}. 
		
		Finally, Eq.~\eqref{out-layer-result3} is concluded from similar ideas,
		\begin{align}
			&-(\tau_k+\lambda_2)\mathrm{KL}(q\|q_k^*)-\eta_k\mathrm{KL}(p_k^*\|p_{k-1}^*)-\rmBiLin_f(p_k^*-p_{k-1}^*, q_k^*-q)\notag
			\\
			\leq&-(\tau_k+\lambda_2)\mathrm{KL}(q\|q_k^*)-\eta_k\mathrm{KL}(p_k^*\|p_{k-1}^*)+\sqrt{\mathrm{KL}(p_k^*\|p_{k-1}^*)}\sqrt{\mathrm{KL}(q\|q_k^*)}\notag
			\\
			\leq&-\left(\eta_k-\frac{1}{4(\tau_k+\lambda_2)}\right)\mathrm{KL}(p_k^*\|p_{k-1}^*).
		\end{align}
	\end{proof}

Based on Theorem \ref{out-layer-thm}, we will provide specific instantiation methods for parameters $\{\tau_t\}_{t=1}^k,\{\eta_t\}_{t=1}^k,\{\gamma_t\}_{t=1}^k$. Additionally, leveraging Theorem \ref{main-text-thm-ULA}, we can control the error bounds $\{\delta_{t,1}\}_{t=1}^k,\{\delta_{t,2}\}_{t=1}^k$ (refer to Theorem \ref{out-layer-thm}) of sub-problems. By merging these two aspects, we can achieve a globally convergent result for solving the minimax problem with accelerated convergence rates. Our goal is to establish the convergence of the Max-Min gap, as evaluated at the output solution
\begin{align}
    \Bar{w}_k=\frac{\sum_{t=1}^k\gamma_t w_t}{\sum_{t=1}^k\gamma_t},
\end{align}
will ultimately converge to zero. In addition, We will establish the last iterate convergence of the PAPAL algorithm by considering the KL distance between $w_k^*$ and $w_*$ in KL divergence, as well as the $W_2$ distance between $w_k$ and $w_*$.

For Corollary \ref{main-text-outer-loop-complexity},  we have the KL distance estimation:
    \begin{align}\label{KL-divergence-convergence-1}
        \mathrm{KL}(p_*\|p_T^*)+\frac{1}{2}\mathrm{KL}(q_*\|q_T^*)
        \leq&\mu^T[\mathrm{KL}(q_*\|q_0^*)+\mathrm{KL}(p_*\|p_0^*)]\notag
        \\
        &+\lambda_2^{-1}\left[8\epsilon+(10\lambda_1+\lambda_2+32)T^{-J}\right],
    \end{align}
    and the Wasserstein distance estimation:
    \begin{align}
    \label{Wassenstain-convergence-1}
        \frac{1}{2}W_2^2(q_T, q_*)+W_2^2(p_T, p_*)
        \leq&\frac{4}{\min\{\alpha_{p_T^*}, \alpha_{q_T^*}\}}\left\{\mu^{T}\left[\mathrm{KL}(q_*\|q_0^*)+\mathrm{KL}(p_*\|p_0^*)\right]\right.\notag
        \\
        &+\left.\lambda_2^{-1}\left[8\epsilon+(10\lambda_1+\lambda_2+32)T^{-J}\right]\right\}+3T^{-2J},
    \end{align}
    with respect to the MNE $w_*=(p_*,q_*)$ of Eq.~\eqref{eq:problem_particle}.
\begin{proof}[Proof of Corollary \ref{main-text-outer-loop-complexity}]
    According to the sub-problem error bound, we have
    \begin{equation}
        \frac{2\lambda_1}{\alpha_{p_t^*}}\delta_{t,1}\leq\lambda_1 T^{-2J},\, \, \frac{8\lambda_1\sigma(p_t^*)}{\sqrt{\alpha_{p_t^*}}}\sqrt{\delta_{t,1}}\leq 4\lambda_1 T^{-J},
        \, \, 
        \frac{2\lambda_1}{\alpha_{q_t^*}}\delta_{t,2}\leq\lambda_1 T^{-2J},\, \, \frac{8\lambda_1\sigma(q_t^*)}{\sqrt{\alpha_{q_t^*}}}\sqrt{\delta_{t,2}}\leq 4\lambda_1 T^{-J}.\nonumber
    \end{equation}
    Notice that $Q(\Bar{w},w)$ is convex with respect to $\Bar{w}$ given a fixing $w$. 
    
    According to Jensen's inequality, we have
    \begin{align}
        Q(\Bar{w}_T,w)\leq\frac{\sum_{t=1}^T\gamma_t Q(w_t,w)}{\sum_{t=1}^T\gamma_t},
    \end{align}
    for any $w\in\mathcal{P}_2(\mathcal{X})\times\mathcal{P}_2(\mathcal{Y})$. Moreover, under problem setting in Eq.~\ref{eq:problem_particle} and the optimality condition, we can derive that $p_*\in\cQ_{x}(C), q_*\in\cQ_{y}(C)$ when $w_*$ exists, and $\Bar{p}^*\in\cQ_{x}(C), \Bar{q}^*\in\cQ_{y}(C)$ for any $\Bar{w}=(\Bar{p},\Bar{q})\in\cP(\cX)\times\cP(\cY)$ with $C=1$ (see coefficient $C$ in \eqref{define-auxi-function-1} and \eqref{define-auxi-function-2}).
    
    Therefore, Eq.~\eqref{Max-Min-gap} is derived from Eq.~\eqref{out-layer-result1}(with $w=\bar{w}_T^*$),
    \begin{align}
        0\leq Q(\Bar{w}_T,\bar{w}_T^*)\leq\frac{\sum_{t=1}^T\gamma_t Q(w_T,\bar{w}_T^*)}{\sum_{t=1}^T\gamma_t}\leq&\mu^T\left[\gamma_1\eta_1\mathrm{KL}(p\|p_0^*)+\gamma_1\tau_1\mathrm{KL}(q\|q_0^*)\right]\notag
        \\
        &+\frac{\sum_{t=1}^T\gamma_t}{\sum_{t=1}^T\gamma_t}\left[8\epsilon+(10\lambda_1+\lambda_2+32)T^{-J}\right].
    \end{align}
    In addition, we have
    \begin{align}
        2\eta_t(\tau_t+\lambda_2)-1=2\lambda_2^2\frac{\mu}{(1-\mu)^2}=1, \notag
    \end{align}
    which leads to
    \begin{align}
        \frac{\tau_t+\lambda_2}{2}-\frac{1}{4\eta_t}>0.
    \end{align}
    The result in Eq.~\eqref{KL-divergence-convergence} then follows from Eq.~\eqref{out-layer-result1} 
    (with $w=w_*$) and $Q(w_t,w_*)\geq0$ that
    \begin{align}
        0\leq\sum_{t=1}^T\gamma_t Q(w_t,w_*)\leq&\gamma_1\eta_1\mathrm{KL}(p_*\|p_0^*)-\gamma_T(\eta_T+\lambda_2)\mathrm{KL}(p_*\|p_T^*)+\gamma_1\tau_1\mathrm{KL}(q_*\|q_0^*)\notag\\
        &-\frac{\gamma_T}{2}(\tau_T+\lambda_2)\mathrm{KL}(q_*\|q_T^*)\notag
        \\
        &+\sum_{t=1}^T\gamma^t\left[8\epsilon+(10\lambda_1+\lambda_2+32)T^{-J}\right].
    \end{align}
    Finally, by combining Eq.~\eqref{KL-divergence-convergence}, Talagrand's inequality, and 
    \begin{align}
        W_2^2(\nu, \rho)\leq 2W_2^2(\nu, p)+2W_2^2(p, \rho),\, \, \forall \, \nu, p, \rho\in\mathcal{P}_2(\Pi),
    \end{align}
    we obtain the convergence of Wasserstein distance.
\end{proof}

\section{Global Convergence}
\begin{proof}[Proof of Theorem \ref{main-text-global-convergence}]
    Noting the explicit form of $p_t^*(x)$ and $q_t^*(y)$ in Lemma \ref{estimation-log}, we obtain that $p_t^*$ satisfies log-Sobolev inequality with $\lambda_1\exp\{-4\lambda_2^{-1}\}\lambda_2^{-1}$ and $q_t^*$ satisfies log-Sobolev inequality with $\lambda_1\exp\{-12\lambda_2^{-1}\}\lambda_2^{-1}$ for any $t\geq0$ by combining Lemma \ref{log-sobolev-exp-bounded}. 
    
    In fact, we also have 
    \begin{align}\label{estimation-square}
        \mathbb{E}_{\Bar{p}}[\|\theta\|^2]=&\int\|\theta\|^2\frac{\exp{[-l(\theta)-(\lambda_1/\lambda_2)\|\theta\|^2]}}{\int\exp{[-l(\theta)-(\lambda_1/\lambda_2)\|\theta\|^2]}}d\theta\notag
        \\
        \leq&\exp{(2C)}\int\|\theta\|^2\frac{\exp{[-(\lambda_1/\lambda_2)\|\theta\|^2]}}{\int\exp{[-(\lambda_1/\lambda_2)\|\theta\|^2]}}d\theta\notag
        \\
        =&\frac{\lambda_2 d_\theta}{2\lambda_1}\exp{(2C)},
    \end{align}
    if probability density $\Bar{p}(\theta)\propto\exp{\left(-l(\theta)-\frac{\lambda_1}{\lambda_2}\|\theta\|^2\right)}$ in $\mathbb{R}^{d_\theta}$ with $\|l\|_\infty\leq C$. 
    
    Therefore, we have $(\sigma(p_t^*))^2\leq\frac{\lambda_2m}{2\lambda_1}\exp{\left(\frac{2}{\lambda_2}\right)}$ and $(\sigma(q_t^*))^2\leq\frac{\lambda_2n}{2\lambda_1}\exp{\left(\frac{6}{\lambda_2}\right)}$. According to Theorem \ref{ULA-error-bound}, we set the step size $\iota$ of ULA as
    \begin{align}
    \iota\leq\frac{\lambda_1\lambda_2 T^{-2J}}{32 (3L_1+\lambda_1)^2\max\{m,n\}}\min\left\{\frac{1}{2\exp{\left(\frac{12}{\lambda_2}\right)}},\, \, \frac{\lambda_1}{\lambda_2\exp{\left(\frac{24}{\lambda_2}\right)}},\, \, \frac{\lambda_1^2}{\lambda_2^2\max\{m,n\}\exp{\left(\frac{30}{\lambda_2}\right)}}\right\},\nonumber
    \end{align}
    and sub-problem iterations $T_t$ at each step $t$ 
    {\footnotesize
    \begin{align}
        T_t\geq \iota^{-1}\frac{\lambda_2e^{\frac{12}{\lambda_2}}}{\lambda_1}\left[\max\left\{3,\, \, \frac{24}{\lambda_2}+\log\left(\frac{\lambda_1}{\lambda_2}\right)+2,\, \, \frac{30}{\lambda_2}+\log(\max\{m,n\})+3\log\left(\frac{\lambda_1}{\lambda_2}\right)+1\right\}+2J\log(T)+C\right],
        \nonumber
    \end{align}
    }
    to guarantee that $\max\{\delta_{t,1}, \delta_{t,2}\}\leq\frac{1}{2}\min\{c(p_t^*), c(q_t^*)\}T^{-2J}$ for any $t\geq 1$.

    Under Corollary \ref{main-text-outer-loop-complexity}, we complete the proof.
\end{proof}
\textbf{Background on MALA and the proximal sampler: }
For simplicity, we use $\mathcal{N}$ denote Gaussian distribution. Considering $\pi$ be a target distribution on $\mathbb{R}^{d_{\theta}}$, and denoting $Q$ as a markov transition kernel, we have the Metropolis-Adjusted Langevin algorithm (Algorithm \ref{alg:mala}) with proposal kernel $Q$. In general, the target distribution $\pi\propto\exp\{-f\}$ and the proposal kernel is taken to be one step of the discretized Langevin algorithm, which is formulated as following:
\begin{align}
    Q(\theta,\cdot)=\mathcal{N}(\theta-\eta\nabla f(\theta), 2\eta\mathbf{I}_{d_{\theta}}).
\end{align}
Next, we introduce the proximal sampler (Algorithm \ref{alg:proximal}) which has been mentioned in \citep{altschuler2024faster}. The third step of the Algorithm \ref{alg:proximal} approximates the objective distribution $\pi(\theta|y_k)\propto\exp\{-f(\theta)-\frac{1}{2\eta}\|\theta-y_k\|^2\}$ using MALA with high accuracy.

\begin{algorithm}[h]
\caption{Metropolis-Adjusted Langevin Algorithm (MALA)} 
\label{alg:mala}
{\small
{\bf Input:} 
 number of outer iterations $T$, initial point $\theta_0$ from a starting distribution $\pi_0$, step size $\eta$ 
 \\
{\bf Output:} sequence of samples $\theta_1,\cdots,\theta_T$
\begin{algorithmic}[1]
    \FOR{$t=0$ to $T$}
        \STATE Propose $y_t\sim Q(\theta_t,\cdot)$.
        \STATE Compute the acceptance probability 
        $\alpha_t=\min\left\{1, \frac{\pi(y_t)Q(y_t,\theta_t)}{\pi(\theta_t)Q(\theta_t,y_t)}\right\}$.
        \STATE Draw $u\sim\mathrm{Unif}[0,1]$.
        \IF{$u<\alpha_k$}
        \STATE Accept the proposal: $\theta_{t+1}\leftarrow y_t$.
        \ELSE
        \STATE Reject the proposal: $\theta_{t+1}\leftarrow \theta_t$.
        \ENDIF
    \ENDFOR
\end{algorithmic}
}
\end{algorithm}

\begin{algorithm}[h]
\caption{Proximal Sampler with MALA} 
\label{alg:proximal}
{\small
{\bf Input:} 
 number of outer iterations $T$, initial point $\theta_0$ and step size $\eta$ 
 \\
{\bf Output:} sequence of samples $\theta_1,\cdots,\theta_T$
\begin{algorithmic}[1]
    \FOR{$t=0$ to $T$}
        \STATE Propose $y_t\sim\mathcal{N}(\theta_t, \eta\mathbf{I}_{d_{\theta}})$.
        \STATE Sample $\theta_{t+1}\sim\pi(\theta|y_k)\propto\exp\{-f(\theta)-\frac{1}{2\eta}\|\theta-y_k\|^2\}$ using Algorithm \ref{alg:mala}.
    \ENDFOR
\end{algorithmic}
}
\end{algorithm}

\begin{proof}[Proof of Corollary \ref{remark-MALA}]
    According to the explicit expressions of $p_t^*$ and $q_t^*$, we have that $h^{(t)}(\cdot)$ has $(3L_1+\lambda_1)\lambda_2^{-1}$-Lipschitz gradient and $g^{(t)}(\cdot)$ has $(L_1+\lambda_1)\lambda_2^{-1}$-Lipschitz gradient for any $t\geq0$. Therefore, combining \citep[Theorem 5.3]{altschuler2024faster} and recalling that $p_t^*\propto\exp\{-g^{(t)}(x)\}$ satisfies log-Sobolev inequality with $\lambda_1\exp\{-4\lambda_2^{-1}\}\lambda_2^{-1}$ and $q_t^*\propto\exp\{-h^{(t)}(y)\}$ satisfies log-Sobolev inequality with $\lambda_1\exp\{-12\lambda_2^{-1}\}\lambda_2^{-1}$ for any $t\geq0$ in the proof of Theorem \ref{main-text-global-convergence}, we set sub-problem iterations $T_t$ at each step $t$
    \begin{align}
        T_t\geq\tilde{\mathcal{O}}\left(\frac{(3L_1+\lambda_1)e^{\frac{12}{\lambda_2}}}{\lambda_1}(\max\{m,n\})^{1/2}\left(2J\log(T)+1+\lambda_2^{-1}\right)^4\right),
    \end{align}
    to guarantee that $\max\{\delta_{t,1}, \delta_{t,2}\}\leq\frac{1}{2}\min\{c(p_t^*), c(q_t^*)\}T^{-2J}$ for any $t\geq 1$.
\end{proof}


\section{Stochastic Version}
	
	
	For simplicity, we denote $\psi_{\xi}(x):=\psi(x,\xi), \phi_{\xi}(y):=\phi(y,\xi)$ for any random variable $\xi$ and functions $\psi:\cX\times\Lambda\rightarrow\RR, \phi:\cY\times\Lambda\rightarrow\RR$. Considering $N$ i.i.d random variables $(\xi_1,\cdots,\xi_n)$ with $\xi_i=\xi$ for any $i\in[1:N]$, we let $f_{\bar{\xi}_N}(\theta):=\frac{1}{N}\sum_{i=1}^{N}f(\theta,\xi_i)$ for any $f:\Pi\times\Lambda\rightarrow\RR$. In addition, we introduce some elementary symbols for $p(x)\in\cP_2(\cX)$ satisfying LSI with constant $\alpha_{p}$ and $q(y)\in\cP_2(\cY)$ satisfying LSI with constant $\alpha_{q}$, and $\delta>0$:
	\begin{align}
		\rmExpectErr_p^x(\delta):=&\left[\lambda_2+\frac{2\lambda_1}{\alpha_{p}}\right]\delta+\left[32+\frac{8\lambda_1\sigma(p)}{\sqrt{\alpha_{p}}}\right]\sqrt{\delta},\notag
		\\
		\rmExpectErr_q^y(\delta):=&\left[\lambda_2+\frac{2\lambda_1}{\alpha_{q}}\right]\delta+\left[40+\frac{8\lambda_1\sigma(q)}{\sqrt{\alpha_{q}}}\right]\sqrt{\delta},\notag
	\end{align}
	We also let $\rmStoErr_{\phi_N}^y(q):=\left|F_{\phi_{\bar{\xi}_N}}^y(q)-F_{\EE_{\xi}\phi}^y(q)\right|$ and $\rmStoErr_{\psi_N}^x(p):=\left|F_{\psi_{\bar{\xi}_N}}^x(p)-F_{\EE_{\xi}\psi}^x(p)\right|$. To account for stochastic version, a modification of the functions $\hat{\phi}_t$ and $\hat{\psi}_t$ in the PAPAL algorithm iteration step for deterministic version is required. The PAPAL algorithm for stochastic version can be obtained by performing stochastic sampling on functions $\hat{\phi}_t$ and $\hat{\psi}_t$ (i.e. substituting $\hat{\phi}_t$ with $(\hat{\phi}_t)_{\Bar{\xi}_{N,t}}:=-(1+\mu_t)\EE_{\hat{p}_{t-1}}\left[G_{\Bar{\xi}_{N,t}}(\cdot,y)\right]+\mu_t\EE_{\hat{p}_{t-2}}\left[G_{\Bar{\xi}_{N,t}}(\cdot,y)\right]$ and $\hat{\psi}_t$ with $(\hat{\psi}_t)_{\Bar{\xi}_{N,t}}:=\EE_{\hat{q}_{t}}\left[G_{\Bar{\xi}_{N,t}}(x,\cdot)\right]$) for each $t\in[1:T]$ in Algorithm \ref{alg:algorithm}. In this context, $\Bar{\xi}_{N,t}$ denotes the sampling of $N$ i.i.d. random variables $(\xi_1,\cdots,\xi_N)$ associated with the $t$-th iteration. In the statement of the following lemma, we use function $(\hat{\phi}_t)_{\Bar{\xi}_N}$ to denote $(\hat{\phi}_t)_{\Bar{\xi}_{N,t}}$, and function $(\hat{\psi}_t)_{\Bar{\xi}_N}$ to denote $(\hat{\psi}_t)_{\Bar{\xi}_{N,t}}$ for consistency.
	
	\begin{lemma}\label{stocastic-out-layer-thm}
		We set
		\begin{align}
			\gamma_t\mu_t&=\gamma_{t-1},\label{sto-out-layer-1}\\
			\gamma_t\tau_t&\leq \gamma_{t-1}(\tau_{t-1}+\lambda_2),\label{sto-out-layer-2}\\
			\gamma_t\eta_t&\leq\gamma_{t-1}(\eta_{t-1}+\lambda_2),\label{sto-out-layer-3}\\
			\tau_t\eta_{t-1}&\geq\mu_t,\label{sto-out-layer-4}\\
			\mu_t&\leq1,\label{sto-out-layer-6}
		\end{align}
		and assume that subproblem error bounds at each step-$t$ satisfy
		\begin{align}
			\mathrm{KL}(p_t\|p_t^*)&\leq\delta_{t,1},\, \mathrm{KL}(q_t\|q_t^*)\leq\delta_{t,2},\label{sto-out-layer-5}
		\end{align}
		then
		\begin{align}
			\sum_{t=1}^k\gamma_t Q(w_t,w)\leq&\gamma_1\eta_{1}\mathrm{KL}(p\|p_0^*)-\gamma_k(\eta_k+\lambda_2)\mathrm{KL}(p\|p_k^*)+\gamma_1\tau_1\mathrm{KL}(q\|q_0^*)\notag
			\\
			&-\gamma_k\left(\tau_k+\lambda_2-\frac{1}{4\eta_k}\right)\mathrm{KL}(q\|q_k^*)\notag
			\\    &+\sum_{t=1}^k\gamma_t\left[4(1+\mu_t)\epsilon+\rmExpectErr_{p_t^*}^x(\delta_{t,1})+\rmExpectErr_{q_t^*}^y(\delta_{t,2})\right]
			\notag
			\\
			&+\sum_{t=1}^k\gamma_t\left[\rmStoErr_{(\hat{\phi}_t)_n}^y(q_t^*)+\rmStoErr_{(\hat{\phi}_t)_n}^y(q)+\rmStoErr_{(\hat{\psi}_t)_n}^x(p_t^*)+\rmStoErr_{(\hat{\psi}_t)_n}^x(p)\right],
		\end{align}
	for any $\epsilon>0$ and $p\in\cQ_{x}(C_1), q\in\cQ_{y}(C_2)$ with probability at least $1-{\delta}$ for sample size $M=\Theta(\epsilon^{-2}[\log(T\delta^{-1})+d\log(1+L_0\sqrt{\lambda_2/\lambda_1}\epsilon^{-1})])$ in each iteration where $d=\max\{m,n\}$ . Moreover, we have
	\begin{align}
		&\gamma_k\left(\tau_k+\lambda_2-\frac{1}{4\eta_k}\right)\mathrm{KL}(q_*\|q_k^*)\notag
        \\
        \leq&\gamma_1\eta_1\mathrm{KL}(p_*\|p_0^*)+\gamma_1\tau_1\mathrm{KL}(q_*\|q_0^*)+\sum_{t=1}^k\gamma_t\left[4(1+\mu_t)\epsilon+\rmExpectErr_{p_t^*}(\delta_{t,1})+\rmExpectErr_{q_t^*}(\delta_{t,2})\right]\notag
		\\
		&+\sum_{t=1}^k\gamma_t\left[\rmStoErr_{(\hat{\phi}_t)_n}^y(q_t^*)+\rmStoErr_{(\hat{\phi}_t)_n}^y(q_*)+\rmStoErr_{(\hat{\psi}_t)_n}^x(p_t^*)+\rmStoErr_{(\hat{\psi}_t)_n}^x(p_*)\right],\label{sto-out-layer-result2}
	\end{align}
	and
	\begin{align}
		&\gamma_k\left(\eta_k-\frac{1}{4(\tau_k+\lambda_2)}\right)\mathrm{KL}(p_k^*\|p_{k-1}^*)\notag
        \\
        \leq&\gamma_1\eta_1\mathrm{KL}(p_*\|p_0^*)+\gamma_1\tau_1\mathrm{KL}(q_*\|q_0^*)+\sum_{t=1}^k\gamma_t\left[4(1+\mu_t)\epsilon+\rmExpectErr_{p_t^*}(\delta_{t,1})+\rmExpectErr_{q_t^*}(\delta_{t,2})\right]\notag
		\\
		&+\sum_{t=1}^k\gamma_t\left[\rmStoErr_{(\hat{\phi}_t)_n}^y(q_t^*)+\rmStoErr_{(\hat{\phi}_t)_n}^y(q_*)+\rmStoErr_{(\hat{\psi}_t)_n}^x(p_t^*)+\rmStoErr_{(\hat{\psi}_t)_n}^x(p_*)\right]\label{sto-out-layer-result3}
	\end{align}
	with respect to the MNE $w_*=(p_*, q_*)$ of Eq.~\eqref{eq:sto-finite-particle-problem}.
	\end{lemma}
	\begin{proof}
		According to optimality condition (Lemma \ref{lemma:subproblem-optimal}), we have
		\begin{align}
			F_{(\hat{\phi}_t)_{\bar{\xi}_n}}^y(q_t^*)-F_{(\hat{\phi}_t)_{\bar{\xi}_n}}^y(q)=&\tau_t[\mathrm{KL}(q\|q_{t-1}^*)-\mathrm{KL}(q_t^*\|q_{t-1}^*)]-(\tau_t+\lambda_2)\mathrm{KL}(q\|q_t^*),\label{optimal-y-stocas}
			\\
			F_{(\hat{\psi}_t)_{\bar{\xi}_n}}^y(p_t^*)-F_{(\hat{\psi}_t)_{\bar{\xi}_n}}^y(p)=&\eta_t[\mathrm{KL}(p\|p_{t-1}^*)-\mathrm{KL}(p_t^*\|p_{t-1}^*)]-(\eta_t+\lambda_2)\mathrm{KL}(p\|p_t^*).\label{optimal-x-stocas}
		\end{align}
	By using Eq.~\eqref{optimal-y-stocas}, we have
	\begin{align}
		F_{\EE_{\xi}\hat{\phi}_t}^y(q_t)-F_{\EE_{\xi}\hat{\phi}_t}^y(q)=&F_{\EE_{\xi}\hat{\phi}_t}^y(q_t)-F_{(\hat{\phi}_t)_{\bar{\xi}_n}}^y(q_t)+F_{(\hat{\phi}_t)_{\bar{\xi}_n}}^y(q)-F_{\EE_{\xi}\hat{\phi}_t}^y(q)\notag
		\\
		&+\underbrace{F_{(\hat{\phi}_t)_{\bar{\xi}_n}}^y(q_t^*)-F_{(\hat{\phi}_t)_{\bar{\xi}_n}}^y(q)+F_{(\hat{\phi}_t)_{\bar{\xi}_n}}^y(q_t)-F_{(\hat{\phi}_t)_{\bar{\xi}_n}}^y(q_t^*)}_{\text{a}}\notag
		\\
		\underset{\text{b}}{\leq}&\rmStoErr_{(\hat{\phi}_t)_n}^y(q_t^*)+\rmStoErr_{(\hat{\phi}_t)_n}^y(q)+\tau_t\left[\mathrm{KL}(q\|q_{t-1}^*)-\mathrm{KL}(q_t^*||q_{t-1}^*)\right]\notag
		\\
		&-(\tau_t+\lambda_2)\mathrm{KL}(q\|q_t^*)+\left[\lambda_2+\frac{2\lambda_1}{\alpha_{q_t^*}}\right]\delta_{t,2}+\left[12+\frac{8\lambda_1\sigma(q_t^*)}{\sqrt{\alpha_{q_t^*}}}\right]\sqrt{\delta_{t,2}},
	\end{align}
	where (b) is derived from applying Eq.\eqref{estimation-phi-2} to (a) and the estimation that
	$$\left|\rmStoErr_{(\hat{\phi}_t)_n}^y(q_t^*)-\rmStoErr_{(\hat{\phi}_t)_n}^y(q_t)\right|\leq 2\sqrt{\delta_{t,2}},
	$$ 
	which follows from Pinsker's inequality directly. 
	Likewise, we can deduce that
	\begin{align}
		F_{\EE_{\xi}\hat{\psi}_t}^x(p_t)-F_{\EE_{\xi}\hat{\psi}_t}^x(p)\leq&\rmStoErr_{(\hat{\psi}_t)_n}^x(p_t^*)+\rmStoErr_{(\hat{\psi}_t)_n}^x(p)+\eta_t\left[\mathrm{KL}(p\|p_{t-1}^*)-\mathrm{KL}(p_t^*\|p_{t-1}^*)\right]\notag
		\\
		&-(\eta_t+\lambda_2)\mathrm{KL}(p\|p_t^*)+\left[\lambda_2+\frac{2\lambda_1}{\alpha_{p_t^*}}\right]\delta_{t,1}+\left[4+\frac{8\lambda_1\sigma(p_t^*)}{\sqrt{\alpha_{p_t^*}}}\right]\sqrt{\delta_{t,1}}.
	\end{align}
	By utilizing the boundedness of $f$ and carefully selecting constants $C_1$, $C_2$, we can adjust the number of sampling points for both $\hat{p}_t$ and $\hat{q}_t$ to obtain the following result
	\begin{align}\label{sto-finite-particle-condition}
		\left|\rmBiLinSf(p_t-\hat{p}_t,q)\right|\leq\epsilon,\, \, \, \, 
		\left|\rmBiLinSf(p,q_t-\hat{q}_t)\right|\leq\epsilon,
	\end{align}
	with probability $1-\delta$ for all $q(y)\in\cQ_\cY(C_2)$, $p(x)\in\cQ_\cX(C_1)$ uniformly (Refer to Lemma \ref{finite-particle-approximation} in Appendix \ref{discre}). In order to deal with finite particles, we denote $(\EE_{\xi}\phi_t)(y)=-(1+\mu_t)\EE_{p_{t-1}}\left[\EE_{\xi}f(x,y,\xi)\right]+\mu_t\EE_{p_{t-2}}\left[\EE_{\xi}f(x,y,\xi)\right]$, $(\EE_{\xi}\psi_t)(x)=\EE_{q_t}\left[\EE_{\xi}f(x,y,\xi)\right]$. Similar to the proof of Theorem \ref{out-layer-thm}, we leverage the definition of the Max-Min gap function $Q$ to arrive at
	\begin{align}
		\sum_{t=1}^k\gamma_tQ(w_t,w)\leq&\sum_{t=1}^k\gamma_t\left[\eta_t\mathrm{KL}(p\|p_{t-1}^*)-(\eta_t+\lambda_2)\mathrm{KL}(p||p_t^*)\right]\notag
		\\
		&+\sum_{t=1}^k\gamma_t\left[\tau_t\mathrm{KL}(q\|q_{t-1}^*)-(\tau_t+\lambda_2)\mathrm{KL}(q\|q_t^*)\right]+4\sum_{t=1}^k\gamma_t(1+\mu_t)\epsilon\notag
		\\
		&+\sum_{t=1}^k\gamma_t\left[\mu_t\rmBiLinSf(p_{t-1}^*-p_{t-2}^*, q_{t-1}^*-q)-\rmBiLinSf(p_{t}^*-p_{t-1}^*, q_{t}^*-q)\right]\notag
		\\
		&-\sum_{t=1}^k\gamma_t\left[\tau_t\mathrm{KL}(q_t^*\|q_{t-1}^*)+\eta_t\mathrm{KL}(p_t^*\|p_{t-1}^*)+\mu_t\rmBiLinSf(p_{t-1}^*-p_{t-2}^*, q_{t-1}^*-q_t^*)\right]\notag
		\\
		&+\sum_{t=1}^k\gamma_t\left[\rmStoErr_{(\hat{\phi}_t)_n}^y(q_t^*)+\rmStoErr_{(\hat{\phi}_t)_n}^y(q)+\rmStoErr_{(\hat{\psi}_t)_n}^x(p_t^*)+\rmStoErr_{(\hat{\psi}_t)_n}^x(p)\right]\notag
		\\
		&+\sum_{t=1}^{k}\gamma_t\left[\rmExpectErr_{p_t^*}^x(\delta_{t,1})+\rmExpectErr_{q_t^*}^x(\delta_{t,2})\right].
	\end{align}
	Therefore, it's direct to obtain 
	\begin{align}
		\sum_{t=1}^k\gamma_t Q(w_t,w)\leq&\gamma_1\eta_1\mathrm{KL}(p\|p_0^*)-\gamma_k(\eta_k+\lambda_2)\mathrm{KL}(p\|p_k^*)+\gamma_1\tau_1\mathrm{KL}(q\|q_0^*)\notag
		\\
		&-\gamma_k(\tau_k+\lambda_2)\mathrm{KL}(q\|q_k^*)-\gamma_k\eta_k\mathrm{KL}(p_k^*\|p_{k-1}^*)-\gamma_k\rmBiLinSf(p_k^*-p_{k-1}^*,q_k^*-q)\notag
		\\
		&+\sum_{t=1}^k\gamma_t\left[\rmStoErr_{(\hat{\phi}_t)_n}^y(q_t^*)+\rmStoErr_{(\hat{\phi}_t)_n}^y(q)+\rmStoErr_{(\hat{\psi}_t)_n}^x(p_t^*)+\rmStoErr_{(\hat{\psi}_t)_n}^x(p)\right]\notag
		\\
		&+\sum_{t=1}^k\gamma_t\left[4(1+\mu_t)\epsilon+\rmExpectErr_{p_t^*}^x(\delta_{t,1})+\rmExpectErr_{q_t^*}^x(\delta_{t,2})\right].
	\end{align}
	In light of the prior analysis, we can complete the proof of our lemma by following the same steps outlined in Theorem \ref{out-layer-thm}. 
	\end{proof}

    By utilizing Assumptions \ref{Sto-Assumptions1}, \ref{Sto-Assumptions2} and \ref{Sto-Assumptions3}, and applying $\hat{\phi}_t(y,\xi)$ and $\hat{\psi}_t(x,\xi)$ to Lemma \ref{finite-particle-approximation}, we can derive that the discretization error of finite particles on $\Lambda$ satisfies
    \begin{align}
        \rmStoErr_{(\hat{\phi}_t)_N}^y(q)\leq\epsilon,\, \rmStoErr_{(\hat{\psi}_t)_N}^x(p)\leq\epsilon,
    \end{align}
    with probability $1-\delta$ for any $p\in\cQ_x(C_1)$ and $q\in\cQ_y(C_2)$, where the required number of particles is
    \begin{align}
        N=\cO\left(\epsilon^{-2}\left[-\log(\delta)+d\log\left(1+L_0r_{\max\{C_1,C_2\}}(\epsilon,\lambda_1,\lambda_2)\epsilon^{-1}\right)\right]\right),
    \end{align}
    and $r_C(\epsilon,\lambda_1,\lambda_2):=8\sqrt{\frac{\lambda_2+C}{\lambda_1}\max\{\log(\epsilon^{-1}), m, n\}}$. Therefore, combining Lemma \ref{stocastic-out-layer-thm} and aforementioned discretization error of finite particles, we obtain stochastic outer loop convergence result Lemma \ref{Stochastic-Outer-Loop-Lemma}. Applying inner loop error bound to Lemma \ref{Stochastic-Outer-Loop-Lemma}, we can derive the global convergence result Theorem \ref{sto-main-text-global-convergence} and Corollary \ref{Sto-col:kl_conv} w.r.t stochastic PAPAL.

\section{Discretization Error of Finite Particles}\label{discre}
\begin{proof}[Proof of Lemma \ref{finite-particle-approximation}] 
    In order to improve the overall clarity of our argument, we concentrate on establishing the estimation in Eq.~\eqref{concentration-1}. We split the proof into two parts.
    \paragraph{Part I}\, Combing Lemma \ref{Tail-bound}, and the definition of $\cQ_y(C_2)$, we have
    \begin{align}
        \left|\iint_{\RR^m\times\{y|\|y\|\geq r_y\}}(\hat{p}_t-p_t)(x)f(x,y)q(y)dydx\right|\leq& 2\exp{\left(\frac{2C_2}{\lambda_2}\right)}\mathbb{P}[\|\hat{y}\|^2\geq r_y^2]\notag
        \\
        \leq&2\exp{\left(\frac{2C_2}{\lambda_2}\right)}\exp{\left(-\frac{\lambda_1 r_y^2}{10\lambda_2}\right)},
    \end{align}
    for any $q\in\cQ_y(C_2)$ with $r_y^2\geq\frac{n\lambda_2}{\lambda_1}$ where $\hat{y}\sim\cN\left(0, \frac{\lambda_2}{2\lambda_1}I_n\right)$. Then, we derive the bound: $$\left|\iint_{\RR^m\times\{y|\|y\|\geq r_y\}}(\hat{p}_t-p_t)(x)f(x,y)q(y)dydx\right|\leq\frac{1}{4}\epsilon$$ for any $q\in\cQ_y(C_2)$ by setting $r_y\geq 4\sqrt{\frac{\lambda_2}{\lambda_1}\max\{\log(\epsilon^{-1}), n\}+2\frac{\lambda_2}{\lambda_1}+2\frac{C_2}{\lambda_1}}$. 
    
    \paragraph{Part II.} We need to consider the uniform bound between $\hat{p}_t$ and $p_t$ when $y\in\BB_{r_y}(\mathbf{0})$ with $r_y=8\sqrt{\frac{\lambda_2+C_2}{\lambda_1}\max\{\log(\epsilon^{-1}), n\}}$. According to Lemma 5.7 in \cite{wainwright2019high}, we have the $\Delta$-covering number (i.e. $N(\Delta;\BB_{r_y}(\mathbf{0}),\|\cdot\|)$) of $\BB_{r_y}(\mathbf{0}):=\left\{y\in\mathbb{R}^n\mid\|y\|\leq r_y\right\}$ in the $\|\cdot\|$-norm obeys the bound
    \begin{align}\label{Covering-number}
        N(\Delta;\BB_{r_y}(\mathbf{0}),\|\cdot\|)\leq \left(1+\frac{2r_y}{\Delta}\right)^n.
    \end{align}
    On the other hand, we obtain
    \begin{align}\label{fixed-point-concentration}
        \mathbb{P}\left[\left|\int_{\RR^m}(\hat{p}_t-p_t)(x)f(x,y)dx\right|>\rho\right]\leq 2\exp{\left(-\frac{\rho^2 M}{2}\right)},
    \end{align}
    under Lemma \ref{Hoeffding's-inequality} when fixing $y\in\BB_{r_y}(\mathbf{0})$. Therefore, we can deduce the concentration inequality on a $\Delta$-cover $\mathcal{C}_{\BB_{r_y}(\mathbf{0}),\Delta}$ of $\BB_{r_y}(\mathbf{0})$
    \begin{align}\label{delta-cover-concentration}
        \mathbb{P}\left[\max\limits_{y\in\mathcal{C}_{\BB_{r_y}(\mathbf{0}),\Delta}}\left|\int_{\RR^m}(\hat{p}_t-p_t)(x)f(x,y)dx\right|>\rho\right]\leq 2\left(1+\frac{2r_y}{\Delta}\right)^n\exp{\left(-\frac{\rho^2 M}{2}\right)},
    \end{align}
    by combining Eq.~\eqref{Covering-number} and Eq.~\eqref{fixed-point-concentration}. Then, the following estimation
    \begin{align}
        \mathbb{P}\left[\max\limits_{y\in\BB_{r_y}(\bf{0})}\left|\int_{\RR^m}(\hat{p}_t-p_t)(x)f(x,y)dx\right|>\rho+L_0\Delta\right]\leq 2\left(1+\frac{2r_y}{\Delta}\right)^n\exp{\left(-\frac{\rho^2 M}{2}\right)},
    \end{align}
    follows from Assumption \ref{ass:a2}, Eq.~\eqref{delta-cover-concentration} and Lemma. 

    Finally, we conclude the result in \eqref{concentration-1} by using \textbf{Part I} and \textbf{Part II} with setting $\Delta=\frac{\epsilon}{8L_0}$, $\rho=\frac{\epsilon}{8}$ and $M\geq128\left[-\log(\delta/2)+n\log\left(1+\frac{16r_yL_0}{\epsilon}\right)\right]\epsilon^{-2}$ and the result in \eqref{concentration-1} can be derived under the similar argument as well.
\end{proof}

\section{Empirical Results}\label{sec:empi}

In this section, we would like to illustrate the motivation of our algorithm compared to weight-driven algorithm -- WFR-DA
\citep{domingo2020mean}. Thus, the theoretical guarantees for particle-based algorithm is important in real practice.

\paragraph{Non-Convex Optimization.} Our algorithm can be recognized as a generalization of non-convex optimization. When the $f(x,y)$ is a constant with respect to $y$, it is a minimization problem.
Our first example is a non-convex minimization task, which
provides us insights to distinguish the optimization on weights and the positions. Figure \ref{fig:w_vs_p} demonstrates that the weight-driven algorithms only ``select" particles close to the solution, while particle-based algorithms move the particles to find the solution. The difference makes weight-driven algorithms need more particles to make sure that the initialized some particles can approximate the solution. Unless the sample size is extremely large, the weight-driven algorithm cannot obtain good solution. In fact, considering the volume of space, the sample size of weight-driven algorithms is exponential with respect to the dimension, which makes the original algorithm (with theoretical guarantees) in \citep{domingo2020mean} hardly usable in high dimensions.
It is worth mentioning that the combination of weight optimization and the particle-based optimization might be better heuristically, but the leading driven force of the dynamics should be the particle-based one. The theoretical analysis for the combined algorithm mainly driven by particle-based force
would be a good future work.

\begin{figure*}[t]
    \centering
        \begin{tabular}{cccc}
  \includegraphics[width=0.45\textwidth]{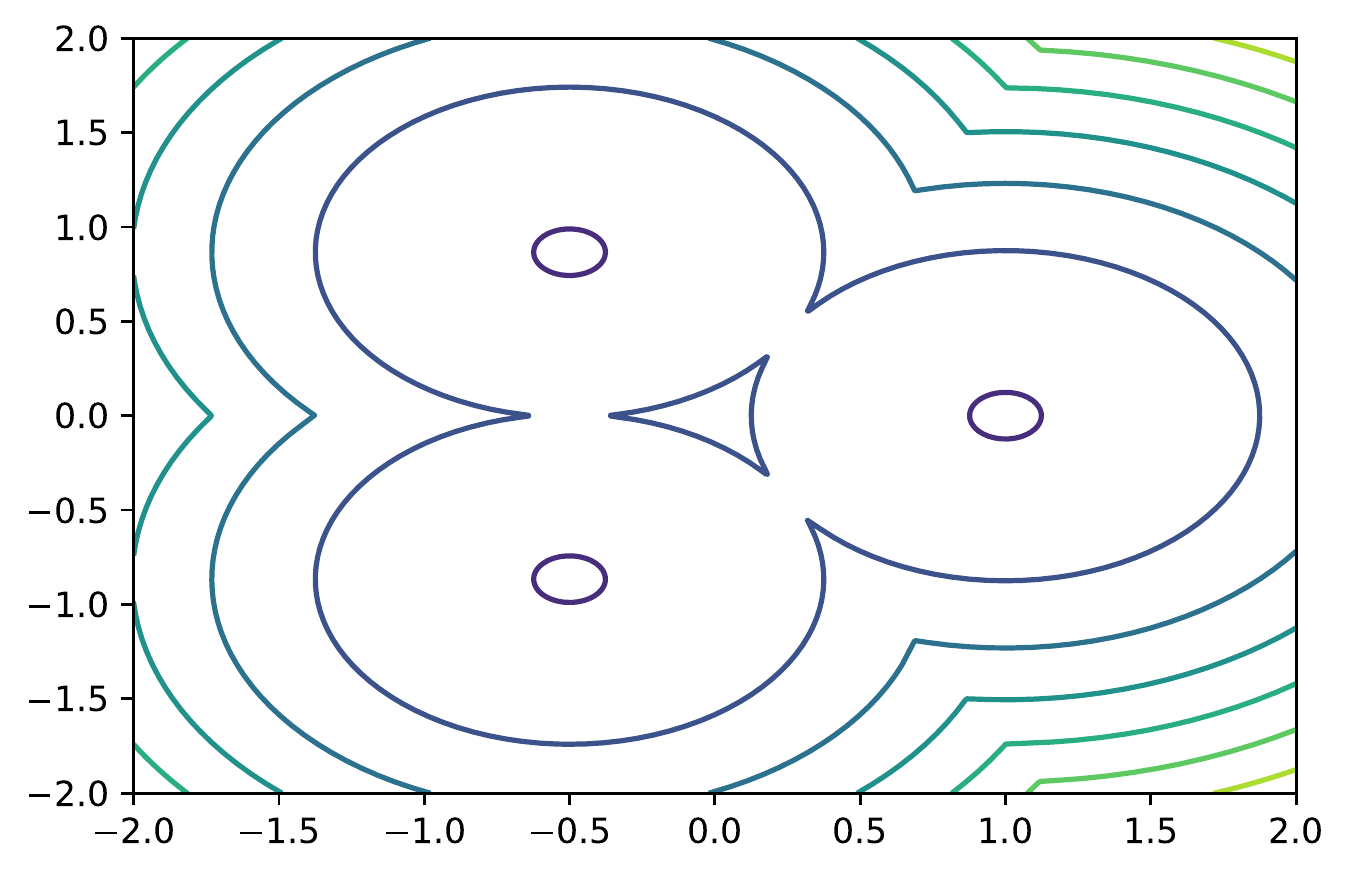}  
        &   \includegraphics[width=0.45\textwidth]{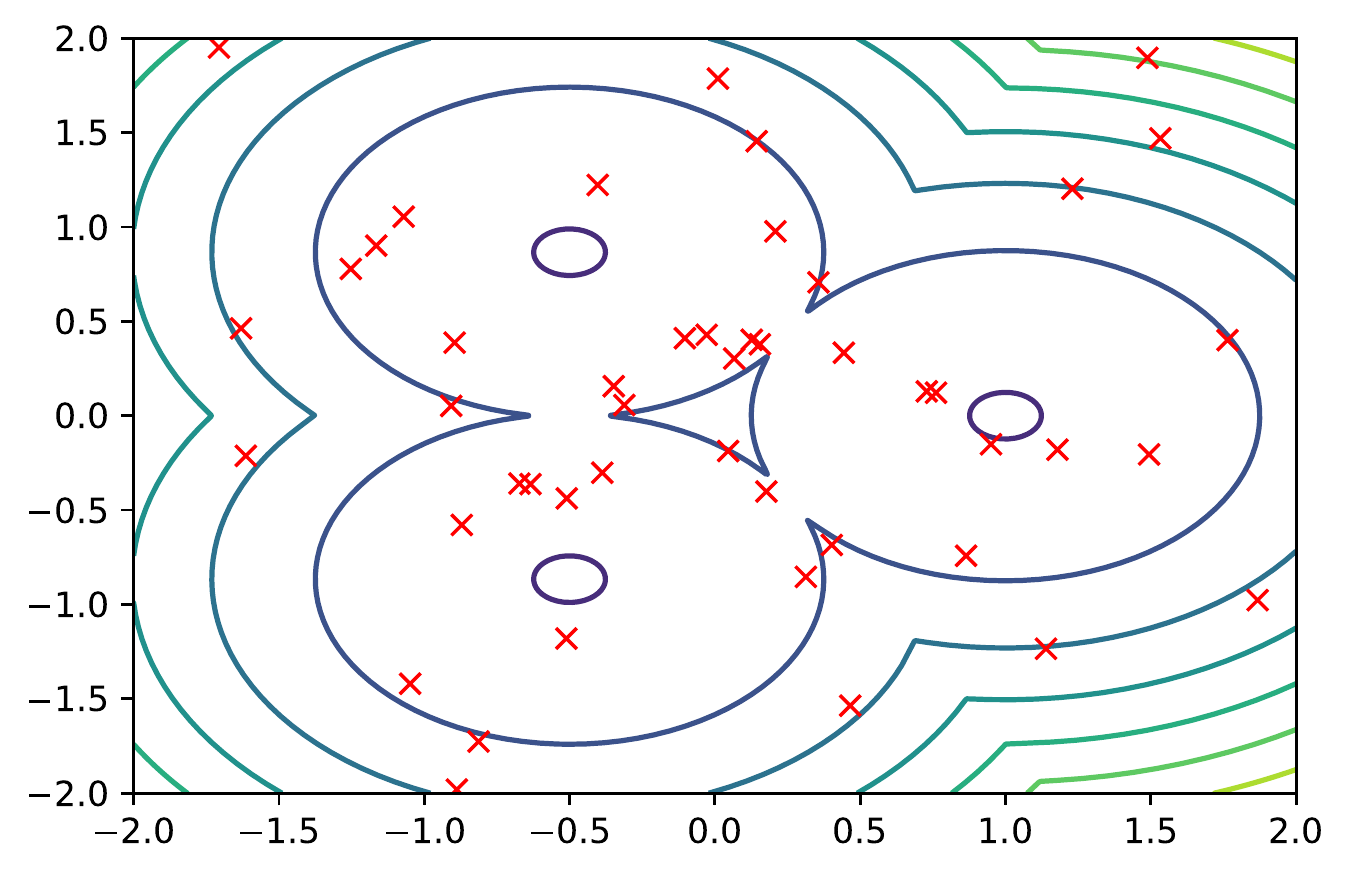}  \\
        {\small Landscape}&
        {\small Initialization}\\
           \includegraphics[width=0.45\textwidth]{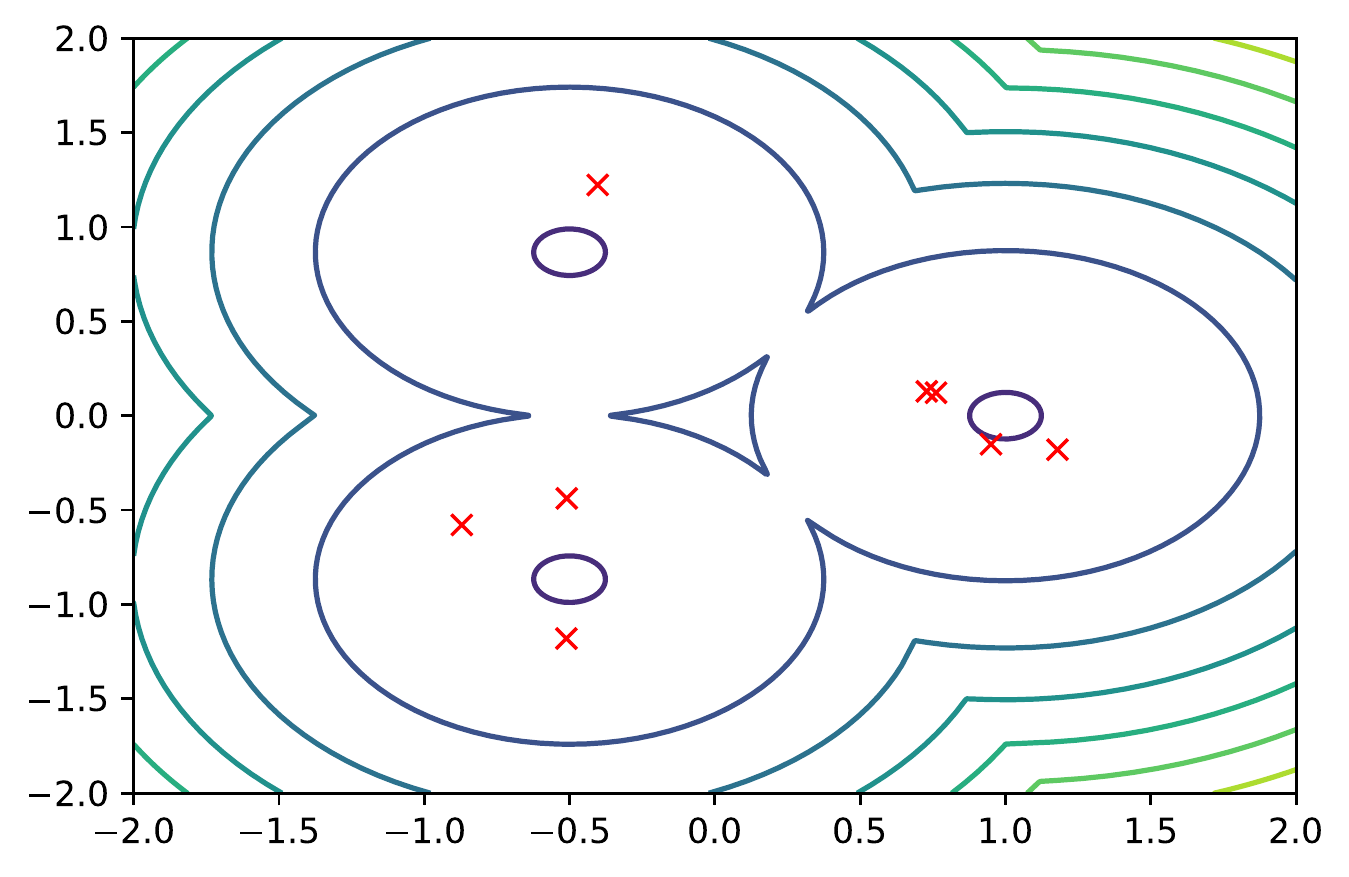}  
        &   \includegraphics[width=0.45\textwidth]{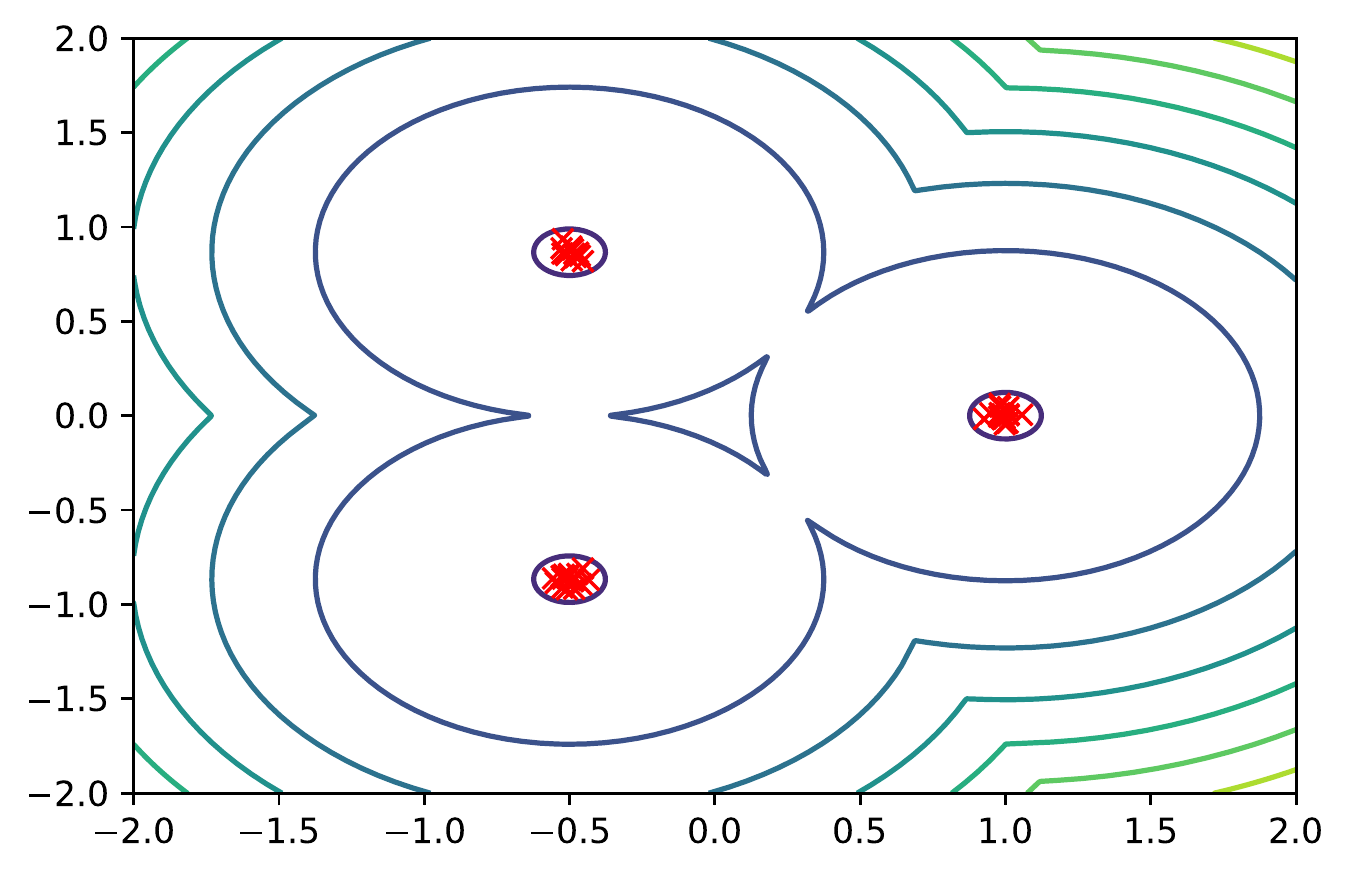}  \\
        {\small Weight-driven solution}&
        {\small Particle-based solution}
    \end{tabular}

    \caption{ Comparison between weight-driven and particle-based algorithms.
    }\label{fig:w_vs_p}
\end{figure*}

\paragraph{Generative Adversarial Networks.} We consider the training process of Generative Adversarial Networks to illustrate the comparison between weight-driven and particle-based algorithms. In particular, when the dimension is large, the weight-driven algorithm would suffer from the curse of dimensionality. For simplicity, we choose the Gaussian target distribution and compute the KL divergence between generated and target distribution. We choose a randomly sampled mean as our target distribution and use a linear model to reproduce it.  We choose the step-size $10^{-3}$ to perform the in our experiments. Figure \ref{fig:gan} indicate that the weight-driven algorithms (WFR-DA) suffer from the curse of dimensionality seriously, while the particle-based algorithm can perform well on high dimension space. The empirical results have justify the importance of the particle-based algorithm and efficacy of PAPAL in high dimension spaces.

\begin{figure*}[t]
    \centering
        \begin{tabular}{cccc}
  \includegraphics[width=0.23\textwidth]{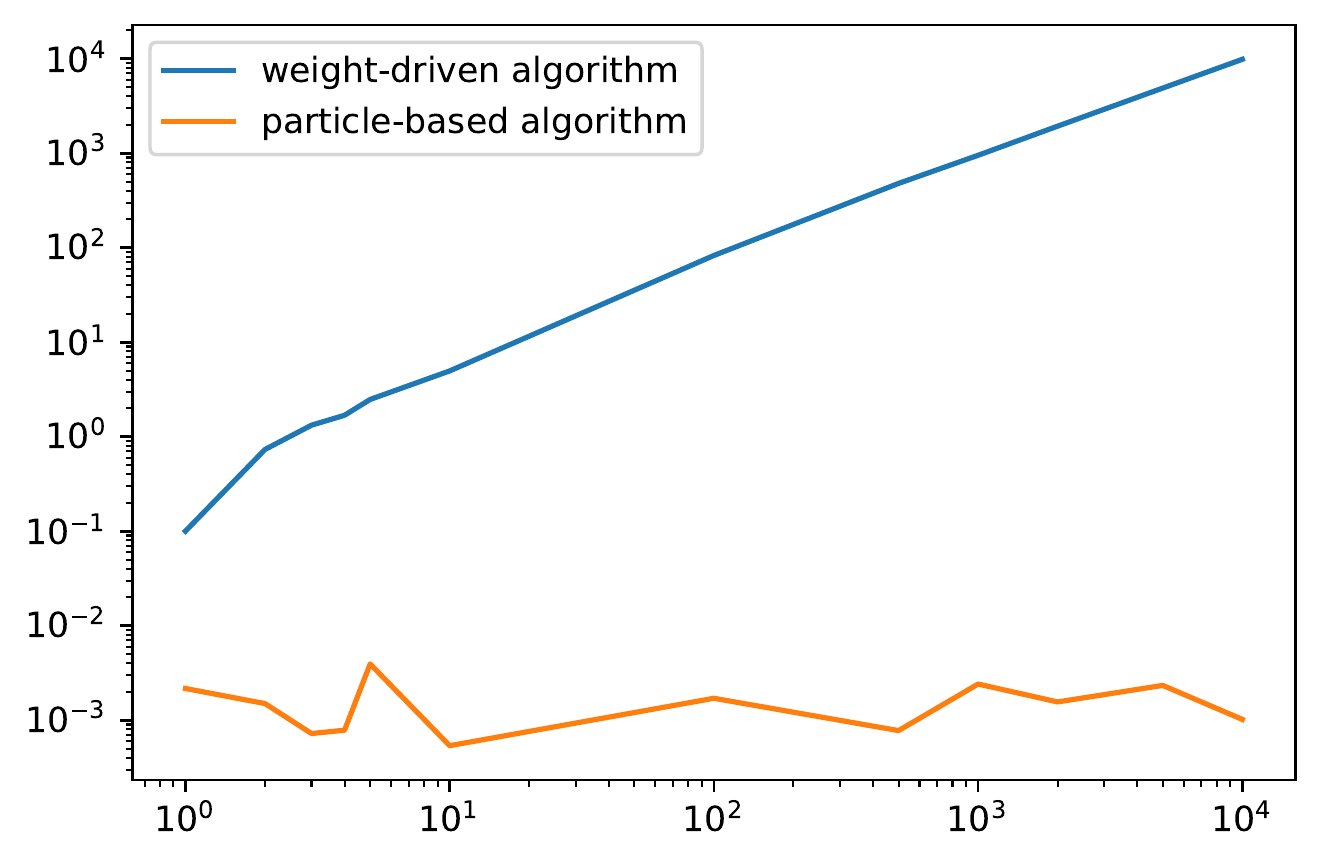}  
        &   \includegraphics[width=0.23\textwidth]{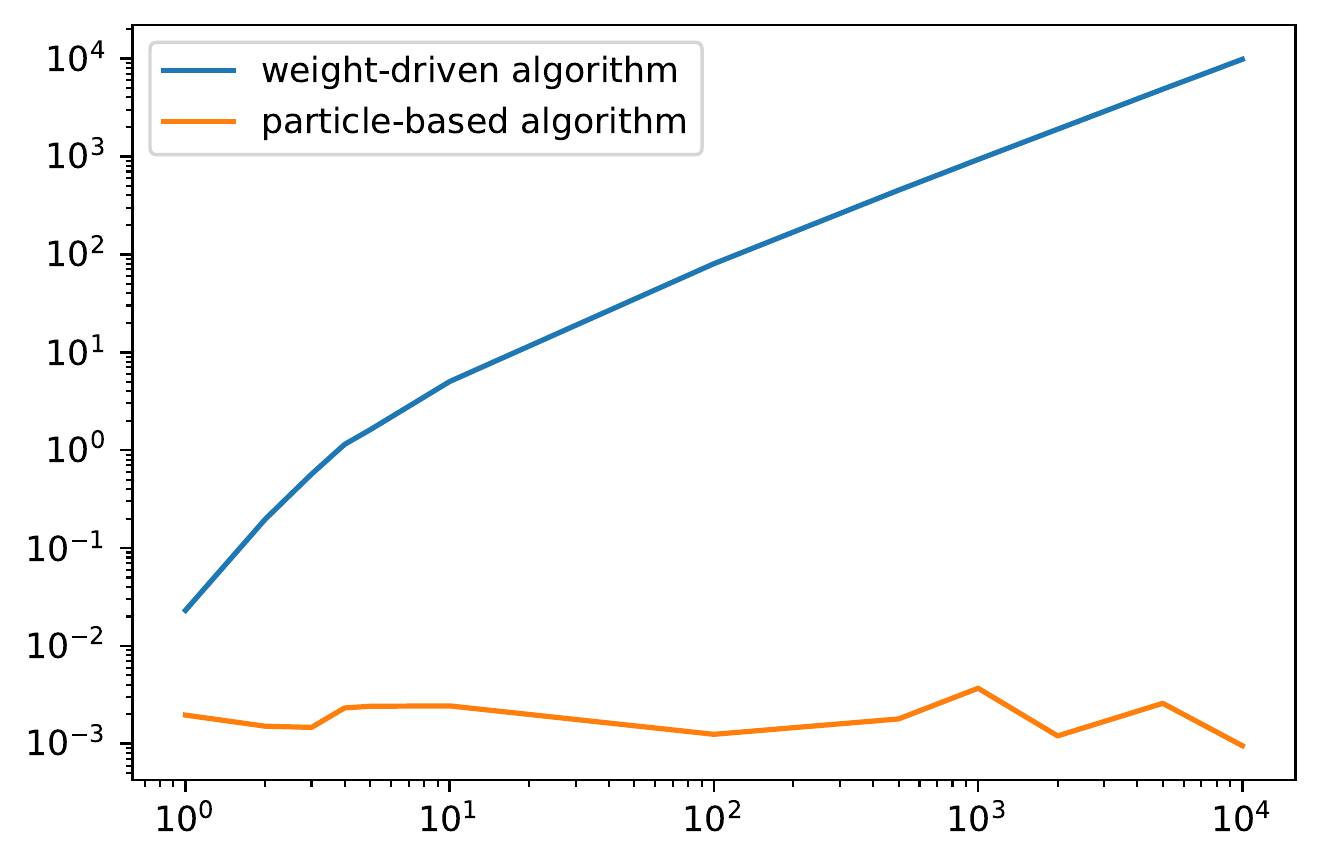}  
        &   \includegraphics[width=0.23\textwidth]{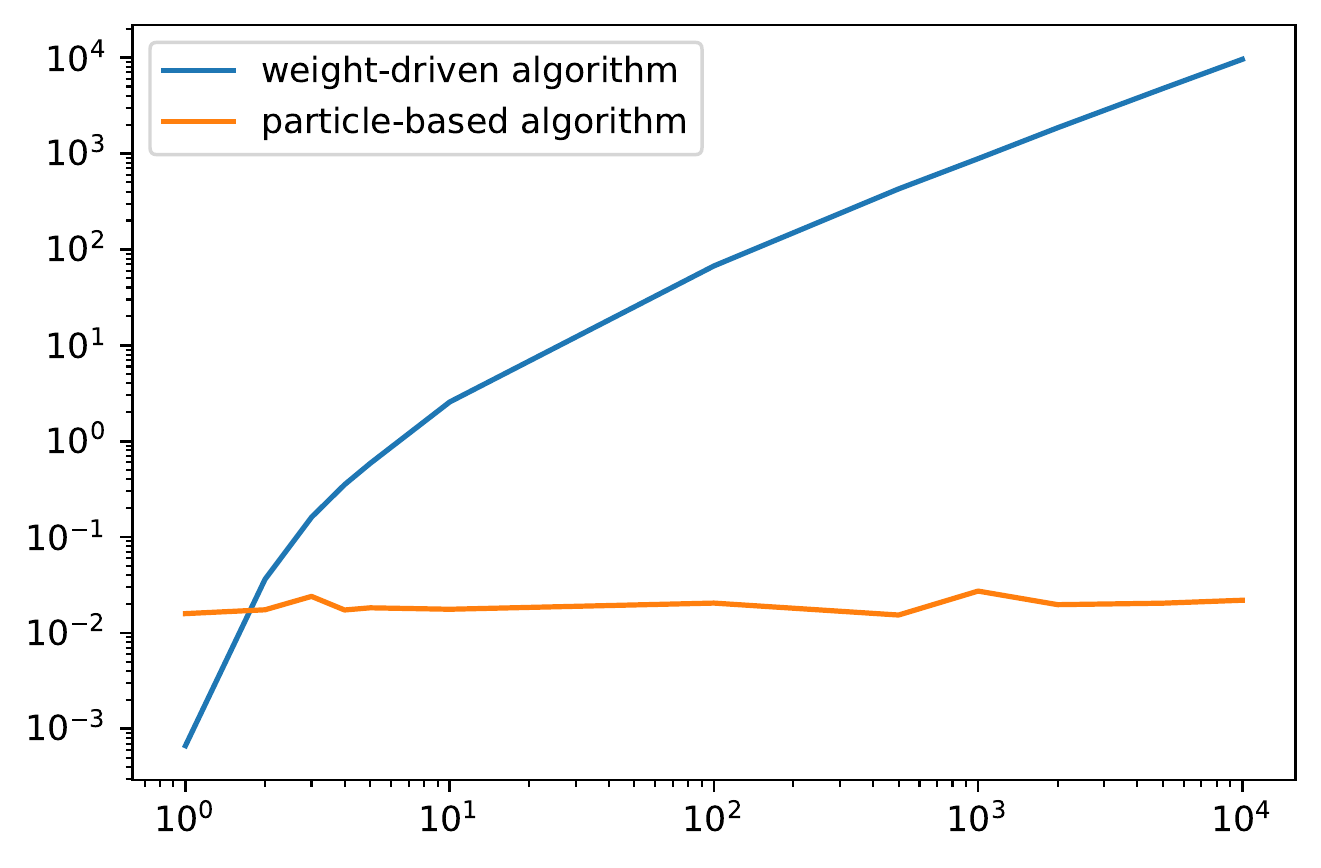}  
        &   \includegraphics[width=0.23\textwidth]{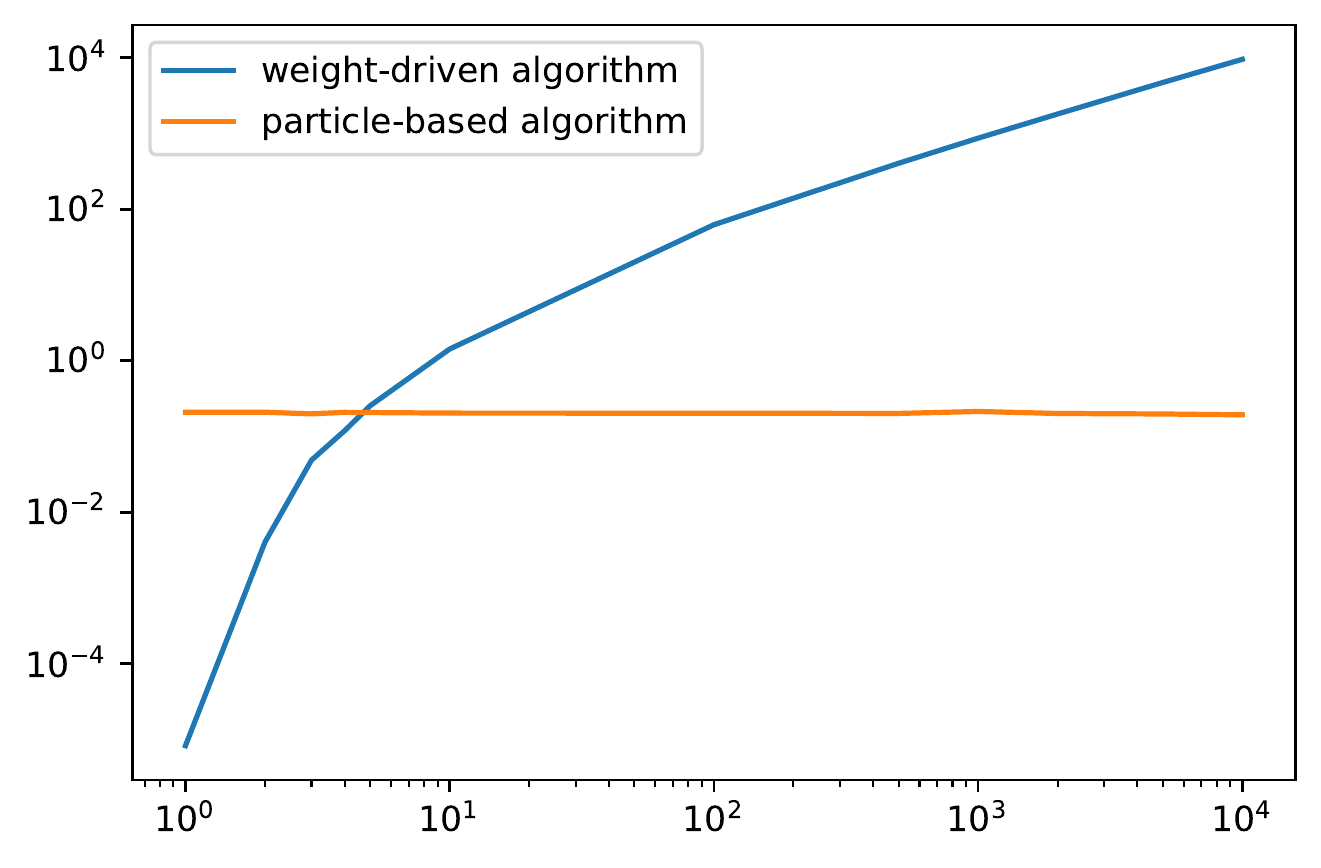}  \\
        {\small $n=5$}&
        {\small $n=10$}&
        {\small $n=100$}&
        {\small $n=1000$}
    \end{tabular}
    \caption{ Illustration of the sample efficiency on Generative Adversarial Networks. ($x$-axis is the dimension; $y$-axis is the KL divergence to measure the distance between real and fake distribution; $n$ denotes the number of particles)
    }\label{fig:gan}
\end{figure*}

\end{document}